\documentclass{tran-l}

\usepackage{amsmath, amssymb, amsthm}
\usepackage[utf8x]{inputenc}
\usepackage{comment}
\usepackage{hyperref}
\usepackage{enumitem}

\usepackage{xcolor}

\newtheorem{theo}{\textbf{Theorem}}[section]

\newtheorem{thm}[theo]{\textbf{Theorem}}
\newtheorem{lem}[theo]{\textbf{Lemma}}
\newtheorem{prop}[theo]{\textbf{Proposition}}

\newtheorem{cor}[theo]{\textbf{Corollary}}

\theoremstyle{definition}

\newtheorem{defn}[theo]{\textbf{Definition}}

\theoremstyle{remark}
\newtheorem{rem}[theo]{\textbf{Remark}}
\newtheorem{assumption}{\textbf{Assumption}}

\numberwithin{equation}{section}

\newcommand\vectorize[1]{\left(\begin{matrix}#1
\end{matrix}\right)}

\newcommand{\bsl}{\backslash}

\newenvironment{system}[1]
{
\left\{\begin{array}{#1}} 
{ \end{array}\right.}

\begin{document}

\title[Singular measure traveling waves]{Singular measure traveling waves in an epidemiological model with continuous phenotypes}
\author[Q. Griette]{Quentin Griette}

\address{IMAG, Université de Montpellier, 163 rue Auguste Broussonnet, 34090 Montpellier, FRANCE}
\curraddr{IMB, Université de Bordeaux, 351 cours de la Libération, 33800 TALENCE, FRANCE}
\email{quentin.griette@math.u-bordeaux.fr}
\urladdr{https://www.quentin-griette.fr}

\date{}
\maketitle

\begin{abstract}
    We consider the  reaction-diffusion equation 
    \begin{equation*}
        u_t=u_{xx}+\mu\left(\int_\Omega M(y,z)u(t,x,z)dz-u\right) + u\left(a(y)-\int_\Omega K(y,z) u(t,x,z)dz\right) ,
    \end{equation*}
    where $ u=u(t,x,y) $ stands for the density of a theoretical population with a spatial ($x\in\mathbb R$) and phenotypic ($y\in\Omega\subset \mathbb R^n$) structure, $ M(y,z) $ is a mutation kernel acting on the phenotypic space, $ a(y) $ is a fitness function and $ K(y,z) $ is a competition kernel. Using a vanishing viscosity method, we construct measure-valued traveling waves for this equation, and present particular cases where  singular traveling waves do exist. We determine that the speed of the constructed traveling waves is the expected spreading speed $ c^*:=2\sqrt{-\lambda_1} $, where $ \lambda_1 $ is the principal eigenvalue of the linearized equation. As far as we know, this is the first construction of a measure-valued traveling wave for a reaction-diffusion equation.
\end{abstract}


\section{Introduction}
In this work we consider the reaction-diffusion equation:
\begin{equation}\label{eq:evol}
    u_t=u_{xx}+ \mu(M\star u-u)+u(a(y)-K\star u),
\end{equation}
where $t>0$, $x\in\mathbb R$, $y\in \Omega$ for a bounded domain $\Omega \subset \mathbb R^n$, $u=u(t,x,y) $, $\mu>0$ is a positive constant, $ a=a(y) $ is a continuous function, $M=M(y,z) $ and $ K=K(y,z) $ are integration kernels, and the $ \star $ operation is defined by \eqref{eq:def_star}. After discussing the existence of stationary states for \eqref{eq:evol}, we construct measure-valued traveling waves and show the existence of a singularity for a subclass of parameters.

Equation \eqref{eq:evol} describes an asexual population living on a linear space, represented by the variable $x$. Several genotypes exist in the population, yielding a continuum of phenotypes, represented by the $ y $ variable. We denote $ \Omega\subset \mathbb R^n $ the set of all  reachable phenotypes. Our basic assumption is that the fitness (or intrinsic growth rate) of each individual is a function $a(y)$ of its phenotype. We also assume the existence of an underlying mutation process, by which an individual of phenotype $ z\in\Omega $ may give birth to an individual of phenotype $ y\in\Omega $, with probability $ M(y,z) $. Such mutations are expected to occur at rate $ \mu>0 $. Finally, the individuals are in competition for e.g. a finite resource, and we denote $ K(y,z) $ the cost on the fitness of $ y $ caused by the presence of $ z $.

In the context of epidemiology, $ u(t,x,y) $ can be thought as a density of hosts at point $x$, infected with a pathogen of trait $ y$. Equation \eqref{eq:evol} is particularly relevant in this context, since evidences suggest that pathogens (like e.g. viruses \cite{Hol-09})  can be subject to rapid evolution, which may then occur at the same time scale as the propagation of the epidemic \cite{Per-Phi-Bas-Has-13,Phi-Pus-13}. Moreover, equation \eqref{eq:evol} can easily be derived from a host-pathogen microscopic model \cite{Gri-Rao-Gan-15} in which we neglect the influence of the pathogen on the hosts' motility.

The study of asymptotic propagation in biological models can be traced back to the seminal works of Fisher \cite{Fis-1937} and Kolmogorov, Petrovsky, and Piskunov \cite{Kol-Pet-Pis-1937}, who investigated simultaneously the equation:
\begin{equation}\label{eq:FKPP}
    u_t=u_{xx}+u(1-u),
\end{equation}
where $ u=u(t,x) $ stands for the density of a spatially structured theoretical population.
They have shown, in particular, that for any compactly supported initial condition, the solution $ u(t,x) $ invades the whole space with constant speed $ c=2 $ (such a result is often called {\em spreading}); and that there exists a particular solution to \eqref{eq:FKPP}, which consists of a fixed profile shifting along the axis at speed $ c $, $ u(t,x)=\tilde u(x-ct) $, and connecting the unstable state $ 0 $ near $ +\infty $ to the stable state $ 1 $ near $ -\infty $ (such a particular solution is called {\em traveling wave}). Since then, these results have been generalized to a variety of related models: see e.g. \cite{Xin-00,Ber-Ham-02,Wei-02}, and the references therein. 

In the last decades, there has been an increasing interest in propagation models that take into account a multiplicity of different species. The {main problems in the field} include the replacement of a species by competitive interaction (see e.g. \cite{Gar-82}), predation \cite{Guo-11}, adaptation to climate change \cite{Alf-Cov-Rao-13}, or cooperation \cite{Lui-89,Wei-Lew-Li-02}. This last class of {\it cooperative} reaction-diffusion system has lead to  particularly strong results, since its properties are somewhat comparable to those of scalar equations.

In a recent work \cite{Gri-Rao-16}, the authors investigated the existence of traveling waves in the spatially homogeneous epidemiological model:
\begin{equation}\label{eq:Gri-Rao}
    \begin{system}{rcl}
        w_t&=&w_{xx}+w(1-(w+m)) + \mu(m-w) \\
        m_t&=&m_{xx}+rm\left(1-\frac{w+m}{K}\right)+\mu(w-m),
    \end{system}
\end{equation}
where $ w $ and $ m $ stand for a density of hosts infected by a wild type  and mutant pathogen, respectively. Though this system is not globally cooperative,  the authors managed to prove the existence and to compute the minimal speed of traveling waves as a function of the principal eigenvalue $\lambda $ of the associated {\em principal eigenvalue problem}:
\begin{equation*}
\left(\begin{matrix} 1-\mu & \mu \\ \mu & r-\mu\end{matrix}\right)\vectorize{w \\ m} + \lambda\vectorize{w \\ m}=0, 
\end{equation*}
via the formula $ c=2\sqrt{-\lambda} $. Intuitively, the spatial dynamics is then guided by the linearized system far away from the front (such a traveling wave is sometimes called a {\it pulled front} \cite{Gar-Gil-Ham-Roq-12,Sto-1976}). Since then, these results have been extended to a more general class of systems in \cite{Gir-18}.

Equation \eqref{eq:evol} can be seen as the continuous limit of system \eqref{eq:Gri-Rao} with a large number of equations. Since we aim at computing the propagation speed for this equation, we turn  to the associated principal eigenvalue problem:
\begin{equation}\label{eq:intro-princ}
    \mu(M\star u-u)+u(a(y)+\lambda) =0.
\end{equation}
This problem has been investigated in \cite{Cov-10} and \cite{Cov-13}, where the author shows an unexpected {\em concentration phenomenon} occurring for very natural fitness functions: if
\begin{equation}\label{eq:intro-cond-conc}
    \frac{1}{\sup_{z\in\Omega} a(z) - a(y)}\in L^1(\Omega),
\end{equation}
and $ \mu $ is small enough, there exists no continuous eigenfunction associated to \eqref{eq:intro-princ}, but rather {\em singular measure eigenvectors} with a singularity concentrated on the maximum of fitness $ \Omega_0:=\{y\in\overline\Omega\,|\, a(y)=\sup_{z\in\Omega}a(z) \} $. According to \eqref{eq:intro-cond-conc}, this phenomenon happens when $ a(y) $ is sufficiently steep near its global maximum, and is highly dependant on the Euclidean dimension of $ \Omega $. For instance,  if $n=1$, a concentration may appear at the optimum $ y=0 $ for the particular fitness function $ a(y)=1-\sqrt{|y|} $, when $ a(y)=1-|y| $ always yields continuous eigenfunctions; if $n=2$, $ a(y)=1-|y| $ may induce concentration, but $ a(y)=1-|y|^2 $ cannot. In dimension $n=3$ or higher, smooth fitness functions such as $ a(y)=1-|y|^2 $ may induce concentration. A similar phenomenon, and in particular the critical mutation rate under which concentration appears for a sufficiently steep fitness function, has been discussed  by Waxman and Peck \cite{Wax-Pec-98,Wax-Pec-06}.

{
The nonlocal competition term $-K\star u(y)$ in \eqref{eq:evol} is quite standard in models involving competition between different phenotypes. Many models focus on the case where the competition is simply the integral of the distribution --- this corresponds to $K(y,z)=1$. As an example, the nonlocal Fisher-KPP equation 
\begin{equation}\label{eq:FKPP-nonlocal}
	u_t-\Delta u=\mu u(1-\Phi*u),
\end{equation}
where $\Phi(y)$ is usually in $L^1(\mathbb R^n) $ with possibly additional restrictions, has attracted a lot of attention in the past \cite{Gou-00,Gen-Vol-Aug-06,Ber-Nad-Per-Ryz-09,Ham-Ryz-14, Fay-Hol-15, Has-Kop-Nab-Tro-16}. Nonlocal competition also appears in numerous other studies in population genetics and population dynamics \cite{Alf-Ber-Rao-17, Gil-Ham-Mar-Roq-17, Alf-Car-17,Bou-Cal-Meu-Mir-Per-Rao-12, Ber-Jin-Syl-16}. {In general, the qualitative behavior of traveling waves, and the long-time behavior of the solutions to the parabolic equation, are still difficult to handle. Recent advances have been made towards a better understanding of the asymptotic location of the front for the solutions to the parabolic equations, see \cite{Ham-Ryz-14, Bou-Hen-Ryz-17-2, Pen-17, Add-Ber-Pen-17} for the nonlocal Fisher-KPP equation; \cite{Bou-Hen-Ryz-17} for the cane toads equation.} In the case of the nonlocal Fisher-KPP equation \eqref{eq:FKPP-nonlocal}, the existence of traveling waves has been established and the associated minimal speed characterized in \cite{Ber-Nad-Per-Ryz-09, Ham-Ryz-14}. The convergence towards a stationary state on the back of the wave, has been  shown in \cite{Ber-Nad-Per-Ryz-09} for small $\mu$ or when the Fourier transform of the competition kernel is positive (in which cases one can prove the stability of the constant steady state $u\equiv 1$); {and more recently, in a perturbative case \cite{Add-Ber-Pen-17}, the convergence in long time has been shown for solutions to the parabolic equation.} In the general case, the convergence towards a stationary solution on the back of the wave is far from being clear. The situation is similar in the case of many other models involving nonlocal competition. 
}

As an indication that spreading happens, in the present paper we construct traveling waves for equation \eqref{eq:evol} which travel at the expected spreading speed. One of the main difficulties we encountered studying equation \eqref{eq:evol} is the lack of a regularizing effect in the mutation operator $M\star u $. {This phenomenon is confirmed by the existence of traveling waves having a nontrivial singular part --- in particular, there is no hope for asymptotic regularity. This lack of regularity also makes it more difficult to apply some of the techniques commonly used in the study of reaction-diffusion equations (in particular, taking the limit of a subsequence of large shifts of a solution). Finally, the non-compactness of the time-1 map prevents an application of the spreading results of Weinberger \cite{Wei-02}. One other challenging issue is the absence of a comparison principle for equation \eqref{eq:evol}, because of the nonlocal competition term. As in many other studies involving a nonlocal competition, this prevents a precise study of the long-time behavior of the solutions to the Cauchy problem and the behavior at the back of the waves (see also the above paragraph). To show that the traveling waves stay away from $0$ on the back, we introduce a secondary problem, constructed by increasing self-competition in equation \eqref{eq:evol}, which satisfies a comparison principle and serves as a sub-solution factory.} To overcome {the lack of regularity, we approximate the solutions of \eqref{eq:evol} by a vanishing viscosity method. We {choose the zero Neumann} boundary conditions for the approximating problem because they behave well with respect to the integration across the domain. Finally, we introduce} a weak notion of traveling waves which admit singularities. As we will see below, there is little hope to obtain more regularity in general, since there exist traveling waves for equation \eqref{eq:evol} which present an actual singularity. As far as we know, the present work constitutes the first construction of a measure-valued traveling wave in a reaction-diffusion equation. 

\section{Main results and comments}

\subsection{Function spaces and basic notions}

Throughout this document we use a number of function spaces that we {make precise} here to avoid any confusion. Whenever $ X $ is a subset of a Euclidean space, we will denote $ C(X)$, $ C_b(X) $, $C_0(X) $, $ C_c(X) $ the space of continuous functions, bounded continuous functions, continuous functions vanishing at $ \infty$ and continuous functions with compact support over $ X $, respectively. Notice that if $ X $ is compact, then those four function spaces coincide. Whenever $ X \subset\mathbb R^d $ is a Borel set, we define $ M^1(X) $ as the set of all Borel-regular measures over $ X $. Let us recall that $ M^1(X) $ is the topological dual of $ C_0(X) $, {by} Riesz's representation theorem \cite{Rud-74}. In our context, $ M^1(X) $ coincides with the set of Borel measures that are inner and outer regular \cite{Rud-74, Bog-07}. We will thus call \textit{Radon measure} an element of $ M^1(X) $.

When $ p\in M^1(X) $, we say that the equality $ p=0 $ holds \textit{in the sense of measures} if
\begin{equation*}
    \forall \psi\in C_c(X), \int_X\psi(x)p(dx)=0.
\end{equation*}

We now define the notion of \textit{transition kernel} (see \cite[Definition 10.7.1]{Bog-07}), which is crucial for our notion of traveling wave:
\begin{defn}[Transition kernel]\label{def:TK}
    We say that $ u\in M^1(\mathbb R\times X) $ has a \textit{transition kernel} if there exists a function $ k(x,dy) $ such that 
    \begin{enumerate}
        \item for any Borel set $ A\subset X $, $ k(\cdot, A) $ is a measurable function, and
        \item for almost every $x\in\mathbb R $ (with respect to the Lebesgue measure on $ \mathbb R $), $ k(x, \cdot)\in M^1(X) $
    \end{enumerate}
    and $ u(dx, dy)=k(x,dy)dx $ in the sense of measures, i.e. for any $ \varphi\in C_c(\mathbb R\times X)$, the following equality holds
    \begin{equation*}
        \int_{\mathbb R\times X} \varphi(x,y)u(dx,dy)=\int_\mathbb R\int_X\varphi(x,y)k(x,dy)dx.
    \end{equation*}
    For simplicity, if the measure $ u $ has a transition kernel, we will often say that $ u $ \textit{is} a transition kernel and use directly the notation $ u(dx, dy)=u(x,dy)dx $.
\end{defn}

We denote $ f\star g $ the function:
\begin{equation}\label{eq:def_star}
    f\star g (y):=\int_{\overline\Omega}f(y,z)g(dz)
\end{equation}
whenever $ f: \overline\Omega^2\to \mathbb R $ and $ g $ is a measure on $ \overline\Omega$. If $ g $ is continuous or $ L^1(\Omega) $ we use the convention $ g(dz):= g(z)dz $ in the above formula. Remark that the operation $ \star $ is not the standard convolution, though both notions share many properties.

{Finally, for $y\in\partial\Omega$ we will call  $\nu(y)$ or simply $\nu$ the outward normal unit vector of $\Omega$.}

\subsection{Main results}

Our main result is the existence of a measure traveling wave, possibly singular, for equation \eqref{eq:evol}. Before {stating the result}, let us give our assumptions, as well as subsidiary results. 
\begin{assumption}[Minimal assumptions]\label{hyp:gen}
    \begin{enumerate}
        \item $ \Omega\subset\mathbb R^n $ is a bounded connected open set with $ C^3 $ boundary. For simplicity we assume $ 0\in\Omega$.        
        \item $ M=M(y,z) $ is a $ C^{\alpha} $ positive function $ \overline\Omega\times\overline\Omega\to\mathbb R $ satisfying
            \begin{equation*}
                \forall z\in\overline\Omega, \int_{\overline\Omega} M(y,z) d y = 1 .
            \end{equation*}
            In particular, 
            $  0<m_0\leq M(y,z)\leq m_\infty<+\infty$
            for any $ (y,z)\in\overline\Omega\times\overline\Omega$. 
        \item $ K=K(y,z) $ is a $C^{\alpha} $ positive function $ \overline\Omega\times\overline\Omega\to\mathbb R $. In particular, we have  $ 0<k_0\leq K(y,z)\leq k_\infty < +\infty$ for any $(y, z)\in\overline\Omega\times\overline\Omega$.
	\item $ a =a(y)\in C^{\alpha }(\overline\Omega) $ is a non-constant function with $\sup_{y\in\overline\Omega} a(y)>0$. We assume that $ a(0)=\sup a$.  In particular, $ -\infty < \inf a < \sup a < +\infty $ holds.
        \item  We let 
            $\Omega_0:=\left\{y\in\overline\Omega\,|\, a(y)=a(0)=\sup_{z\in\overline\Omega} a(z)\right\}$ be the set of maximal value for $a$
            and assume $ \Omega_0\subset\subset \Omega $.
        \item $ 0<\mu<\sup a-\underset{z\in\partial\Omega}\sup a^+(z) $.
    \end{enumerate}
\end{assumption}

We are particularly interested in a more restrictive set of assumptions, under which we hope to see a concentration phenomenon in \eqref{eq:evol}:
\begin{assumption}[Concentration hypothesis]\label{hyp:dega}
    In addition to Assumption \ref{hyp:gen}, we suppose
    \begin{equation*}
        y\mapsto \frac{1}{\sup_{z\in\Omega} a(z)-a(y)}\in L^1(\Omega) .
    \end{equation*}
\end{assumption}

Let us introduce the principal eigenvalue problem that guides our analysis:
\begin{defn}[Principal eigenvalue]\label{def:vpp}
    We call \textit{principal eigenvalue} associated with \eqref{eq:evol} the real number:
    \begin{equation}\label{eq:vpp}
        \lambda_1:=\sup\{\lambda\,|\, \exists \varphi\in C(\overline\Omega), \varphi>0 \text{ s.t. }\mu(M\star \varphi - \varphi) +(a(y)+\lambda)\varphi \leq 0\}.
    \end{equation}
\end{defn}
Clearly, $ \lambda_1 $ is well-defined and we have  $ \lambda_1\leq -(\sup a -\mu) $ by evaluating  \eqref{eq:vpp} at $ y=0 $.
Though we call $ \lambda_1 $ the principal eigenvalue, we stress that $ \lambda_1 $ is not always associated with a usual eigenfunction. In particular, Coville, in his work \cite{Cov-13, Cov-10}, gives conditions on the coefficients of \eqref{eq:evol} under which there exists no associated eigenfunction. We will recall and {extend} these results in section \ref{ssec:eigen}.  
\begin{prop}[On the principal eigenvalue]\label{prop:vpp}
    Under Assumption \ref{hyp:gen}, there exists a unique $ \lambda\in\mathbb R $ such that the equation 
    \begin{equation}\label{eq:eigen}
        \mu (M\star\varphi-\varphi) + (a(y)+\lambda)\varphi = 0
    \end{equation}
    has a nonnegative nontrivial solution in the sense of measures, and $ \lambda=\lambda_1 $. 

    Moreover, under Assumption \ref{hyp:dega}, there exists $ \mu_0 > 0$ such that if $ \mu<\mu_0 $, we have 
    \begin{equation*}
        \lambda_1 = -(\sup a - \mu)
    \end{equation*}
	and, in this case, there exists a nonnegative measure $ \varphi $ solution to \eqref{eq:eigen} with a non-trivial singular part concentrated in $ \Omega_0 $.
\end{prop}

{The most part of Proposition \ref{prop:vpp} comes  from the work of Coville \cite{Cov-10, Cov-13}. Our contribution to the result is the uniqueness of the real number $\lambda$ such that there exists a nonnegative nontrivial measure solution to \eqref{eq:eigen}. We use this uniqueness result several times in the paper, in particular,  in many of the arguments involving a vanishing viscosity; for instance in the proofs of Theorem \ref{thm:eigen} and Theorem \ref{thm:survival}.}

As well-known in KPP situations, we expect the sign of $ \lambda_1 $ to dictate the long-time persistence of solutions to equation \eqref{eq:evol}. In particular, when $ \lambda_1>0 $, we expect that any nonnegative solution to the Cauchy problem \eqref{eq:evol} starting from a positive bounded initial condition goes to 0 as $ t\to\infty $. Indeed, in this case there exists a positive continuous function $ \psi>0 $ such that 
\begin{equation*}
    \mu(M\star \psi - \psi) +\left(a+\frac{\lambda_1}{2}\right)\psi \leq 0.
\end{equation*}
One can check that $ Ce^{-\frac{\lambda_1}{4}t}\psi(y) $ and $ u(t,x, y) $ are respectively a super- and subsolution of the equation 
\begin{equation*}
    u_t=u_{xx}+\mu(M\star u-u)+a(y) u.
\end{equation*}
with ordered initial data (for $ C $ large enough).
The result is then a consequence of the comparison principle satisfied by the (linear) above equation.

In the $\lambda_1=0$ case, we expect extinction as in the $\lambda_1>0$ case. This is generally the case for scalar reaction-diffusion equations, as well as in the case of some systems (see in particular \cite[Proposition 5.2]{Gir-18}).  
However, the usual strategy, which consists in establishing a contradiction by studying the least multiple of the principal eigenfunction which lies above the $\omega$-limit set of a solution to \eqref{eq:evol}, seems difficult to apply here. Indeed we lack three of the main ingredients for this argument: a Harnack inequality, compactness, and a $L^\infty$ bound on the orbit which would allow us to place a multiple of the principal eigenvector above the $\omega$-limit set. Thus, in the present paper, we leave this particular point open. {Note however that, in the case where $M$ is symmetric ($M(y,z)=M(z,y)$), an argument similar to the one employed in \cite[Section 5]{Bou-Cha-Hen-Kim-17} may lead to an actual proof, by working directly on the parabolic problem.}

In the present paper we focus on the $ \lambda_1<0 $ case, in which we expect survival of the population. To confirm this scenario, we first prove the existence of a nonnegative nontrivial stationary state for equation \eqref{eq:evol}.
\begin{thm}[Survival of the population]\label{thm:survival}
    Let Assumption \ref{hyp:gen} hold and assume further $ \lambda_1<0 $. Then, there exists a nonnegative nontrivial stationary state for equation \eqref{eq:evol}, i.e. a nonnegative nontrivial measure $ p\in M^1(\overline\Omega) $ which satisfies
    \begin{equation}\label{eq:stat}
        \mu(M\star p-p) + p(a(y)-K\star p)=0
    \end{equation}
    in the sense of measures. 
\end{thm}
Under the hypothesis for concentration (Assumption \ref{hyp:dega}) and in the special case where the competition kernel $ K(y,z) $ is independent of the trait $ y $,  Bonnefon, Coville and Legendre \cite{Bon-Cov-Leg-17} have shown that the solution to \eqref{eq:stat} has a singularity concentrated in $ \Omega_0 $ when $ \mu $ is small. A key argument was a separation of variables method, allowed by the assumption $ K(y,z)=K(z)$. Here we show that the concentration phenomenon occurs under a  more general hypothesis on $ K $, namely that the trait $ y\in\Omega_0 $ suffers less from the competition than any other trait. Since $ \Omega_0 $ also maximizes the basic reproductive ratio $ a(y) $, it seems natural to expect concentration in $ \Omega_0$ in this case. 
\begin{assumption}[Nonlinear concentration]\label{hyp:dom}
    In addition to Assumption \ref{hyp:dega}, we suppose that
    \begin{equation*}
        \forall (y,z)\in\overline\Omega\times\overline\Omega,\quad K(0, z)\leq K(y,z).
    \end{equation*}
\end{assumption}
\begin{thm}[Concentration on dominant trait]\label{thm:concentration}
    Let Assumption \ref{hyp:dom} hold, and assume $ \lambda_1<0 $. Then, there exists $ \mu_0>0 $ such that, for any $ \mu<\mu_0 $, the measure $ p $, constructed in Theorem \ref{thm:survival}, has a singular part concentrated in $ \Omega_0  $.
\end{thm}

To better characterize the spatial dynamics of solutions to \eqref{eq:evol}, we are going to construct traveling waves for \eqref{eq:evol}. 
\begin{defn}[Traveling wave]\label{def:TW}
    A \textit{traveling wave} for equation \eqref{eq:evol} is a couple $ (c,u) $ where $ c\in\mathbb R $ and  $ u $ is a locally finite transition kernel (see Definition \ref{def:TK}) defined on $ \mathbb R\times\overline\Omega $. We require that $ (c,u) $ satisfies:
    \begin{equation}\label{eq:TW}
        -cu_x-u_{xx}=\mu(M\star u-u) + u(a-K\star u)
    \end{equation}
    in the sense of distributions, and that the measure $ u $ satisfies the limit conditions:
    \begin{eqnarray}
        \label{eq:limit-TW-left}
        \underset{\bar x\to+\infty}{\liminf}\int_{\mathbb R\times\overline\Omega} \psi(x+\bar x,y)u(dx, dy)>0, \\
        \label{eq:limit-TW-right}
        \underset{\bar x\to-\infty}{\limsup}\int_{\mathbb R\times\overline\Omega} \psi(x+\bar x,y)u(dx, dy)=0
    \end{eqnarray}
    for any positive test function $ \psi\in C_c(\mathbb R\times\overline\Omega)$.
\end{defn}
{Condition \eqref{eq:limit-TW-left} differs from the usual behavior of traveling waves as defined, for instance, in \cite{Wei-82,Ber-Ham-02,Ogi-Mat-99-DCDS}, in which the convergence to a stationary state is required. Because of the nonlocal competition, indeed, it is very difficult to prove that a solution to equation \eqref{eq:evol} converges to a stationary state when $t\to\infty$. Imposing a weak condition like \eqref{eq:limit-TW-left} on the back of the wave is the usual way to go around this issue. One can refer for instance to \cite{Ber-Nad-Per-Ryz-09,Alf-Cov-Rao-13, Bou-Cal-14,Gir-18}, where a similar condition is imposed on the back of traveling waves.
}
\medskip

We are now in the position to state our main result, which concerns the existence of a traveling wave for \eqref{eq:evol}.
\begin{thm}[Existence of a traveling wave]\label{thm:TW}
    Under Assumption \ref{hyp:gen} and if $\lambda_1<0$, there exists a traveling wave $(c,u)$ for \eqref{eq:evol} with $c=c^*:=2\sqrt{-\lambda_1} $. 
\end{thm}

{As it is the case in many nonlocal problem, the uniqueness and stability of the traveling waves are unknown. In this paper, we focus on the construction of a traveling wave for $c=c^*$. Altough this is expected, we leave the construction of traveling waves for $c>c^*$ for future work, as well as a proof of the non-existence of traveling waves for $c<c^*$.}\
In the general case, it seems very involved to determine whether $ u $ has a singular part or not. Nevertheless, there are some particular cases where singular traveling waves do exist.

\begin{rem}[Traveling waves with a singular part]
    In the special case where $ K $ is independent from $y$ ($K(y,z)=K(z) $), a separation of variables argument --- see \cite{Ber-Jin-Syl-16} for a related argument--- allows us to construct traveling waves that actually have a singular part in $ \overline\Omega $. From Proposition \ref{prop:vpp}, under Assumption \ref{hyp:dega},  there is $ \mu_0>0 $ such that, for any $ \mu<\mu_0 $, there exists a  measure eigenvector $ \varphi\in M^1(\overline\Omega) $ with a singular part concentrated in $ \Omega_0 $. We choose such a $ \varphi $ with normalization $ \int_{\overline\Omega}K(z)\varphi(dz) = 1 $. If moreover $ \lambda_1<0 $, then there exists a positive front $ \rho $, connecting $ -\lambda_1 $ to 0,  for the Fisher-KPP equation 
    \begin{equation}\label{eq:eigen-KPP}
        -\rho_{xx}-c\rho_x=\rho(-\lambda_1-\rho)
    \end{equation}
    for any $ c\geq 2\sqrt{-\lambda_1} $. If we define $ u(x,dy):=\rho(x)\varphi(dy)$, we see that $ u $ matches the definition of a traveling wave. Hence for any $ x\in\mathbb R $, $ u(x, \cdot) $ possesses a singular part concentrated in $ \Omega_0 $.
\end{rem}

\bigskip

The organization of the paper is as follows. In Section \ref{sec:eigen} we study related eigenvalue problems for which concentration may occur. Section \ref{sec:stat} is devoted to the construction of stationary states through a bifurcation method. Last, we construct a (possibly singular) measure traveling wave in Section \ref{sec:TW}.

\section{On the principal eigenvalue problem}\label{sec:eigen}

In this section, we prove  Proposition \ref{prop:vpp}, which allows an approximation by an elliptic Neumann eigenvalue problem in Theorem \ref{thm:eigen} of crucial importance for the construction of steady states in Section \ref{sec:stat}.

\subsection{The principal eigenvalue of nonlocal operators}\label{ssec:eigen}

Under Assumption \ref{hyp:gen}, Coville \textit{et al.}\ \cite{Cov-10,Cov-13,Cov-Dav-Mar-13} have extensively studied the principal eigenvalue problem associated with \eqref{eq:evol}. We summarize and {extend} the results in \cite{Cov-13}. Our contribution is to show the uniqueness of the principal eigenvalue as a solution to \eqref{eq:eigen} in the sense of measures. 

\begin{thm}[On the principal eigenproblem \eqref{eq:eigen}]\label{thm:eigenpair}
    \begin{enumerate}
        \item \label{item:eigenpair-unique}
            Let Assumption \ref{hyp:gen} be satisfied. Then, there exists a unique $ \lambda\in\mathbb R $ such that \eqref{eq:eigen} admits a nonnegative nontrivial Radon measure solution, and $ \lambda=\lambda_1$. 

        \item \label{item:eigenpair-dega}
            Let Assumption \ref{hyp:dega} hold, and let $ -\gamma_1 $ be the principal eigenvalue\footnote{We use the "minus" sign  for consistency between  Definition \ref{def:vpp} and the algebraic notion generally used in the Krein-Rutman Theorem : $ \mathcal M[\Phi]=\gamma_1^1\Phi$.}of the operator  
            \begin{equation*}
                \mathcal M[\psi]:=\int_\Omega \mu M(y,z)\frac{\psi(z)}{\sup a-a(z)}dz,
            \end{equation*}
            acting on $ \psi\in C_b(\Omega) $. Then the following holds: 
            \begin{enumerate}[label=(\textit{\roman*}),ref=(\textit{\roman*})]
		    \item \label{item:eigenpair-C0} $\gamma_1>1 $ if, and only if, $ \lambda_1 < -(\sup a-\mu) $. In this case, any solution to \eqref{eq:eigen} in the sense of measures  is a {pointwise solution}. 
                \item \label{item:eigenpair-L1} $ \gamma_1=1 $ if, and only if,  $ \lambda_1=-(\sup a-\mu) $ and there exists a nonnegative nontrivial function $\varphi\in L^1(\Omega) $ solution to \eqref{eq:eigen} almost everywhere. In this case, $ \varphi $ is unique (up to multiplication by a positive constant).      
                \item \label{item:eigenpair-M1} $\gamma_1<1 $ if, and only if,  $ \lambda_1=-(\sup a-\mu) $ and there exists a nonnegative singular measure $ \varphi\in M^1(\overline\Omega) $ solution to  \eqref{eq:eigen}. In this case, any nonnegative nontrivial solution to \eqref{eq:eigen} has a singularity concentrated in $ \Omega_0 $.
            \end{enumerate}
    \end{enumerate}
\end{thm}
\begin{proof}
    The existence of a measure-valued solution to \eqref{eq:eigen} has been shown in \cite[Theorem 1.2]{Cov-13}. Here we focus on the uniqueness of $ \lambda $. We first prove the uniqueness of $ \lambda$  when the complement of Assumption \ref{hyp:dega} holds, by showing that any eigenvector is in fact a continuous eigenfunction.  Then, we show that uniqueness holds under Assumption \ref{hyp:dega}. Finally we prove the trichotomy in item \ref{item:eigenpair-dega}.

    \medskip

	\textbf{Step 1}: Let the complement of Assumption \ref{hyp:dega} hold, i.e. $ \frac{1}{\sup a-a(y)}\not\in L^1(\Omega)$. Let $ \varphi\in M^1(\overline\Omega) $ be a nonnegative nontrivial Radon measure solution to \eqref{eq:eigen}. Then {by} the Lebesgue-Radon-Nikodym Theorem \cite[Theorem 6.10]{Rud-74}, there exists a nonnegative $ \varphi_{ac}\in L^1(\Omega) $ and a nonnegative measure $ \varphi_s\in M^1(\overline\Omega) $, which is singular with respect to the Lebesgue measure on $ \Omega $, such that:
    \begin{equation*}
        \varphi= \varphi_{ac}dy + \varphi_s.
    \end{equation*}
    Equation \eqref{eq:eigen} is then equivalent to the following system:
    \begin{equation}\label{eqlem:twomeasures}
        \begin{system}{lr}
            \mu M\star \varphi + (a(y)-\mu+\lambda)\varphi_{ac}=0 &\quad a.e. (dy) \\
            a(y)-\mu+\lambda  = 0 &\quad a.e.(\varphi_s).
        \end{system}
    \end{equation}
    This readily shows that $ (a(y)-\mu+\lambda)\varphi_{ac}=-\mu M\star\varphi $ is a continuous negative function and in particular 
    $\lambda\leq-(\sup a-\mu)$.
    \medskip

    We distinguish two cases:

	{\it Case 1:} Assume first that $ \lambda < -(\sup a-\mu) $. Then the second line of \eqref{eqlem:twomeasures} implies $ \mathrm{supp\,}\varphi_s=\varnothing $, i.e. $ \varphi_s\equiv0 $. In this case we have $\varphi_{ac}(y)=\frac{\mu M\star \varphi_{ac}(y)}{-\lambda - (a(y)-\mu) }$, which is a positive continuous function {since the kernel $M(y,z)$ is itself continuous}. A classical comparison argument {(such as the one presented below on Step 2 case 1)} then shows $ \lambda=\lambda_1 $.

    {\it Case 2:} Assume $ \lambda = -(\sup a-\mu) $. Then 
    \begin{equation*}
        \varphi_{ac}(y)=\frac{\mu M\star \varphi}{\sup a-a(y)},
    \end{equation*}
    and since $ \mu (M\star \varphi)(y)\geq \mu m_0\int_\Omega\varphi(dz)>0 $, this implies $ \varphi_{ac}\not\in L^1(\Omega) $, which contradicts the definition of $ \varphi_{ac} $.

    We have thus shown the uniqueness of the real number $ \lambda $ such that there exists a solution $(\lambda, \varphi) $ to \eqref{eq:eigen}.

    \medskip

	\textbf{Step 2:} Let Assumption \ref{hyp:dega} hold. We first establish {that  $ \gamma_1 $ is well-defined}, then resume the proof. 

    The operator $ \mathcal M $ defined above is compact by virtue of the Arzel\`a-Ascoli Theorem \cite[Theorem 4.25]{Bre-11}. Since for any $\psi\geq 0 $, $ \psi\not\equiv 0 $, we have 
    \begin{align*}
        \forall y\in\overline\Omega, \quad\mathcal M[\psi](y)&=\int_\Omega\mu M(y,z)\frac{\psi(z)}{\sup a-a(z)}dz \\
                                                             &\geq \mu m_0\int_\Omega\frac{\psi(z)}{\sup a-a(z)}dz>0,
    \end{align*}
	$ \mathcal M $ satisfies the hypotheses of the Krein-Rutman Theorem \cite[Theorem 6.13]{Bre-11}, which ensures {that the real number $ \gamma_1 $, defined by  
    $\mathcal M[\Psi]=\gamma_1\Psi$ 
	for a positive $ \Psi\in C_b(\Omega) $, is  well-defined and positive.}

	Let us resume the proof. Let $ (\lambda, \varphi) $ be a solution to \eqref{eq:eigen} in the sense of measures. Then, as above, {by} Lebesgue-Radon-Nikodym Theorem \cite[Theorem 6.10]{Rud-74}, there exists a nonnegative $ \varphi_{ac}\in L^1(\Omega) $ and a nonnegative measure $ \varphi_s \perp dy$, such that
    $ \varphi= \varphi_{ac}dy + \varphi_s$.
	In this context, equation \eqref{eq:eigen} is equivalent to system \eqref{eqlem:twomeasures}, and in particular we have $ \lambda\leq-(\sup a-\mu)$. We subdivide the rest of the proof in two {cases}.

    \medskip

	1. Let us first assume $ \lambda< -(\sup a-\mu) $. Then {it follows from} equation \eqref{eqlem:twomeasures} that $ \varphi_s\equiv 0 $. Moreover, $ \varphi_{ac}(y)=\frac{\mu (M\star \varphi)(y)}{-\lambda - (a(y)-\mu)} $ is then a positive bounded continuous function and satisfies: 
    \begin{equation*}
        \mu( M\star \varphi_{ac} -\varphi_{ac}) +(a(y)+\lambda)\varphi_{ac}=0
    \end{equation*}
    in the classical sense.

	Let us show that $ \lambda=\lambda_1 $. Let $ ({\overline\lambda}, {\overline\varphi})\in \mathbb R\times {C(\Omega)} $ be a supersolution to \eqref{eq:eigen}, i.e. $ \overline\varphi>0 $ and 
    \begin{equation*}
        \mu M\star {\overline\varphi} + {\overline\varphi}(a(y)-\mu+{\overline\lambda}) \leq 0.
    \end{equation*}
    Then $ {\overline\varphi}(0){(-a(0)+\mu-{\overline\lambda})} \geq \mu M\star {\overline\varphi}>0 $ and thus $ {\overline\lambda} <-(a(0)-\mu)=-(\sup a-\mu) $.
    {Moreover $\overline\varphi(y)\geq \frac{\mu M\star\overline\varphi(y)}{\mu-a(y)-\overline\lambda}\geq\frac{\mu m_0 \int\varphi}{-(\inf a + \overline\lambda-\mu)}>0$ and thus $\overline\varphi$ is uniformly bounded from below. In particular,}
    $\alpha:=\sup\{\zeta>0\,|\, \forall y\in\overline\Omega, \zeta\varphi_{ac}(y)\leq {\overline\varphi}(y)\}$
    is well-defined and positive. By definition of $ \alpha $ we have $ \alpha\varphi_{ac}(y)\leq {\overline\varphi}(y) $ for any $ y\in\overline\Omega $, and there exists {a converging sequence $ \Omega\ni y_n\to y\in\overline\Omega $ such that $ \alpha\varphi_{ac}(y_n)-{\overline\varphi}(y_n)\to 0 $. Up to further extraction $\varphi_{ac}(y_n) $ converges to a positive limit that we denote $\varphi_{ac}(y)$. We have then
    \begin{align*}
	    0&\geq\mu\int_{\overline\Omega}M(y_n, z)\big({\overline\varphi}(z)-\alpha\varphi_{ac}(z)\big)dz \\
	    &\quad +  \big({\overline\varphi}(y_n)-\alpha\varphi_{ac}(y_n)\big)(a(y_n)-\mu) +{\overline\lambda}\overline\varphi(y_n) - \lambda\alpha\varphi_{ac}(y_n)\\
         &\geq 0 + \big({\overline\varphi}(y_n)-\alpha\varphi_{ac}(y_n)\big)(a(y_n)-\mu) +{\overline\lambda}\overline\varphi(y_n) - \lambda\alpha\varphi_{ac}(y_n) \\
	    &=  ({\overline\lambda} - \lambda)\alpha\varphi_{ac}(y)+o_{n\to\infty}(1).
    \end{align*}
    Taking the limit $n\to \infty$, we have shown $ {\overline\lambda}\leq \lambda $. Hence, 
    \begin{equation*}
        \lambda \geq \sup\{{\overline\lambda}\,|\, \exists \psi\in C(\Omega), \psi>0 \text{ s.t. } \mu(M\star\psi - \psi) + \psi(a(y)+{\overline\lambda})\leq 0\} = \lambda_1. 
    \end{equation*}}
    The reverse inequality $ \lambda\leq \lambda_1 $ is clear since $ \varphi_{ac} $ is a supersolution to \eqref{eq:eigen}. Thus $ \lambda=\lambda_1 $. 

    In this case, we notice that
    \begin{equation*}
        \mathcal M[(\sup a-a(y))\varphi_{ac}] = \mu M\star \varphi_{ac}>(\sup a-a(y))\varphi_{ac}.
    \end{equation*}
    Hence, by a classical comparison argument, $ \gamma_1> 1 $.

    \medskip

    2. Let us assume now  $ \lambda = -(\sup a-\mu )$.

    We define the auxiliary function  $ \Psi(y):=\varphi_{ac}(y)(\sup a - a(y))=\mu(M\star \varphi)  $. Then $ \Psi $ is a nontrivial positive bounded continuous function which satisfies:
    \begin{equation*}
        \mathcal M[\Psi]-\Psi=\mu(M\star \varphi_{ac}-\varphi_{ac})+(a(y)+\lambda)\varphi_{ac}=-\mu M\star \varphi_s\leq 0.
    \end{equation*}
    Thus, by a classical comparison argument,  $ \gamma_1\leq 1 $.

    We claim that $ \lambda_1=\lambda$. {As above, $\varphi_{ac} $ is a supersolution to \eqref{eq:eigen}, and thus $\lambda\leq \lambda_1$. Assume} by contradiction that $ \lambda_1< \lambda $.  By the existence property \cite[Theorem 1.1]{Cov-13}, there exists  a continuous function $ \varphi_1 >0 $ associated with $ \lambda_1 $. Since $ \lambda_1 < \lambda = -(\sup a-\mu) $, point 1 above then applies to  $ (\lambda_1, \varphi_1) $ and we have $ \gamma_1> 1 $. This is a contradiction.  Hence $ \lambda=\lambda_1 $. 

    \medskip

    \textbf{Step 3}: We show \ref{item:eigenpair-C0}, \ref{item:eigenpair-L1}, and \ref{item:eigenpair-M1}.

    Assume $ \lambda_1<-(\sup a-\mu) $. Then, $ \gamma_1>1 $, and the fact that any measure eigenvector is a continuous eigenfunction has been shown in Step 2.

    Assume  $ \lambda_1=-(\sup a-\mu) $ and $ \varphi_s\equiv0 $. {Let 
    $ \Psi(y):=(\sup a -a(y))\varphi_{ac}(y) $. Then, by a straightforward computation, $\Psi$ satisfies $\Psi(y)=\mu M\star\varphi(y)$, which shows that $\Psi $ is bounded and continuous. 
    We remark that:}
    \begin{equation*}
        \mathcal M[\Psi]-\Psi=\mu M\star \varphi_{ac}-(\sup a -a )\varphi_{ac}=-\mu M\star \varphi_s=0.
    \end{equation*}
    By the Krein-Rutman Theorem, we have $ \gamma_1=1 $ and  $ \varphi\equiv \varphi_{ac} $ is unique up to multiplication by a scalar.

    Assume that $ \lambda_1=-(\sup a-\mu) $ and $ \varphi_s\not\equiv 0 $. Let   $ \Psi(y):=(\sup a -a(y))\varphi_{ac}(y) $, then
    \begin{equation*}
        \mathcal M[\Psi]-\Psi=-\mu M\star \varphi_s<0
    \end{equation*}
    and thus $ \gamma_1<1 $. Notice that in this case, the second line in equation \eqref{eqlem:twomeasures} implies by definition
    $ \varphi_s\left(\{y\in\overline\Omega\,|\, a(y)\neq \sup a\}\right) = 0 $,
    hence $ \mathrm{supp}~\varphi_s\subset\Omega_0 $.

    Since we have investigated all the possibilities (recall $ \lambda\leq-(\sup a - \mu) $), the equivalence holds in each case. 
    This finishes the proof of Theorem \ref{thm:eigenpair}.
\end{proof}

\subsection{The critical mutation rate}\label{ssec:critmut}

In this subsection we investigate further the linear eigenvalue problem \eqref{eq:eigen}, with $ \lambda = \lambda_1 $ as compelled by Theorem \ref{thm:eigenpair}, under Assumption \ref{hyp:dega}.

We introduce the notion of {\it critical mutation rate}, which distinguishes between the existence of a bounded continuous eigenfunction for equation \eqref{eq:eigen} and the existence of a singular measure. 

\begin{thm}[Critical mutation rate]\label{thm:smallmu}
    Let Assumption \ref{hyp:dega} hold. Then, there exists $ \mu_0=\mu_0(\Omega, M, \sup a-a) $ such that for any $ 0<\mu< \mu_0$, problem \eqref{eq:eigen} has only singular measures solutions with a singularity concentrated in $ \Omega_0$ (in which case $ \lambda_1=-(\sup a-\mu) $ from Theorem \ref{thm:eigenpair}), whereas for $ \mu> \mu_0$ equation \eqref{eq:eigen} has only bounded continuous eigenfunctions. 

    Finally, $ \mu_0 = \frac{1}{\gamma_1^1} $ where $ -\gamma_1^1 $ is the principal eigenvalue of the operator
    \begin{equation*}
        \mathcal M^1[\psi]=\int_\Omega M(y,z)\frac{\psi(z)}{\sup a-a(z)}dz,
    \end{equation*}
    acting on bounded continuous functions.
\end{thm}
\begin{proof}
    Let us define, for $ \psi\in C_b(\Omega)$,  
    $\mathcal M^\mu[\psi]=\mu\int_\Omega M(y,z)\frac{\psi(z)}{\sup a-a(z)}dz$.
    Then by the Krein-Rutman Theorem there exists 
    a unique principal eigenpair $ (-\gamma_1^\mu, \Phi^\mu)$ satisfying $ \gamma_1^\mu>0 $,  $ \Phi^\mu(y) >0 $, $ \sup\Phi^\mu=1 $ and
    $ \mathcal M^\mu[\Phi^\mu]=\gamma_1^\mu \Phi^\mu$.
    Since $ \mathcal M^\mu = \mu\mathcal M^1 $, we deduce from the uniqueness of $ (-\gamma_1^\mu, \Phi^\mu) $ that the equalities 
    $ \gamma_1^\mu=\mu\gamma_1^1 $ and $\Phi^\mu=\Phi^1$ hold for any $ \mu>0$.
    The result then follows from the trichotomy in Theorem \ref{thm:eigenpair}
\end{proof}

We can now summarize our findings and prove Proposition \ref{prop:vpp}.
\begin{proof}[Proof of Proposition \ref{prop:vpp}]
    The first part, under Assumption \ref{hyp:gen}, follows from Pro\-position \ref{thm:eigenpair}, while the second part, under Assumption \ref{hyp:dega}, follows from Theorem \ref{thm:smallmu}. 
\end{proof}

We prove below that $ \mu_0 $  is linked to the steepness of 
the fitness function $a$ near its maximum. This property will be used  in the proof of Theorem \ref{thm:concentration}.
\begin{cor}[Monotony of $ \mu_0 $]\label{cor:compmucrit}
    Let Assumption \ref{hyp:dega} hold and $b $ be a continuous function on $ \overline\Omega $, satisfying
    \begin{equation*}
        \forall y\in\overline\Omega,\quad \sup a-a(y)\leq \sup b-b(y).
    \end{equation*}
    Then  we have
    \begin{equation*}
        \mu_0(\Omega, M, \sup a-a) \leq \mu_0(\Omega, M, \sup b-b),
    \end{equation*}
    where $ \mu_0 $ is defined in Theorem \ref{thm:smallmu}.
\end{cor}
\begin{proof}
	{It follows from our assumptions that}, for $ y\in\overline\Omega$:
    \begin{equation}\label{eqcor:compmucrit}
        0<\frac{1}{{\sup b}-b(y)}\leq \frac{1}{{\sup a}-a(y)}.
    \end{equation}
    In particular $ y\mapsto\frac{1}{{\sup b}-b(y)}\in L^1(\Omega) $. Thus Theorem \ref{thm:smallmu} can be applied with both $a $ and $ b$. 

    We claim that  $ \gamma_1^b\leq \gamma_1^a $, where $\gamma_1^b $, $ \gamma_1^a $ denote the first eigenvalue of the operator 
    $  \mathcal M_b[\psi]= \int_\Omega M(y,z)\frac{\psi(z)}{\sup b-b(z)}dz $ and $  \mathcal M_a[\psi]=\int_\Omega M(y,z)\frac{\psi(z)}{\sup a-a(z)}dz $ acting on the function $ \psi \in C_b(\Omega) $, respectively. 
	Indeed, let $ \varphi^a\in C_b(\Omega)$, $ \varphi^a>0 $ satisfy $ \int_\Omega M(y,z) \frac{\varphi^a(z)}{{\sup a}-a(z)}dz=\gamma_1^a\varphi^a(y) $ and $\varphi^b\in C_b(\Omega) $, $ \varphi^b>0 $ respectively satisfy $ \int_\Omega M(y,z) \frac{\varphi^b(z)}{{\sup b}-b(z)}dz=\gamma_1^b\varphi^b(y)$. Up to multiplication by a positive constant, we assume {without loss of generality} that $ \varphi^b\leq \varphi^a $ and that there exists $ y\in\overline\Omega $ satisfying $ \varphi^b(y)=\varphi^a(y)=1 $. At this point, we have
    \begin{equation*}
        \gamma_1^b = \int_\Omega M(y,z) \frac{\varphi^b(z)}{{\sup b}-b(z)}dz \leq \int_\Omega M(y,z) \frac{\varphi^a(z)}{{\sup a}-a(z)}dz=\gamma_1^a.
    \end{equation*}

    We conclude that 
    \begin{equation*}
        \mu_0(\Omega, M, \sup a-a)=\frac{1}{\gamma_1^a}\leq \frac{1}{\gamma_1^b} = \mu_0(\Omega, M, \sup b-b)
    \end{equation*}
    which finishes the proof of Corollary \ref{cor:compmucrit}.
\end{proof}

\subsection{Approximation by a degenerating elliptic eigenvalue problem}

Here we show that the previously introduced principal eigenvalue can be approximated by an elliptic Neumann eigenvalue. 

\begin{thm}[Approximating $ \lambda_1 $ by vanishing viscosity]\label{thm:eigen}
    Let Assumption \ref{hyp:gen} hold, and $ (\lambda_1^\varepsilon,\varphi^\varepsilon(y)>0) $ be the solution to the principal eigenproblem:
    \begin{equation}\label{eq:eigeneps}
        \begin{system}{lr} 
            -\varepsilon\Delta \varphi^\varepsilon -\mu(M\star \varphi^\varepsilon - \varphi^\varepsilon) = a(y)\varphi^\varepsilon + \lambda_1^\varepsilon\varphi^\varepsilon & \text{ in } \Omega \\
            \frac{\partial\varphi^\varepsilon}{\partial \nu}=0 & \text{ on } \partial \Omega,
        \end{system}
    \end{equation}
	with $ \int_\Omega\varphi^\varepsilon(z)dz=1 $, {where $\nu $ is the unit normal vector}.

    Then
    $ \lim_{\varepsilon\to 0}\lambda_1^\varepsilon=\lambda_1 $, where $\lambda_1 $ is the principal eigenvalue defined by \eqref{eq:vpp}. 
\end{thm}
\begin{proof}
    We divide the proof into three steps.

    \textbf{Step 1:} We show that $ \lambda_1^\varepsilon $ is bounded when $\varepsilon\to 0$.

    Integrating equation \eqref{eq:eigeneps} by parts, we have
    $0 = \int_\Omega(\lambda_1^\varepsilon+a(y))\varphi^\varepsilon dy$.
    In particular, the function $ a(y)+\lambda_1^\varepsilon $ takes both nonnegative and nonpositive values. Hence, we have  
    $-\sup a\leq \lambda_1^\varepsilon\leq -\inf a$, and  $ (\lambda_1^\varepsilon)_{\varepsilon>0} $ is bounded.
    \medskip

    \textbf{Step 2:} We identify the limit of converging subsequences.

    Let $ \lambda_1^{\varepsilon_n} $ be a converging sequence and $ \lambda_1^0:=\lim \lambda_1^{\varepsilon_n} $. Then $ \varphi^{\varepsilon_n} $ satisfies,
    for any $ \psi\in C^2(\overline\Omega) $,
    \begin{equation*}
	    \int_\Omega -{\varepsilon_n} \varphi^{\varepsilon_n}\Delta \psi dy - \varepsilon_n \int_{\partial\Omega}\varphi^{\varepsilon_n}\frac{\partial\psi}{\partial\nu}dS - \int \mu(M\star \varphi^{\varepsilon_n}-\varphi^{\varepsilon_n})\psi-a(y)\varphi^{\varepsilon_n}\psi 
        =\lambda_1^{\varepsilon_n}\int\varphi^{\varepsilon_n}\psi.
    \end{equation*}
	Let {
	\begin{equation}\label{eq:F0}
		F_0:=\left\{ \psi\in C^2(\overline\Omega)\,|\,\forall y\in\partial\Omega, \frac{\partial\psi}{\partial\nu}(y)=0\right\}
	\end{equation}}
	denote the space of functions in $ C^2(\Omega) $ with zero boundary flux as in Lemma \ref{lem:zerofluxdense} item \ref{item:zeroflux-Omega}. For $ \psi\in F_0 $, this equation becomes:
    \begin{equation*}
        \int_\Omega -{\varepsilon_n}\varphi^{\varepsilon_n}\Delta \psi dy - \int \mu(M\star \varphi^{\varepsilon_n}-\varphi^{\varepsilon_n})\psi-a(y)\varphi^{\varepsilon_n}\psi =\lambda_1^{\varepsilon_n}\int\varphi^{\varepsilon_n}\psi.
    \end{equation*}

	Since $ \int_\Omega\varphi^{\varepsilon_n}(y)dy =1 $ and $ \overline\Omega $ is compact {and by} Prokhorov's Theorem \cite[Theorem 8.6.2]{Bog-07}, the sequence $ (\varphi^{\varepsilon_n}) $ is precompact for the weak topology in $ M^1(\overline\Omega) $, and there exists a weakly convergent subsequence $ \varphi^{\varepsilon_n'} $, which converges  to a nonnegative Radon measure $ \varphi $. Since $ 1\in C_c(\overline\Omega)$, we have $\lim \int_{\overline\Omega}\varphi^{\varepsilon_n'}=\int_{\overline\Omega}\varphi (dy) = 1$. Hence $ \varphi $ is non-trivial. Moreover, we have
    \begin{equation}\label{eq:limiteigenweak}
        \mu\int_\Omega \int_\Omega M(y,z)d\varphi(z) \psi(y)dy + \int_\Omega (a(y)-\mu)\psi(y)d\varphi(y) + \lambda_1^0\int_\Omega\psi(y)d\varphi(y) =0
    \end{equation}
    for any test function $ \psi\in F_0 $. Since $ F_0 $ is densely embedded in $ C_b(\overline\Omega) $ by Lemma \ref{lem:zerofluxdense}, \eqref{eq:limiteigenweak} holds for any $\psi\in C_b(\overline\Omega) $. 
	{Applying}  Proposition \ref{prop:vpp}, we have then 
    $     \lambda_1^0=\lambda_1 $.
    \medskip

    \textbf{Step 3:} Conclusion.

    We have shown that for any sequence $ \varepsilon_n\to 0$, there exists a subsequence $ \varepsilon'_n\to 0 $ such that $ \lambda_1^{\varepsilon'_n}\to \lambda_1 $. Thus  $\lambda_1^\varepsilon\to \lambda_1 $ when $ \varepsilon\to 0$.
\end{proof}

\section{Stationary states in trait} \label{sec:stat}

This section deals with stationary states for \eqref{eq:evol}. In particular, we prove Theorem \ref{thm:survival} and Theorem \ref{thm:concentration} via a bifurcation argument. 

\subsection{Regularized solutions}\label{sec:reg-evolstat}

We investigate the existence of positive solutions $p=p(y) $ to the following problem 
\begin{equation}\label{eq:evolstat-beta}
    \begin{system}{ll}
        -\varepsilon \Delta p - \mu(M\star p-p)=p(a(y)-K\star p-\beta p) & \text{ in } \Omega \\
        \frac{\partial p}{\partial\nu}=0 & \text{ on } \partial\Omega,
    \end{system}
\end{equation}
for any $ \beta\geq 0 $. We prove the existence of positive solutions for \eqref{eq:evolstat-beta} when $ \lambda_1^\varepsilon<0 $. We plan to let $ \varepsilon\to 0 $ with $ \beta=0 $ in Section \ref{sec:evolstat}, in order to prove the existence of stationary solutions to \eqref{eq:evol}. The reason why we include a weight $ \beta\geq 0 $ on the competition term in equation \eqref{eq:evolstat-beta} is that solutions to the latter  will be used as subsolutions in the construction of traveling waves in Section \ref{sec:TW}. 

Throughout this subsection we denote $ (\lambda_1^\varepsilon, \varphi^\varepsilon) $ the eigenpair of the regularized problem, solving  \eqref{eq:eigeneps}. Notice that $ (\lambda_1^\varepsilon, \varphi^\varepsilon) $ is independent from $ \beta $. Our main result is the following:

\begin{thm}[Regularized steady states]\label{thm:ex-stat-eps}
    Let Assumption \ref{hyp:gen} hold, $\varepsilon>0 $, $ (\lambda_1^\varepsilon, \varphi^\varepsilon) $ be defined by  \eqref{eq:eigeneps}, and $ \beta\geq 0$.
    \begin{enumerate}[label=(\textit{\roman*}), ref=(\textit{\roman*})]
        \item\label{item:exstat-extinction} Assume $\lambda_1^\varepsilon >0 $. Then $ 0 $ is the only nonnegative solution to \eqref{eq:evolstat-beta}.
        \item\label{item:exstat-survival} Assume $ \lambda_1^\varepsilon <0 $. Then there exists a positive solution to  \eqref{eq:evolstat-beta} for any $ \beta\geq 0$.
    \end{enumerate}
\end{thm}
Item \ref{item:exstat-extinction} is rather trivial and we will discuss it later in the proof of Theorem \ref{thm:ex-stat-eps}. The actual construction in the case $ \lambda_1^\varepsilon<0 $ is more involved. Our method is inspired by the similar situation in \cite{Alf-Gri-18}. We start by establishing \textit{a priori} estimates on the solutions $ p $ to \eqref{eq:evolstat-beta}. 
\begin{lem}[\textit{A priori} estimates on $p$]\label{lem:apriorip}
    Let Assumption \ref{hyp:gen} hold, let $\varepsilon>0 $, $\beta\geq 0 $ and  $ p $ be a nonnegative nontrivial solution to \eqref{eq:evolstat-beta}. Then:
    \begin{enumerate}[label=(\textit{\roman*}), ref=(\textit{\roman*})]
        \item \label{item:apriorip-positive} $ p $ is positive.
        \item \label{item:apriorip-boundinfty} If $ \beta=0 $, there exists a positive constant $ C=C(\Omega, \varepsilon, \mu, \Vert a\Vert_{L^\infty}, m_\infty, k_0, k_\infty) $ such that $ \Vert p\Vert_{L^\infty}\leq C $. If $ \beta>0 $ then we have  $ \sup p\leq \frac{\sup a}{\beta} $.
    \end{enumerate}
\end{lem}
\begin{proof}
	Point \ref{item:apriorip-positive} {follows from the strong maximum principle}. We turn our attention to point \ref{item:apriorip-boundinfty}. 

    Assume first $ \beta>0$. Let $ y\in\overline\Omega$ such that $ p(y)=\sup_{z\in\Omega} p(z) $ and assume by contradiction that $ p(y)>\frac{\sup a}{\beta} $. If $ y\in\Omega$, then we have 
    \begin{equation*}
        0\leq -\varepsilon \Delta_y p(y) -\mu(M\star p-p) = p(a(y)-K\star p-\beta p)<0
    \end{equation*}
    which is a contradiction. If $ y\in\partial\Omega $, then $ \mu(M\star p-p)\leq 0 $ and  $ a-K\star p-\beta p\leq 0 $ in a neighbourhood of $ y $, and thus $- \varepsilon \Delta p - (a(y)-K\star p-\beta p)p\leq 0 $ in a neighbourhood of $ y$. {It follows from Hopf's Lemma that} $ \frac{\partial p}{\partial\nu}(y)>0 $, which contradicts the Neumann boundary conditions satisfied by $ p$. Hence $ \sup p\leq \frac{\sup a}{\beta} $.
    \medskip

    We turn our attention to the case $ \beta=0 $, which is more involved. We divide the proof in four steps. 

    \textbf{Step 1: }We establish a bound on $ \int_\Omega p(y)dy $.

    Integrating over $ \Omega $, we have 
    \begin{equation*}
        \int_\Omega a(y)p(y)\mathrm dy - \int_\Omega\int_\Omega p(y)K(y,z)p(z)\mathrm dy\mathrm dz =\beta\int_\Omega p^2(y)dy\geq 0.
    \end{equation*}
    Thus $ \int a(y)p(y)dy\geq k_0\left(\int_\Omega p(y)dy \right)^2 $ and
    \begin{equation}\label{eq:apriorip-mass} 
        \int_\Omega p(y)dy\leq \frac{ \sup a}{k_0}.
    \end{equation}

    \textbf{Step 2:} We reduce the problem to a boundary estimate. 

	{By a direct application of} the local maximum principle \cite[Theorem 9.20]{Gil-Tru-01}, for any ball 
    $ B_{R}(y)\subset \Omega $, there exists a constant  $ C=C(R,\varepsilon, \Vert a\Vert_{L^\infty}, k_0, k_\infty, \mu, m_{\infty} )>0 $ such that 
    $ 
    \underset{B_{R/2}(y)}{\sup}\,p\leq C 
    $. 
    This shows an interior bound for any point at distance $ R $ from $ \partial\Omega $.

    To show that this estimate does not degenerate near the boundary, we use a coronation argument. Let $ d(y,\partial\Omega):=\inf_{z\in\partial\Omega}|y-z|$, and
    \begin{equation*}
        \Omega_R:=\{y\in\Omega\,|\, d(y,\partial\Omega)< R\}
    \end{equation*}
    for any $ R>0$. As noted in \cite{Foo-84}, the function $ y\mapsto d(y, \partial\Omega) $ has $ C^3 $ regularity on a tubular neighbourhood of $ \partial \Omega $. In particular, $ \partial \Omega_R \bsl\partial\Omega $ is $ C^3 $  for $ R $ small enough, since $ \nabla d \neq 0 $ in this neighbourhood.  Moreover, by the comparison principle in narrow domains \cite[Proposition 1.1]{Ber-Nir-91}, the maximum principle holds for the operator $ -\varepsilon\Delta v-(a(y)-\mu) v$
    in $ \Omega_R$ provided $ |\Omega_R| $ is small enough, meaning that if $ v $ satisfies $ -\varepsilon\Delta v - (a(y)-\mu) v\geq 0 $ in $ \Omega_R $ and $ v\geq 0 $ on $ \partial\Omega_R$, then $ v\geq 0$. In particular, we choose $ R $ small enough for  this property to hold. 

    At this point, $p\leq C $ in $ \Omega\bsl\Omega_R $ and comparison holds in $ \Omega_R $. 
    \medskip

    \textbf{Step 3:} We construct a supersolution.

    Notice that, in contrast with \cite{Alf-Cov-Rao-13} where Dirichlet boundary conditions are used, we need an additional argument to deal with the Neumann boundary conditions. 
    Since the comparison principle holds in the narrow domain $ \Omega_R$, the Fredholm alternative implies that,  for any $ \delta\in (0,1] $, there exists a unique (classical) solution to the system: 
    \begin{equation*}
        \begin{system}{ll}
            -\varepsilon\Delta v^\delta-(a(y)-\mu)v^\delta = \mu m_\infty \frac{\sup a}{k_0} & \text{ in } \Omega_R \\
            v^\delta=C & \text{ on }\partial \Omega_R\bsl\partial\Omega \\
            \delta v^\delta + (1-\delta)\frac{\partial v^\delta}{\partial\nu}=\delta & \text{ on }~\partial\Omega.
        \end{system}
    \end{equation*}
	{As a result of} the classical Schauder interior and boundary estimates, the mapping $ \delta\mapsto v^\delta $ is  continuous from $ (0, 1] $ to $ C_b(\Omega_R) $. Moreover, $ v^{\delta} $ is positive for $ \delta\in(0, 1] $ by virtue of the maximum principle. 

	Next, {still by a direct application of} the Schauder estimates, $ (v^\delta)_{0<\delta\leq 1} $ is precompact and there exists a sequence $ \delta_n\to 0 $ and $ v\in C^2 $ such that $ v^{\delta_n}\to v $ in $ C^2_{loc}(\Omega_R)\cap C^1(\overline\Omega_R) $. Then $ v\geq 0 $ satisfies:
    \begin{equation*}
        \begin{system}{ll}
            -\varepsilon\Delta v-(a(y)-\mu)v = \mu m_\infty \frac{\sup a}{k_0} & \text{ in } \Omega_R \\
            v=C & \text{ on }\partial \Omega_R\bsl\partial\Omega \\
            \frac{\partial v}{\partial\nu}=0  & \text{ on }\partial\Omega.
        \end{system}
    \end{equation*}
    {By a direct application of} the strong maximum principle and Hopf's Lemma, we have  $ v>0 $ on $\overline\Omega_R $.
    \medskip

    \textbf{Step 4:} We show that $ p\leq v $ on $ \Omega_R $.

    Let $ p $ be a solution to \eqref{eq:evolstat-beta} and select
    $\alpha:=\inf\{\zeta>0\,|\, \zeta v\geq p \text{ in } \Omega_R\}$.

    Assume by contradiction that $ \alpha>1 $. Then there exists $ y_0\in\overline\Omega_R $ such that the equality $ p(y_0)=\alpha v(y_0) $ holds, and $ \alpha v-p\geq 0 $. In particular $ y_0 $ is a zero minimum for the function $ \alpha v - p $. {Because of} the boundary conditions satisfied by $ p $ and $ v $,  $ y_0 $ cannot be in $\partial\Omega_R $. $ y_0 $ is then an interior local minimum to $ \alpha v-p $ and thus
    \begin{align*}
        0&\geq -\varepsilon\Delta (\alpha v-p)(y_0) = (a(y_0)-\mu)(\alpha v-p)(y_0)+\alpha \mu m_\infty\frac{ \sup a}{k_0} \\
         &\quad- \mu (M\star  p)(y_0) +p(y_0)(K\star p)(y_0) \\ 
         &> \alpha \mu m_\infty\frac{ \sup a}{k_0}-\mu (M\star  p)(y_0)\geq 0,
    \end{align*}
    {using estimate \eqref{eq:apriorip-mass},} which is a contradiction. Thus $ \alpha\leq 1 $. 

    This shows that $ p\leq v $. Since $ v $ is a bounded function, we have our uniform bound for $ p $ in $ \Omega_R$. In $ \Omega\bsl\Omega_R $,  we have $ p\leq C $. This ends the proof of Lemma \ref{lem:apriorip}.
\end{proof}
In order to proceed to the proof of Theorem \ref{thm:ex-stat-eps}, we yet need an additional technical remark.
\begin{lem}[Fréchet differentiability at 0]\label{lem:Frechet}
    Let Assumption \ref{hyp:gen} hold, $ \beta\geq 0 $ and 
    \begin{equation*}
        \begin{array}{rccl}
            G:&C_b(\Omega)&\to&C_b(\Omega) \\
              &p(y) & \mapsto & p(y) (K\star p)(y) + \beta p^2(y),
        \end{array}
    \end{equation*}
    then $ G $ is Fréchet differentiable at $ p=0 $ and its derivative is $ DG(p)=0 $.
\end{lem}
\begin{proof}
    This comes from the remark 
    \begin{align*}
        \left|\int_\Omega K(y,z)p(z)dzp(y)+p^2(y)\right|&\leq \int_\Omega K(y,z)|p(z)|dz |p(y)| + \beta p^2(y)\\
                                                        & \leq k_\infty|\Omega|\Vert p\Vert_{C_b(\Omega)}^2 + \beta \Vert p\Vert_{C_b(\Omega)}^2
        \qedhere
    \end{align*}
\end{proof}

\begin{proof}[Proof of Theorem \ref{thm:ex-stat-eps}]
    \textbf{Step 1:} We prove item \ref{item:exstat-extinction}. 

    We  assume $ \lambda_1^\varepsilon >0 $. We recall that  $ (\lambda_1^\varepsilon,\varphi^\varepsilon) $ is the solution to \eqref{eq:eigeneps} with the normalization $ \int_\Omega\varphi^\varepsilon(y)dy=1 $.
    Let $ p>0 $ be a nonnegative solution to \eqref{eq:evolstat-beta} in $ \Omega $. Since $ p $ is bounded and $ \varphi^\varepsilon $ is positive on $ \overline\Omega $, the quantity 
    $ \alpha := \inf\{\zeta>0\,|\, \zeta\varphi^\varepsilon > p\} $
    is well-defined and finite. 
    Then, there exists $ y_0\in\overline\Omega $ such that $ p(y_0)=\alpha\varphi^\varepsilon(y_0) $. Remark that $ y_0 $ is a minimum to the nonnegative function $ \alpha\varphi^\varepsilon - p $. If $ y_0 \in\partial\Omega $, then Hopf's Lemma implies $ \frac{\partial (\alpha\varphi^\varepsilon - p)}{\partial\nu}(y_0)<0 $, which contradicts the Neumann boundary conditions satisfied by $ p $ and $ \varphi^\varepsilon $. Thus $ y_0\in\Omega$. Evaluating equation \eqref{eq:evolstat-beta}, we have:
    \begin{align*}
        0&\geq-\varepsilon \Delta(\alpha\varphi^\varepsilon - p)(y_0)=\mu\big(M\star (\alpha\varphi^\varepsilon-p) - (\alpha\varphi^\varepsilon-p)\big)\\
         &\quad +a(y_0)\big(\alpha\varphi^\varepsilon(y_0)-p(y_0)\big)+p(y_0)(K\star p)(y_0)+\beta p^2(y_0) + \lambda_1^\varepsilon \alpha\varphi^\varepsilon \\
         &\geq p(y_0)(K\star p)(y_0)+\beta p^2(y_0) + \lambda_1^\varepsilon \alpha\varphi^\varepsilon >0
    \end{align*}
    which is a contradiction.

    \medskip

    \textbf{Step 2 :} We prove item \ref{item:exstat-survival}. 

    We assume $ \lambda_1^\varepsilon < 0$.
    We argue as in \cite{Alf-Gri-18}:  if the nonlinearity is negligible near $ 0 $ and we can prove local boundedness of the solutions in $ L^\infty $, then we can prove existence through a bifurcation argument.  This requires a topological result stated in Appendix \ref{appendix:topo}. 

    More precisely, for $ \alpha\in\mathbb R $ and  $ p \in C_b(\Omega) $, we let $ F(\alpha, p)=\tilde p $ where $ \tilde p $ is the unique solution to:
    \begin{equation*}
        \begin{system}{ll} 
            -\varepsilon\Delta \tilde p + (\sup a-a(y))\tilde p - \mu (M\star \tilde p-\tilde p) = \alpha p - G(p) & \text{ in } \Omega \\
            \frac{\partial \tilde p}{\partial \nu} = 0   & \text{ on } \partial \Omega
        \end{system}
    \end{equation*}
	where $ G $ is as in Lemma \ref{lem:Frechet}. Notice that $ \sup a-a(y)\geq 0 $, so comparison applies and the operator $ F $ is well-defined {due to} the Fredholm alternative. In particular, for each $ \alpha\in\mathbb R $,  $ F(\alpha, \cdot) $ is Fréchet differentiable near $ 0 $ and its derivative is the linear operator $ \alpha T $, where $ Tp=\tilde p $ and $ \tilde p $ is defined by:
    \begin{equation*}
        \begin{system}{ll} 
            -\varepsilon\Delta \tilde p + (\sup a-a(y))\tilde p - \mu (M\star \tilde p - \tilde p) =  p  & \text{ in } \Omega \\
            \frac{\partial \tilde p}{\partial \nu} = 0   & \text{ on } \partial \Omega.
        \end{system}
    \end{equation*}
	Let $ C:=\{p\in C_b(\Omega)\,|\, p\geq 0\} $. {By} a classical comparison argument, $ T $ maps the cone $ C\bsl\{0\} $ into $ Int~C=\{p\in C_b(\Omega)\,|\, p>0 \} $. By virtue of the Krein-Rutman Theorem \cite[Theorem 6.13]{Bre-11}, $ T $ has a {\em first}\footnote{We stress that this {\em first} eigenvalue is {\em not} the {\em principal} eigenvalue of the problem $ F(\alpha, p)=p $, but the algebraic eigenvalue.}  eigenvalue $ \lambda(T) $ (satisfying $ T\psi=\lambda(T)\psi $ for a $ \psi>0$) and we have the formula $ \lambda(T) = \frac{1}{\lambda_1^\varepsilon + \sup a}$.

    We now check one by one the hypotheses of Theorem \ref{thm:orientedbifurcation}:

    {\bf 1.} Clearly we have $ F(\alpha, 0)=0 $ for any $ \alpha\in\mathbb R$.

	{\bf 2.} {It follows from Lemma  \ref{lem:Frechet} that}  $ G $ is Fréchet differentiable near 0 with derivative 0. As a consequence, $ F(\alpha, \cdot) $ is Fréchet differentiable near 0 with derivative $ \alpha T $.

    {\bf 3.} $T$ satisfies the hypotheses of the Krein-Rutman Theorem.

	{\bf 4.} {It follows from Lemma \ref{lem:apriorip} that} the solutions to $ F(\alpha, p)=p $ are locally uniformly bounded in $ \alpha $.

    {\bf 5.} Since any nontrivial nonnegative fixed point $ p $ is positive, there is no nontrivial fixed point in the boundary of $ C $.

    Thus, applying Theorem \ref{thm:orientedbifurcation}, there exists a branch of solutions $ \mathcal C $ connecting $ \alpha=\frac{1}{\lambda(T)} $ to either $ \alpha\to+\infty $ or $ \alpha\to-\infty $. 

	{By} the uniqueness in the Krein-Rutman Theorem, if $ \lambda^\alpha $ denotes the principal eigenvalue associated with $ F(\alpha, p)=p$, we have   $ \lambda^\alpha=\lambda_1^\varepsilon+\sup a-\alpha=\frac{1}{\lambda(T)}-\alpha$. In particular, for $ \alpha < -\sup a -\lambda_1^\varepsilon $, we deduce from Step 1 that there cannot exist a solution to $ F(\alpha, p)=p $ in $ C $.
    Thus $ \mathcal C $  connects $\frac{1}{\lambda(T)} $ to $ +\infty$. In particular, there exists a positive  solution for $ \alpha = \sup a = \frac{1}{\lambda(T)}-\lambda_1^\varepsilon>\frac{1}{\lambda(T)} $, which solves \eqref{eq:evolstat-beta}. This ends the proof of Theorem \ref{thm:ex-stat-eps}.
\end{proof}

We now prove a lower estimate for solutions to \eqref{eq:evolstat-beta}, which  is crucial for the construction of traveling waves, but will not be used in the meantime.  
We stress that in the lemma below, the constant $ \rho_\beta $ is independent from $ \varepsilon $.
\begin{lem}[$p^{\varepsilon, \beta} $ does not vanish]\label{lem:beta-lowerbound}
    Let Assumption \ref{hyp:gen} be satisfied,  let  $\beta>0 $ and  $ \lambda_1<0 $. Let finally $ p^{\varepsilon, \beta}$ be a solution to \eqref{eq:evolstat-beta}. Then, there exists constants $\varepsilon_0=\varepsilon_0(\Omega, \mu, M, a)>0 $ and  $ \rho_\beta=\rho_\beta(\Omega, M, a, \beta)>0 $ such that if $ \varepsilon\leq\varepsilon_0 $, then
    \begin{equation*}
        \inf_\Omega p^{\varepsilon,\beta}\geq\rho_\beta .
    \end{equation*}
\end{lem}
\begin{proof}
    This proof is inspired by the one of \cite[Lemma 5.2]{Cov-Dav-Mar-13}.

    \paragraph{Step 1:} Setting of an approximating eigenvalue problem.

    Here we introduce an approximating eigenvalue problem, that will be used to estimate from below the solutions to \eqref{eq:evolstat-beta}.

    Let $ \delta>0 $, $ \varepsilon>0 $, $ a^\delta(y):=\min(a(y), \sup a - \delta) $ and $ (\lambda^{\delta, \varepsilon}, \varphi^{\delta, \varepsilon}) $ be the principal eigenpair  solving the problem 
    \begin{equation}\label{eq:beta-lowerbound2}
        \begin{system}{ll}
            \varepsilon\Delta\varphi^{\delta, \varepsilon} + \mu (M\star \varphi^{\delta, \varepsilon}-\varphi^{\delta, \varepsilon}) + (a^\delta(y)+\lambda^{\delta, \varepsilon})\varphi^{\delta, \varepsilon}=0 & \text{ in } \Omega \\
            \frac{\partial \varphi^{\delta, \varepsilon}}{\partial \nu}=0 & \text{ on } \partial\Omega,
        \end{system}
    \end{equation}
    with $ \int_\Omega \varphi^{\delta, \varepsilon}(y)dy = 1$. It follows from Theorem \ref{thm:eigen} that $ \lambda^{\delta, \varepsilon} $ converges to the principal eigenvalue $ \lambda^{\delta, 0} $ of the operator $ \psi\mapsto \mu(M\star \psi-\psi) + a^\delta(y)\psi $ when $ \varepsilon\to 0 $. $ \lambda^{\delta, 0} $, in turn, converges to $ \lambda_1 $ when $ \delta\to 0$ by Lipschitz continuity \cite[Proposition 1.1]{Cov-10}. Thus we may approximate $ \lambda_1 $ by $ \lambda^{\delta, \varepsilon} $ for $ \delta>0 $ and $ \varepsilon>0 $ small enough.

    Since $ y\mapsto \frac{1}{\sup a^\delta - a^\delta(y)}\not\in L^1(\Omega) $, it follows from \cite[Theorem 1.1]{Cov-10} (which can be adapted in our context; see \cite{Cov-13}) that there exists a continuous eigenfunction associated with $ \lambda^{\delta, 0} $. In this case \cite[Theorem 1.1]{Cov-13} shows the strict upper bound $ \lambda^{\delta, 0} < -\sup a^\delta+\mu = -\sup a + \delta+\mu $. 

    In what follows we fix the real number  $ \delta>0 $ small enough  so that the inequality 
    $\delta < \frac{1}{2}\min\left(\mu,\sup a-\inf a, \sup a-\underset{\partial\Omega}{\sup}~ a^+-\mu\right) $ holds, together with 
    $ \lambda^{\delta, 0} \leq \frac{3\lambda_1}{4}  $. We define $ \eta:=-\lambda^{\delta, 0} - \sup a + \delta + \mu >0 $. 
    Since $ \lambda^{\delta, \varepsilon}\to\lambda^{\delta, 0} $ as $ \varepsilon\to 0 $, we fix $ \varepsilon_0>0 $ such that for any $ 0<\varepsilon<\varepsilon_0 $,
    $ |\lambda^{\delta, \varepsilon} - \lambda^{\delta, 0}|\leq \frac{-\lambda_1}{4} $, and 
    $ \lambda^{\delta, \varepsilon}\leq \lambda^{\delta, 0} + \frac{\eta}{2}$.

    Finally, integrating equation \eqref{eq:beta-lowerbound2}, we have
    $0 = \int_\Omega (a^\delta(y) + \lambda^{\delta, \varepsilon})\varphi^{\delta, \varepsilon}(y) dy ,$
    thus the function $ a^\delta(y) +\lambda^{\delta, \varepsilon} $ takes nonpositive and nonnegative values. This shows
    \begin{equation*}
        \inf a = \inf a^{\delta}\leq -\lambda^{\delta, \varepsilon}\leq \sup a^\delta=\sup a-\delta.
    \end{equation*}

    \textbf{Step 2:} Estimates from above and from below of $ \varphi^{\delta, \varepsilon} $.

	Let us establish some upper and lower bounds for $ \varphi^{\delta, \varepsilon} $. Since $ \varphi^{\delta, \varepsilon} $ is continuous on $ \overline\Omega$,  there exists $ y_0\in\overline\Omega $ such that $ \varphi^{\delta, \varepsilon}(y_0)=\inf_{z\in\overline\Omega}\varphi^{\delta, \varepsilon}(z) $. If $ y_0\in\partial\Omega$, then {it follows from Hopf's Lemma that} $ \frac{\partial\varphi^{\delta, \varepsilon}}{\partial\nu}(y_0) < 0 $, which contradicts the Neumann boundary conditions satisfied by $ \varphi^{\delta, \varepsilon} $ (recall that $ a(y_0)+\lambda^{\delta, \varepsilon}< 0 $ for $ y_0\in\partial\Omega $). We conclude that $ y_0\in\Omega$. Thus we can evaluate equation \eqref{eq:beta-lowerbound2}:
    \begin{align*}
        0\geq -\varepsilon\Delta\varphi^{\delta, \varepsilon}(y_0) &= \mu\left(M\star \varphi^{\delta, \varepsilon}-\varphi^{\delta, \varepsilon}\right)+\big(a^\delta(y_0)+\lambda^{\delta, \varepsilon}\big)\varphi^{\delta, \varepsilon}, \\
        (\sup a - \inf a+\mu)\varphi^{\delta, \varepsilon}(y_0)&\geq \big(-\lambda^{\delta, \varepsilon} - a^\delta(y_0)+\mu\big)\varphi^{\delta, \varepsilon}(y_0)\geq \mu m_0\int_\Omega \varphi^{\delta, \varepsilon}, \\
        \min_{z\in\Omega}\varphi^{\delta, \varepsilon}(z)&\geq \frac{\mu m_0}{\sup a-\inf a+\mu}.
    \end{align*}

    Similarly, there exists $ y_0\in\Omega $ such that $ \varphi^{\delta, \varepsilon}(y_0)=\max_{z\in\Omega}\varphi^{\delta, \varepsilon}(z) $. Evaluating equation \eqref{eq:beta-lowerbound2}, we get (recalling $ a^\delta(y_0)-\mu + \lambda^{\delta, \varepsilon} \leq -\frac{\eta}{2}<0 $):
    \begin{align*}
        0\leq-\varepsilon \Delta \varphi^{\delta, \varepsilon} &= \mu M\star \varphi^{\delta, \varepsilon} + \big(a^\delta(y_0)-\mu+\lambda^{\delta, \varepsilon}\big)\varphi^{\delta, \varepsilon}(y_0), \\
        \frac{\eta}{2}\varphi^{\delta, \varepsilon}&\leq \mu m_\infty\int_\Omega \varphi^{\delta, \varepsilon}=\mu m_\infty, \\
        \max_{z\in\Omega}\varphi^{\delta, \varepsilon}(z)&\leq 2\frac{\mu m_\infty}{\eta}.
    \end{align*}
    Hence, for $ 0< \varepsilon \leq \varepsilon_0 $ and $ y_0\in\Omega $, we have shown that 
    \begin{equation*}
        \frac{\mu m_0}{\sup a-\inf a+\mu}\leq \varphi^{\delta, \varepsilon}(y_0)\leq 2\frac{\mu m_\infty}{\eta}.
    \end{equation*}

    \textbf{Step 3:} Lower estimate for $ p^{\varepsilon, \beta} $.

    We are now in a position to derive a lower bound for $ p^{\varepsilon, \beta} $.  Since $ p^{\varepsilon, \beta}>0 $ in $ \overline\Omega $, we can define 
    $\alpha:=\sup\left\{\zeta>0\,|\, \forall y\in\Omega, \zeta\varphi^{\delta, \varepsilon}(y)\leq p^{\varepsilon, \beta}(y)\right\}$.

    Assume by contradiction that 
    $ \alpha< \alpha_0:=\min\left(\frac{m_0\eta}{2 k_\infty m_\infty}, \frac{-\lambda^{\delta, \varepsilon}\eta}{2\beta\mu m_\infty+\eta k_\infty}\right) $. 
    By definition of $ \alpha $ there exists $ y_0\in\overline\Omega$ such that $ \alpha\varphi^{\delta, \varepsilon}(y_0)=p(y_0)$. Assume $ y_0\in\partial\Omega $, then {it follows from Hopf's Lemma that} $ \frac{\partial(p^{\varepsilon, \beta}-\alpha \varphi^{\delta, \varepsilon})}{\partial\nu}(y_0)<0 $, which contradicts the Neumann boundary conditions satisfied by $ p^{\varepsilon, \beta} $ and $ \varphi^{\delta, \varepsilon} $. Thus $ y_0\in\Omega$. We have:
    \begin{align*}
        0&\geq -\varepsilon\Delta(p^{\varepsilon, \beta}-\alpha\varphi^{\delta, \varepsilon})(y_0)=\mu \left(M\star (p^{\varepsilon, \beta}-\alpha\varphi^{\delta, \varepsilon}) - (p^{\varepsilon, \beta}-\varphi^{\delta, \varepsilon})\right) \\ 
         &\quad + p^{\varepsilon, \beta}\big(a(y_0)-K\star p^{\varepsilon, \beta}-\beta p^{\varepsilon, \beta}\big)-\big(\lambda^{\delta, \varepsilon}+a^\delta(y_0)\big)\alpha\varphi^{\delta, \varepsilon} \\
         &=\int_\Omega \big(\mu M(y_0,z)-\alpha \varphi^{\delta, \varepsilon}(y_0)K(y_0, z)\big)\big(p^{\varepsilon, \beta}(z)-\alpha\varphi^{\delta, \varepsilon}(z)\big)dz\\
         &\quad-\alpha\varphi^{\delta, \varepsilon}(y_0)\int_\Omega K(y_0,z)\big(\alpha\varphi^{\delta, \varepsilon}(z)\big)dz\\
         &\quad+\alpha\varphi^{\delta, \varepsilon}\big(a(y_0)-a^\delta(y_0)\big)-\beta \left(p^{\varepsilon,\beta}\right)^2 -\lambda^{\delta, \varepsilon}\alpha\varphi^{\delta, \varepsilon}(y_0).
    \end{align*}
    By definition,
    $ \mu M(y_0,z)-\alpha \varphi^{\delta, \varepsilon}(y_0)K(y_0,z)\geq \mu m_0-k_\infty\frac{m_0\eta}{2 k_\infty m_\infty}\frac{2\mu m_\infty}{\eta}= 0$ for any $ z\in\Omega $,
    and thus, recalling $ a(y_0)\geq a^{\delta}(y_0) $,
    \begin{equation*}
        0\geq -\alpha\lambda^{\delta, \varepsilon}\varphi^{\delta, \varepsilon}(y_0)-\alpha^2\varphi^{\delta, \varepsilon}(y_0)\left(\beta\varphi^{\delta, \varepsilon} + \int_\Omega K(y_0,z)\varphi^{\delta, \varepsilon}(z)dz\right),
    \end{equation*}
    \begin{equation*}
        \left(2\beta\frac{\mu m_\infty}{\eta} + k_\infty\right)\alpha\geq \alpha\left(\beta \varphi^{\delta, \varepsilon}(y_0) + \int_\Omega K(y_0,z)\varphi^{\delta, \varepsilon}(z)dz\right)\geq -\lambda^{\delta, \varepsilon}, 
    \end{equation*}
    which is a contradiction since $ \alpha < \alpha_0= \min\left(\frac{m_0\eta}{2 k_\infty m_\infty}, \frac{-\lambda^{\delta, \varepsilon}\eta}{2\beta\mu m_\infty+\eta k_\infty}\right)$.

    We conclude that $ \alpha\geq \alpha_0 $ and thus (recalling $ \lambda^{\delta, \varepsilon}\leq \frac{\lambda_1}{2}$) 
    \begin{align*}
        \min_{y\in\Omega}p^{\varepsilon, \beta}(y)&\geq \alpha_0 \min_{y\in\Omega}\varphi^{\delta, \varepsilon}(y) \\
                                                  &\geq\min\left(\frac{m_0\eta}{2 k_\infty m_\infty}, \frac{(-\lambda_1)\eta}{4\beta\mu m_\infty+2\eta k_\infty}\right)\frac{\mu m_0}{\sup a-\inf a+\mu}>0.
    \end{align*}
    Since this lower bound is independent from $ \varepsilon $, this ends the proof of Lemma \ref{lem:beta-lowerbound}.
\end{proof}

\subsection{Construction of a stationary solution at $ \varepsilon=0$}\label{sec:evolstat}

In this section we assume $ \lambda_1 < 0 $. Then, Theorem \ref{thm:ex-stat-eps} guarantees the existence of a positive solution to \eqref{eq:evolstat-beta} for $ \varepsilon $ small enough, since $ \lambda_1^\varepsilon\to\lambda_1 $ as $ \varepsilon \to 0 $ (recall Theorem \ref{thm:eigen}). In this context, we expect the solution constructed in Theorem \ref{thm:ex-stat-eps} with $ \beta=0 $ to converge weakly to a (possibly singular) Radon measure, solution to \eqref{eq:stat}. Here we prove this result, and complete the proof of Theorem \ref{thm:survival}. In particular, in this section we assume $ \beta=0 $.

Before we can prove Theorem \ref{thm:survival}, we need a series of estimates on the previously constructed solutions $p^\varepsilon:=p^{\varepsilon, 0} $.
\begin{lem}[Estimates on the mass]\label{lem:lowupintpeps}
    Let Assumption \ref{hyp:gen} hold, let $ \varepsilon >0 $,  $ \lambda_1^\varepsilon < 0 $, and   $p^\varepsilon$ be a solution to equation \eqref{eq:evolstat-beta} with $ \beta=0 $. Then 
    \begin{equation}\label{eq:lowupinteps}
        \frac{-\lambda_1^\varepsilon}{k_\infty}\leq \int_\Omega p^\varepsilon(y)dy\leq \frac{\sup a}{k_0}.
    \end{equation}
\end{lem}
\begin{proof}
    The upper bound in equation \eqref{eq:lowupinteps} has been established in Lemma \ref{lem:apriorip}. We turn our attention to the lower estimate.

    We assume by contradiction that $ \lambda_1^\varepsilon+ k_\infty\int_\Omega p^\varepsilon(y)dy <0 $.  Let $ (\lambda_1^\varepsilon, \varphi^\varepsilon >0) $ be the solution to the eigenproblem  \eqref{eq:eigeneps}, normalized with $ \int \varphi^\varepsilon=1 $.
    Then, we define the  real number $\alpha:=\sup\{\zeta>0\,|\, \forall y\in\Omega, \zeta\varphi^\varepsilon\leq p^\varepsilon\}>0$, which is then well-defined since $ p^\varepsilon >0 $ and $ \varphi^\varepsilon $ is bounded.

	By definition of $ \alpha $ we have $ \alpha \varphi^\varepsilon \leq p^\varepsilon $ and there exists a point $ y_0\in\overline\Omega $ such that $p^\varepsilon(y_0)=\alpha\varphi^\varepsilon(y_0) $. If $ y_0\in\partial\Omega $, since $ y_0 $ is a maximum point for the function $ \alpha\varphi^\varepsilon -p^\varepsilon $, then {it follows from Hopf's Lemma that} $ \frac{\partial  \alpha \varphi^\varepsilon-p^\varepsilon}{\partial \nu}(y_0)>0 $, which contradicts the Neumann boundary conditions satisfied by $ p^\varepsilon $ and $\varphi^\varepsilon $. Thus $ y_0\in\Omega $ and we compute 
    \begin{align*}
        0&\leq -\mu \big(M\star (\alpha\varphi^\varepsilon-p^\varepsilon)-(\alpha\varphi^\varepsilon-p^\varepsilon)\big)-\varepsilon\Delta(\alpha\varphi^\varepsilon-p^\varepsilon)  \\
         &=\lambda^\varepsilon_1\alpha\varphi^\varepsilon+(K\star p^\varepsilon)p^\varepsilon+a(y_0)(\alpha\varphi^\varepsilon-p^\varepsilon)=  \big(\lambda_1^\varepsilon + (K\star p^\varepsilon)(y_0)\big)p^\varepsilon(y_0), 
    \end{align*}
    which is a contradiction since $ \lambda_1+(K\star p)(y_0)\leq \lambda_1+k_\infty\int_\Omega p^\varepsilon(y)dy < 0 $.
    This finishes the proof of Lemma \ref{lem:lowupintpeps}.
\end{proof}

\begin{proof}[Proof of Theorem \ref{thm:survival}]
	{It follows from Lemma \ref{lem:lowupintpeps} that} the family $ (p^{\varepsilon})_{0<\varepsilon\leq 1} $ of solutions to \eqref{eq:evolstat-beta} with $ \varepsilon >0 $ and $ \beta=0 $ is uniformly bounded in $ M^1(\overline\Omega) $. Hence, {applying} Prokhorov's Theorem \cite[Theorem 8.6.2]{Bog-07}, $(p^\varepsilon)_{0<\varepsilon<1} $   is precompact for the weak topology in $ M^1(\overline\Omega)$, and  there exists a sequence $ p^{\varepsilon_n} $ (with $ \varepsilon_n\to 0 $) and a measure  $p$ such that $ p^{\varepsilon_n}\rightharpoonup p $ in the sense of measures.    In particular, taking $ \psi=1 $, we recover the estimate of Lemma \ref{lem:lowupintpeps}:
    $ 0<\frac{-\lambda_1}{k_\infty}\leq \int_{\overline\Omega}p(dy)\leq \frac{\sup a}{k_0}$.
    Hence $ p $ is non-trivial.

    Let us show that $ p $ is indeed a solution to \eqref{eq:stat}. {Multiplying equation \eqref{eq:evolstat-beta} by  $ \psi\in F_0 $, where $F_0$ is the set of functions with zero boundary flux as defined in \eqref{eq:F0}}, and integrating by parts, we get
    \begin{align}\label{eqthm:survival}
        -\varepsilon_n\int_\Omega  p^{\varepsilon_n}\Delta \psi dy &= \int_\Omega \mu(M\star  p^{\varepsilon_n}-p^{\varepsilon_n})\psi+a(y) p^{\varepsilon_n}\psi dy \\
								   &\quad{-\int_{\Omega}(K\star p^{\varepsilon_n})(y) \psi(y)p^{\varepsilon_n}(y) dy.}\nonumber
    \end{align}
    Since $ \Delta \psi\in C_b(\overline\Omega) $ and $ \int_\Omega p^{\varepsilon_n} $ is bounded uniformly in $ n $, the left-hand side of \eqref{eqthm:survival} goes to $ 0 $ when $ n\to\infty $. Moreover since $ \psi(y)(a(y)-\mu)\in C(\overline\Omega) $, then by definition the convergence
    $\int_{\Omega} \psi(y)(a(y)-\mu)p^{\varepsilon_n}(y)dy\to_{n\to\infty}\int_{\overline\Omega}\psi(y)(a(y)-\mu)p(dy)$ holds.

	We turn our attention to the term $ \int_\Omega M\star p^{\varepsilon_n}(y)\psi(y)dy $. {We notice that
	\begin{equation*}
		\int_\Omega\int_\Omega M(y, z)p^{\varepsilon_n}(z)dz \psi(y)dy=\int_\Omega {p^{\varepsilon_n}}(z)\int_\Omega M(y,z)\psi(y)dy=\int_\Omega {p^{\varepsilon_n}}(z)\check{M}\star\psi(z)dz, 
	\end{equation*}
	where $\check M(y,z)=M(z, y)$. Since $\check M\star\psi(z) $ is a valid test function, we have indeed $\int_\Omega M\star p^{\varepsilon_n}(y)\psi(y)dy\to \int_\Omega M\star p(y) \psi(y)dy$.

	We turn to the convergence of the nonlinearity $\int_{\Omega}(K\star p^{\varepsilon_n})(y) \psi(y)p^{\varepsilon_n}(y) dy$. Since the sequence $p^{\varepsilon_n} $ appears twice in this term, the above argument cannot be used directly. Therefore, we first show a stronger convergence for $K\star p^{\varepsilon_n}$, namely that it converges uniformly to a continuous limit. We notice that
	
    \begin{align*}
        \left|(K\star p^{\varepsilon_n})(y)-(K\star p^{\varepsilon_n})(y')\right| =& \left|\int_\Omega \big(K(y,z)-K(y', z)\big)p^{\varepsilon_n}(z)dz\right|\\
        \leq & 2\Vert K(y, \cdot)-K(y', \cdot)\Vert_{C_b(\overline\Omega)}\frac{\sup a}{k_0}.
    \end{align*}
	Thus, the modulus of continuity of $ K\star p^{\varepsilon_n} $ is uniformly bounded. Up to the extraction of a subsequence, $ K\star p^{\varepsilon_n} $ converges in $ C_b(\overline\Omega) $ to a limit which we identify as $ K\star p $ (by using another test function and the weak convergence $p^{\varepsilon_n}\rightharpoonup p$). Along this subsequence, we have then 
	\begin{equation*}
        \lim_{n\to\infty}\int_\Omega (K\star p^{\varepsilon_n})(y)\psi(y)p^{\varepsilon_n}(y)dy = \int_\Omega (K\star p)\psi(y)p(dy).
    \end{equation*}}

	We have shown that equation \eqref{eqthm:survival} is satisfied for any $ \psi\in F_0 $. {Applying} Lemma \ref{lem:zerofluxdense}, $ F_0 $ is densely embedded in $ C_b(\overline\Omega) $. Equation \eqref{eqthm:survival} is thus satisfied for any $ \psi\in C_b(\overline\Omega) $. This ends the proof of Theorem \ref{thm:survival}.
\end{proof}

\subsection{Proof of Theorem \ref{thm:concentration}}

Under Assumption \ref{hyp:dom} and if $ \mu $ is small enough according to Theorem \ref{thm:smallmu}, we can actually prove that the measure solution to \eqref{eq:stat} is concentrated in $ \Omega_0 $ (Theorem \ref{thm:concentration}). To do so, we write a solution $ p $ to \eqref{eq:stat} as an eigenvector for a problem similar to \eqref{eq:eigen}, and make use of Theorem \ref{thm:eigenpair}. 
\begin{proof}[Proof of Theorem \ref{thm:concentration}]
    Let $ p\in M^1(\overline\Omega) $, $ p\not\equiv 0 $ be a nonnegative  solution to \eqref{eq:stat}. Define $ b(y):=a(y)-(K\star p)(y) $. Then $ b $ is a continuous function and we have
    \begin{equation*}
        \forall y\in\overline\Omega, b(y)=a(y)-\int_{\overline\Omega} K(y,z)p(dz)\leq a(0)-\int_{\overline\Omega}K(0, z)p(dz) = b(0),
    \end{equation*}
	{as a result of} Assumption \ref{hyp:dom}. This shows that $ \sup b = b(0) $. Next we compute:
    \begin{equation*}
        b(0)-b(y) = a(0)-a(y) + \int_{\overline\Omega}\big(K(y,z)-K(0,z)\big) p(dz) \geq a(0)-a(y).
    \end{equation*}
    Thus, $ b $ satisfies Assumption \ref{hyp:dega}.

    We remark  that $ p $ solves 
    \begin{equation} \label{eqthm:concentration}
        \mu(M\star p-p) + b(y) p = 0
    \end{equation}
    in the sense of measures. Thus $ p$ is a solution to \eqref{eq:eigen} with $ a $ replaced by $ b $. Applying Corollary \ref{cor:compmucrit}, since $ \mu < \mu_0(\Omega, M, \sup a-a)$, then $ \mu<\mu_0(\Omega, M, \sup b - b) $ and thus the only solutions to  \eqref{eqthm:concentration} are singular measures which singular part is concentrated in $ \{y\in\Omega\,|\, b(y)=\sup b\} $. Let us show that $ \{y\,|\, b(y)=\sup b\}\subset\Omega_0 $. Let $ y\in\overline\Omega $ such that  $ y\not\in\Omega_0 $. Then $ a(y)<a(0) $ and
    \begin{equation*}
        b(y)=a(y)-(K\star p)(y) < a(0)-(K\star p)(y)\leq a(0)-(K\star p)(0)= \sup b,
    \end{equation*}
    which shows that $ y\not\in \{y\,|\, b(y)=\sup b\} $. This ends the proof of Theorem \ref{thm:concentration}.
\end{proof}

\section{Construction of traveling waves}\label{sec:TW}
In this section, we prove our main result Theorem \ref{thm:TW}. To construct the desired measure traveling wave, we first consider a regularized problem in a box $ -l\leq x\leq l $, $ y\in\Omega$.

\subsection{Construction of a solution in a box}

Here we aim at constructing solutions $(c, u=u(x,y))$ to 
\begin{equation}\label{pb:box-beta}
    \begin{system}{ll}
        -\varepsilon\Delta_yu-u_{xx}-cu_x= \mu(M\star u-u) \hspace{2cm} \\
        \hfill+ u(a(y)-K\star u-\beta u)  & \text{ in } (-l, l)\times\Omega\\
        \nabla_y u(x,y)\cdot\nu=0 & \text{ on } (-l,l)\times\partial\Omega  \\
        u(l,y)=0 & \text{ in } \Omega  \\
        u(-l,y)=p(y) & \text{ in } \Omega, 
    \end{system}
\end{equation}
for $ \beta\geq 0 $ and $ p$ solving \eqref{eq:evolstat-beta}. Notice that any solution to \eqref{pb:box-beta} with $ \beta>0 $ is a subsolution to \eqref{pb:box-beta} with $ \beta=0 $. In particular, we will use some solutions to \eqref{pb:box-beta} with $ \beta>0 $ to get lower estimates on solutions to \eqref{pb:box-beta} with $ \beta=0 $.

In contrast with \cite{Ber-Nic-Sch-1985,Ber-Nad-Per-Ryz-09,Alf-Cov-Rao-13, Gri-Rao-16}, we use a global continuation theorem (in the proof of Theorem \ref{thm:solbox} below) instead of a topological degree to construct solutions to the local problem \eqref{pb:box-beta}. Though both arguments have the same topological basis, we believe that this is an improvement of the usual method, since it spares the need to explicitly compute the topological degree associated with \eqref{pb:box-beta}.

Let us also introduce the following quantity, which is the minimal speed for traveling waves (as we will show later):
\begin{equation}\label{def:minimal-speed}
    c^*_\varepsilon:=2\sqrt{-\lambda_1^\varepsilon}.
\end{equation}

Our result is the following:
\begin{thm}[Existence of solutions in the box]\label{thm:solbox}
    Let Assumption \ref{hyp:gen} hold,  $ \varepsilon>0$ be such that $ \lambda_1^\varepsilon<0$, and  $ \beta\geq 0 $. 
    Then, there exists a nonnegative solution to \eqref{pb:box-beta}. 
    Moreover, let $ l_0:=\frac{\pi}{\sqrt{-\lambda_1^\varepsilon}}>0 $,  $ \tau_0:=\frac{-\lambda_1^\varepsilon}{2}>0 $. Then,  for any  $0 <\tau<\tau_0 $, there exists $ \bar l(\tau)\geq  l_0 + 1 $ such that if  $l>\bar l(\tau) $, there exists a nonnegative solution $(c, u) $  to \eqref{pb:box-beta} with $ 0< c\leq c^*_\varepsilon $, which also satisfies the normalization condition
    \begin{equation}\label{eq:norm}
        \sup_{(x, y)\in(-l_0, l_0)\times \Omega}\left(\int_\Omega K(y,z)u(x,z)\mathrm dz+\beta u(x,y)\right) = \tau. 
    \end{equation}
\end{thm}

Before we prove Theorem \ref{thm:solbox}, we need to establish some \textit{a priori} estimates on the solutions to \eqref{pb:box-beta}.
For technical reasons, we actually study the solutions to 
\begin{equation}\label{pb:box-beta-pos}
    \begin{system}{ll}
        -\varepsilon\Delta_yu-u_{xx}-cu_x =\sigma \big( \mu (M\star u-u) \hspace{1.5cm} \\
        \hfill+ u\chi_{u\geq 0}(a(y)-K\star u-\beta u)\big) & \text{ in } (-l,l)\times\Omega \\
        \nabla_y u(x,y)\cdot\nu=0 & \text{ on } (-l,l)\times\partial\Omega \\
        u(l,y)=0 & \text{ in } \Omega \\
        u(-l,y)=p(y) & \text{ in } \Omega,
    \end{system}
\end{equation}
where 
$ \chi_{u\geq 0}=0 \text{ if } u\leq 0 $, $ \chi_{u\geq 0}= 1  \text{ if } u>0 $,  and $ \sigma\in (0, 1]$. {We introduce the positive-part cutoff involving $\chi$ in Problem \eqref{pb:box-beta-pos} in order to ensure that the nontrivial solutions to this problem are positive.}

\begin{lem}[A priori estimates on the solutions to \eqref{pb:box-beta-pos}]\label{lem:aprioribox}
    Let Assumption \ref{hyp:gen} hold, $ \varepsilon>0$ such that $ \lambda_1^\varepsilon<0 $, $ \beta \geq 0 $, and $ |c|\leq c^*_\varepsilon $. We define $ l_0:=\frac{\pi}{\sqrt{-\lambda_1^\varepsilon} }$. Let $ u $ be a solution to \eqref{pb:box-beta-pos}, then
    \begin{enumerate}[label=(\textit{\roman*}), ref=(\textit{\roman*})]
        \item \label{item:aprioribox-regular} 
		$ u \in C^2_{loc}\big((-l, l)\times\Omega\big)\cap C^1_{loc}\big((-l,l)\times\overline\Omega\big)\cap C_b\big([-l,l]\times\overline\Omega\big) $.
        \item \label{item:aprioribox-positive} 
		$ u $ is positive in $ (-l, l)\times \overline\Omega $.
        \item \label{item:aprioribox-boundint} For any $ x\in [-l, l] $, we have $ \int_\Omega u(x, y)dy\leq \frac{\sup a}{k_0} $.
        \item \label{item:aprioribox-boundinfty} 
		There exists a positive constant {$ C^\varepsilon $}, independent from $ c $, $ l $ and $ \sigma $, such that we have 
            $\Vert u\Vert_{C_b((-l, l)\times \Omega)}\leq C^\varepsilon$. If $ \beta>0$, then we have the estimate $ \Vert u\Vert_{C_b((-l, l)\times\Omega)}\leq \frac{\sup a}{\beta} $.
        \item \label{item:aprioribox-minspeed} If $ \sigma=1 $, $ c= 0 $,  and $ l> l_0 $, then
            \begin{equation*}
                \sup_{(x, y)\in(-l_0, l_0)\times \Omega}\left(\int_\Omega K(y,z)u(x,z)\mathrm dz+\beta u(x,y)\right) > \frac{-\lambda_1^\varepsilon}{2}. 
            \end{equation*}
            Remark that for this result to hold, $ u $ needs only be defined on $ (-l_0, l_0)\times\Omega $.
        \item \label{item:aprioribox-maxspeed} If $ \sigma=1 $ and $ c=c_\varepsilon^* $, then there exists a constant $ A $ (independent from $ l $) and $ \lambda:=\frac{c_\varepsilon^*}{2}>0 $ such that
            $\forall (x, y)\in(-l,l)\times\Omega,  u\leq A e^{-\lambda (x+l)}$.

            In particular for any 
            $ l\geq \bar l(\tau):= \frac{1}{\lambda}\ln\left(\frac{\tau}{2A(k_\infty\int_\Omega\varphi^\varepsilon + \beta\sup_\Omega\varphi^\varepsilon)}\right)-l_0 $ and  $ 0<\tau\leq \tau_0= \frac{-\lambda_1^\varepsilon}{2} $, we have
            \begin{equation*}
                \sup_{(x, y)\in(-l_0, l_0)\times \Omega}\left(\int_\Omega K(y,z)u(x,z)\mathrm dz+\beta u(x,y)\right) < \tau .
            \end{equation*}
    \end{enumerate}
\end{lem}
\begin{proof}
    Item \ref{item:aprioribox-regular} holds by a direct application of  \cite[Lemma 7.1]{Ber-Nir-91}, and item \ref{item:aprioribox-positive} by a classical comparison argument. Let us resume to the remaining items.

	\textbf{Item \ref{item:aprioribox-boundint}:} {By the estimate in} Lemma \ref{lem:lowupintpeps}, we have $ \int_\Omega p(y)dy\leq\frac{\sup a}{k_0} $. Assume that the function $ x\mapsto \int_\Omega u(x,y)dy $ has a maximal value at $ x_0\in (-l, l) $, then integrating \eqref{pb:box-beta-pos} over $ \Omega $ we have
    \begin{align*}
        0&\leq -\frac{d^2 \int_\Omega u(x_0,y)dy}{dx^2}-c\frac{d \int_\Omega u(x_0,y)dy}{dx}\\
	    &=\sigma \int_\Omega a(y)u(x_0,y) - (K\star u)(x_0,y)u(x_0, y)dy,
    \end{align*}
	{and thus:}
	\begin{align*}
        k_0\left(\int_\Omega u(x_0,y)dy\right)^2&\leq \int_\Omega\int_{\Omega} K(y,z)u(x_0,z)u(x_0,y)dydz \\
                                                &= \int_\Omega a(y)u(x_0,y)dy\leq \sup a\int_\Omega u(x_0,y)dy.
    \end{align*}

    This shows item \ref{item:aprioribox-boundint}.

	\textbf{Item \ref{item:aprioribox-boundinfty}:} Assume first $ \beta >0 $ and let $ u(x_0,y_0)=\sup u$ at $ (x_0,y_0)\in[-l,l]\times \overline\Omega $. Assume by contradiction that $ u(x_0,y_0)>\frac{\sup a}{\beta} $. If $x_0=-l $, since  $ p$ satisfies the upper bound $ \sup p\leq \frac{\sup a}{\beta} $ {by the estimate in} Lemma \ref{lem:apriorip} item \ref{item:apriorip-boundinfty}, we have a contradiction. If $ x_0=+l $, since $ u(x_0,y_0)=0 $, we have a contradiction. Assume  $x_0\in (-l, l) $. If $y_0\in\partial\Omega $, then {it follows from Hopf's Lemma that} $ \frac{\partial u}{\partial\nu}(x_0,y_0)>0 $, which is a contradiction. Thus $ y_0\in \Omega $. Now, testing \eqref{pb:box-beta-pos} at $ (x_0,y_0)$, we have
    \begin{align*}
        0&\leq -\varepsilon \Delta_y u(x_0,y_0) -u_{xx}(x_0,y_0)-cu_x(x_0,y_0)-\sigma \mu\big((M\star u)-u\big) (x_0,y_0)\\
         &=\sigma u(x_0,y_0)\big(a(y_0)-(K\star u)(x_0,y_0)-\beta u(x_0,y_0)\big)<0
    \end{align*}
    which is a contradiction. Thus $ u\leq \frac{\sup a}{\beta} $.

    \medskip

    We turn our attention to the case $ \beta=0 $. In this case, 
    we construct a supersolution  as in Lemma \ref{lem:apriorip}.
    Recalling that $ u $ satisfies Dirichlet boundary conditions at $ x=\pm l $, the local maximum principle up to the boundary \cite[Theorem 9.26]{Gil-Tru-01} shows the existence of $ C=C(\Omega, R, \varepsilon, \Vert a\Vert_{L^\infty}, k_0, k_\infty, \mu, c^*_\varepsilon) $ such that  for any  
    $ x\in [-l, l] $, $ y\in\Omega $ with $ d(y, \partial\Omega) \geq R $, we have the estimate
    $\underset{B_{R/2}(x,y)}{\sup}\,u\leq C$.

    To show that this estimate does not degenerate near the boundary, we use the same kind of supersolution as in Lemma \ref{lem:apriorip}. Let    
    \begin{equation*}
        \Omega_R:=\{y\in\Omega\,|\, d(y,\partial\Omega)< R\}
    \end{equation*}
    for any $ R>0$. We select $ R $ small enough so that  $ \Omega_R $ has a $ C^3 $ boundary and the comparison principle \cite[Proposition 1.1]{Ber-Nir-91} holds in the narrow domain $ \Omega_R$. Let us stress that since $ \sigma\in (0,1)$,  $R$ can be chosen uniformly in $ \sigma$.

    This allows us to  construct a positive solution to 
    \begin{equation*}
        \begin{system}{ll}
            -\varepsilon\Delta v-\sigma (a(y)-\mu)v = \mu m_0 \frac{m_\infty \sup a}{k_0} & \text{ in }\Omega \\
            v=C & \text{ on }\partial \Omega_R\bsl\partial\Omega \\
            \frac{\partial v}{\partial\nu}=0 & \text{ on }~\partial\Omega,
        \end{system}
    \end{equation*}
    which is bounded uniformly in $ \sigma $, as we did in the proof of Lemma \ref{lem:apriorip}. Now we  select
    $\alpha:=\inf\{\zeta>0\,|\,\forall x\in (-l, l), \forall y\in\Omega, \zeta v(y)\geq u(x,y)\}$.
    Assume by contradiction that $ \alpha>1 $. Then there exists $ (x_0, y_0)\in [-l,l]\times\overline\Omega $ such that $ u(x_0,y_0)=\alpha v(y_0) $. If $x_0=l $ then $ u=0 $, which is a contradiction. If $ x_0=-l $, since $ u(-l, y_0)=p(y_0) $ solves \eqref{eq:evolstat-beta}, we argue as in Lemma \ref{lem:apriorip} and get a contradiction. We are left to investigate the case $ x_0\in (-l, l) $.  If $ y_0\in\partial\Omega $, since $ (x_0, y_0) $ is a minimum to the function $ \alpha v - u $, then by Hopf's Lemma we have $ \frac{\partial (\alpha v-u)}{\partial\nu}(x_0,y_0)<0 $ which is a contradiction. Thus $ y_0\not\in\partial\Omega $. Since $ \alpha>1 $ and $ u\leq C $ on $  \partial \Omega_R\bsl\partial \Omega $, then $ y_0\in\Omega_R $. Now $ (x_0,y_0) $ is a local minimum to $ \alpha v-u $ and thus
    \begin{align*}
        0&\geq -\varepsilon\Delta (\alpha v-u)(x_0,y_0) = \sigma(a(y_0)-\mu)(\alpha v-u)(x_0,y_0)+\alpha \mu m_0 \frac{m_\infty \sup a}{k_0}\\ 
         &\quad - \sigma \mu (M\star  u)(x_0,y_0)    +u(x_0,y_0)(K\star u)(x_0,y_0) \\ 
         &> \alpha \mu m_0 \frac{m_\infty \sup a}{k_0}-\sigma\mu (M\star  u)(x_0,y_0)\geq 0
    \end{align*}
    which is a contradiction. Thus $ \alpha\leq 1 $. 

    This shows that $ u(x,y)\leq v(y) $ in $ (-l, l)\times\Omega_R $. Since $ v $ is bounded uniformly in $ \sigma $, we have our uniform bound for $ u $ in $ [-l,l]\times\Omega_R$. In the rest of the domain $(-l, l)\times \Omega\bsl\Omega_R $,  we have $ u\leq C $.

	This proves item \ref{item:aprioribox-boundinfty}, {with $C^\varepsilon:=\max(\sup_{y\in\Omega_R} v(y), C)$.}

    \textbf{Item \ref{item:aprioribox-minspeed}:} This proof is similar to the one in \cite{Alf-Cov-Rao-13}. Assume by contradiction that 
    $\sup_{(x, y)\in(-l_0, l_0)\times \Omega}\left(\int_\Omega K(y,z)u(x,z)\mathrm dz+\beta u(x,y)\right) \leq \tau_0$. 
    Then $ u $ satisfies: 
    \begin{equation}\label{eq:cnot0}
        -u_{xx}-\varepsilon \Delta_y u-\mu(M\star u-u) - a(y)u\geq -\tau_0u \qquad\text{ in } (-l_0, l_0).
    \end{equation}

    Define $ \psi(x,y):=\cos\left(\frac{\pi}{2l_0}x\right)\varphi^{\varepsilon}(y) $, where $ \varphi^\varepsilon $ is the principal eigenfunction solution to \eqref{eq:eigeneps} satisfying $ \sup_{y\in\Omega}\varphi^\varepsilon = 1 $. Since $ u $ is positive in $ [-l_0, l_0]\times \overline\Omega $, we can define the real number
    $\alpha:=\sup\{\zeta>0\,|\,\forall (x,y)\in (-l_0, l_0)\times \Omega,  \zeta\psi(x,y)\leq u(x,y) \}$, and we have $ \alpha>0$.

    Then, there exists $ (x_0,y_0)\in [-l_0, l_0]\times \overline\Omega $ such that $ \alpha\psi(x_0,y_0)=u(x_0,y_0)$. {Because of} the boundary conditions satisfied by $ u $ and $ \psi $,  $ (x_0,y_0) $ has to be in $ (-l_0, l_0)\times \Omega $. Since $ (x_0, y_0) $ is the minimum of $ u-\alpha\psi $, we have
    \begin{align*}
        0&\geq -\varepsilon \Delta_y(u-\alpha\psi)(x_0,y_0) -(u-\alpha\psi)_{xx}(x_0,y_0)               \\
         &\quad -\mu\big(M\star(u-\alpha\psi)-(u-\alpha\psi)\big)(x_0,y_0)   -a(y_0)(u-\alpha\psi)(x_0,y_0)    \\  
         &\geq-\tau_0 u(x_0,y_0) +\alpha\left(- \lambda_1^\varepsilon-\left(\frac{\pi}{2l_0}\right)^2\right)\psi(x_0,y_0)         \\
         &=\left(\frac{-3\lambda_1^\varepsilon}{4}-\tau_0\right)u(x_0,y_0)>0,
    \end{align*}
    since $ -\tau_0= \frac{\lambda_1^\varepsilon}{2}  $. This is a contradiction.

    This proves item \ref{item:aprioribox-minspeed}.

    \textbf{Item \ref{item:aprioribox-maxspeed}:} Let $ \psi(x, y):=e^{-\frac{c^*_\varepsilon}{2} x} \varphi^\varepsilon(y) $ with $(\lambda_1^\varepsilon, \varphi^\varepsilon) $ solution to \eqref{eq:eigeneps}. Then $ \psi $ satisfies:
    \begin{equation*}
        -c_\varepsilon^*\psi_x-\psi_{xx}-\varepsilon\Delta_y\psi - \mu( M\star \psi-\psi) = a(y)\psi.
    \end{equation*}
    Since $ \psi >0$ on $ [-l, l]\times \overline\Omega $, there exists $ \zeta>0 $ such that $ \zeta\psi\geq u $ on $ (-l, l)\times\Omega $. Let us select 
    $\alpha:=\inf\{\zeta>0\,|\,\forall (x,y)\in (-l_0, l_0)\times \Omega,  \zeta\psi(x,y)\geq u(x,y) \}$.
    By definition of $ \alpha $ we have $ \alpha\psi\geq u $ and  there exists $ (x_0,y_0)\in [-l, l]\times \overline\Omega $ such that $ \alpha\psi(x_0,y_0)=u(x_0,y_0)$. {Because of} the boundary conditions satisfied by $ u $ and $ \psi $, $ (x_0, y_0) $ has to be in $ [-l, l)\times\Omega $. If $ x_0\in (-l,l) $, we have:
    \begin{align*}
        0&\leq -\varepsilon \Delta_y(u-\alpha\psi)(x_0,y_0) -(u-\alpha\psi)_{xx}(x_0,y_0)               \\
         &\quad -\mu\big(M\star(u-\alpha\psi)-(u-\alpha\psi)\big)(x_0,y_0)        -a(y_0)(u-\alpha\psi)(x_0,y_0)    \\
         &<0
    \end{align*}
    which is a contradiction. We conclude that $ x_0=-l_0 $ and thus $ \alpha\leq \frac{\sup_\Omega p}{\inf_\Omega \varphi^\varepsilon}e^{-\frac{c_\varepsilon^*}{2} l} $.
    By definition of $ \alpha $, we can then write:
    \begin{equation*}
        u(x,y)\leq \alpha e^{-\frac{c_\varepsilon^*}{2} x}\varphi^\varepsilon(y)\leq \frac{\sup_\Omega p}{\inf_\Omega \varphi^\varepsilon}e^{-\frac{c^*_\varepsilon}{2}(x+l)}\varphi^\varepsilon(y)
    \end{equation*}
    which concludes the proof of item \ref{item:aprioribox-maxspeed}.
\end{proof}
We are now in the position to prove Theorem \ref{thm:solbox}, {by using} the global continuation principle \cite[Theorem 14 C]{Zei-86}.
\begin{proof}[Proof of Theorem \ref{thm:solbox}]
    For $ c\in\mathbb R $ and $ u\in C_b\big((-l, l)\times\Omega\big)$. We define $ F(c,u)=\tilde u $ where $ \tilde u $ solves:
    \begin{equation}\label{eqthm:solbox-1}
        \begin{system}{ll}
            -\tilde u_{xx}-c\tilde u_x-\varepsilon\Delta_y\tilde u = \mu M\star u  \hspace{2cm} \\
            \hfill+u\chi_{u\geq 0}(a(y)-\mu-K\star u-\beta u) & \text{ in } (-l,l)\times\Omega  \\
            \frac{\partial\tilde u}{\partial\nu}(x,y)=0 & \text{ on } (-l,l)\times\partial \Omega \\
            \tilde u(l,y)=0  &\text{ in } \Omega \\
            \tilde u(-l,y)=p(y).&\text{ in } \Omega.
        \end{system}
    \end{equation}
	{It follows from \cite[Lemma 7.1]{Ber-Nir-91} that} for any $ u$ the function  $ \tilde u $ is well-defined and belongs to the space $ C_b([-l, l]\times \overline\Omega)\cap W^{2, p}_{loc}([-l, l]\times \overline\Omega\bsl\{-l,l\}\times \partial\Omega) $ for any $ p>0 $.

    \textbf{Step 1:} Let us briefly show that $ F $ is in fact a compact operator. Since the right-hand side of the first equation in \eqref{eqthm:solbox-1} is bounded, it is easily seen that the function $ (x, y)\mapsto \big(1+\gamma (x+l)^\alpha\big) p(y) $ is a local supersolution to equation \eqref{eqthm:solbox-1} near $ x=l $ for $ 0\leq\alpha<\alpha_0 $, $ \gamma\geq \gamma_0>0 $, where $ \alpha_0 $ and $ \gamma_0 $ depend only on $ \Vert u\Vert_{C_b((-l,l)\times\Omega)} $, a bound for $ c $ and the data and coefficients of the problem. Similarly, $ (1-\gamma(x+l)^\alpha\big) p(y) $ is a local subsolution, provided  $ \alpha $ is chosen small enough. Thus the inequality $\big(1-\gamma (x+l)^\alpha\big) p(y) \leq \tilde u(x,y) \leq \big(1+\gamma (x+l)^\alpha\big) p(y) $ holds for $ \alpha>0 $ small enough. In particular,  the function $ x\in [-l, 0]\mapsto \tilde u(x,y) $ is uniformly in $ C^\alpha $ for $ y\in\overline\Omega $. It then follows from \cite[Corollary 9.28]{Gil-Tru-01} (and the classical interior Sobolev embeddings) that $ \tilde u\in C^\alpha([-l, 0]\times\overline\Omega) $. Regularity near $ x=l $ can be shown the same way. Thus $ \tilde u\in C^\alpha([-l, l]\times\overline\Omega) $ where $ \alpha $ depends only on a bound for $ \Vert u\Vert_{C_b((-l, l)\times \Omega)} $ and $ c $, and the data and coefficients of problem \eqref{eqthm:solbox-1}. In particular, $F $ maps bounded sets of $ \mathbb R\times C_b((-l, l)\times \Omega) $ into relatively compact sets in $ C_b((-l,l)\times \Omega)$. 

	\textbf{Step 2:} We aim at applying the Leray-Schauder fixed-point theorem \cite[Corollary 13.1 item (iii)]{Zei-86} to $F(0, \cdot)$. We remark that the solutions to $ u=\sigma F(0,u) $ for $\sigma\in (0,1] $ are in fact the solutions to \eqref{pb:box-beta-pos}. In particular, Lemma \ref{lem:aprioribox} gives us a positive  constant $ C>0 $ such that any solution to \eqref{pb:box-beta-pos} satisfies the inequality $ \Vert u\Vert_{C_b(-l,l)\times \Omega}\leq C $. Let $ G:=\{u\in C_b((-l,l)\times\Omega)\,|\, \Vert u\Vert_{C_b((-l,l)\times\Omega)}\leq 2C\} $, then 
    \begin{enumerate}
        \item $ G $ is a bounded open subset of the Banach space $ C_b((-l, l)\times \Omega) $,
        \item $ 0\in G $,
        \item $ F(0, \cdot):G\to C_b((-l, l)\times \Omega) $ is a compact mapping, and
	\item {applying} Lemma \ref{lem:aprioribox},  there is no solution to $ u=\sigma F(0, u) $ with $ u\in\partial G$ and $ \sigma\in(0,1] $.
    \end{enumerate}
    Thus the Leray-Schauder fixed-point Theorem \cite[Corollary 13.1 item (iii)]{Zei-86} applies and we have $ ind(F(0, \cdot), G)=1 $, where $ ind $ is the Leray-Schauder fixed-point index. 

    \textbf{Step 3:} Let us now check that the hypotheses of the global continuation principle \cite[Theorem 14 C]{Zei-86} are satisfied. We have:
    \begin{enumerate}
        \item $ F $ is a compact mapping from $ (0, c_\varepsilon^*)\times \overline G $ into $ C_b((-l, l)\times \Omega) $, 
	\item {applying} Lemma \ref{lem:aprioribox}, there is no solution to $ u=F(c,u) $ with $ u\in\partial G $ and $ c\in [0, c_\varepsilon^*] $, and 
        \item $ ind(F(0, \cdot), G)=1 $.
    \end{enumerate}
    Thus, the global continuation principle applies and there exists a connected set of solutions $ \mathcal C $  to $ u=F(c,u) $ connecting $ \{0\}\times G $ to $ \{c_\varepsilon^*\}\times G $. In particular, there exists a solution to \eqref{pb:box-beta} for any $ c\in [0, c_\varepsilon^*] $. 

    \textbf{Step 4:} Now let us assume $ l\geq \bar l(\tau) $ (where $ \bar l(\tau) $ is given by Lemma \ref{lem:aprioribox}, item 5). Since the mapping 
    \begin{equation*}
        u\in C_b((-l, l)\times\Omega)\overset{N}{\longmapsto} \sup_{(x,y)\in (-l_0, l_0)\times \Omega} \int_{\Omega} K(y,z)u(x,z)dz+\beta u(x,y) 
    \end{equation*}
    is continuous, then $ N(\mathcal C) $ is a connected subset of $ \mathbb R $, i.e. an interval. Applying Lemma \ref{lem:aprioribox}, we have:
    \begin{itemize}
        \item[--] From point \ref{item:aprioribox-minspeed}, if $ (c,u)\in \mathcal C $ and $ c=0 $, then $ N(u)>\tau $.
        \item[--] From point \ref{item:aprioribox-maxspeed}, if $ (c,u)\in \mathcal C $ and $ c=c_\varepsilon^* $ then $ N(u)<\tau $. 
    \end{itemize}
    Thus there exists $ c\in (0, c_\varepsilon^*) $ and $ u $ such that $ (c,u)\in \mathcal C $ and  $ N(u) = \tau $. This finishes the proof of Theorem \ref{thm:solbox}.
\end{proof}
An immediate consequence is the following: 
\begin{cor}[Existence of a solution on the line]\label{cor:existenceline}
    Let Assumption \ref{hyp:gen} hold, $ \varepsilon >0 $ be such that $ \lambda_1^\varepsilon<0 $,  $ \beta\geq 0$  and  $ 0< \tau\leq\frac{-\lambda_1^\varepsilon}{2} $. Then there exists a classical positive solution to 
    \begin{equation}\label{pb:lineeps-beta}
        \begin{system}{ll}
            -u_{xx}-cu_x=\varepsilon\Delta_yu + \mu(M\star u-u)\hspace{2cm} \\
            \hfill+ u(a(y)-K\star u-\beta u) & \text{ on } \mathbb R\times \Omega \\
            \frac{\partial u}{\partial\nu}=0  & \text{ on } \mathbb R\times \partial \Omega, \\
        \end{system}
    \end{equation}
	with $ 0<c\leq { c_\varepsilon^*} $. Moreover $ u\in C_b(\mathbb R\times \Omega)\cap C^{2}(\mathbb R\times\overline\Omega) $, satisfies \eqref{eq:norm} and 
    \begin{equation*}
        \forall x\in\mathbb R, \quad \int_\Omega u(x,y)dy\leq \frac{\sup a}{k_0 }.
    \end{equation*}
\end{cor}
\begin{proof}
	Let $ 0<\tau\leq \tau_0 $ and $\bar l=\bar l(\tau)$ as in Theorem \ref{thm:solbox}. Then {it follows from} the existence theorem  (Theorem \ref{thm:solbox}), that  for any $ n\in\mathbb N $, there exists a positive classical solution $ (c_n, u_n)\in (0, c_\varepsilon^*)\times C^2((-l_n,l_n)\times\Omega)\cap C^1((-l_n, l_n)\times \overline\Omega) $  to \eqref{pb:box-beta} which satisfies \eqref{eq:norm}, where $ l_n:=\bar l+n $. {By} the uniform bound satisfied by $ \sup u_n $ (Lemma \ref{lem:aprioribox} point \ref{item:aprioribox-boundinfty}),  the classical Schauder interior estimates \cite[Theorem 6.2]{Gil-Tru-01} and the boundary Schauder estimates \cite[Theorem 6.29]{Gil-Tru-01}, there exists a constant $ C_k>0 $, independent from $ n $ such that 
    $\Vert u_n\Vert_{C^{2, \alpha}((-l_k, l_k)\times\overline\Omega)}\leq C_k$
	for any $ k < n$. {Using} a classical diagonal extraction process, there exists $u$, $ c^0 $ and a subsequence such that $ c_n\to c^0 $, and  $ \Vert u_n-u\Vert_{C^2((-l_k, l_k)\times\overline\Omega)}\to 0 $ for any $ k\in\mathbb N $. {Since $ u $ solves \eqref{eq:norm} with $\tau>0$, it is a nontrivial solution to \eqref{pb:lineeps-beta}.}  Then, {by a direct application of} Lemma \ref{lem:aprioribox} point \ref{item:aprioribox-minspeed}, $ c^0>0$ {(indeed Lemma \ref{lem:aprioribox} point \ref{item:aprioribox-minspeed} also applies to solutions defined on the whole line)}. 

    Finally  we have 
	$\forall x\in\mathbb R, \int_\Omega u(x, y)dy\leq \frac{\sup a}{k_0 }$ {by the estimate in} Lemma \ref{lem:aprioribox} point \ref{item:aprioribox-boundint}.
    This finishes the proof of Corollary \ref{cor:existenceline}.
\end{proof}

\subsection{Proof of minimality for $ \beta\geq \beta_0$}
In the case $ \beta\geq\beta_0:=\frac{k_\infty\sup a}{\mu m_0} $, we  recover the comparison principle. Indeed, the increased self-competition (via large $ \beta $) enforces the solution to remain in the region ``$u$ small'' where the system is cooperative (see Lemma \ref{lem:comp}). We can then retrieve  many of the classical properties satisfied by traveling waves in a KPP situation. 

\begin{thm}[Minimal speed traveling waves for $ \beta \geq \beta_0$]\label{thm:minspeedTW-beta}
    Let Assumption \ref{hyp:gen} hold, $ 0<\varepsilon\leq \varepsilon_0 $ --- where $ \varepsilon_0 $ is as in Lemma \ref{lem:beta-lowerbound}--- be such that $ \lambda_1^\varepsilon < 0 $, and $\beta\geq \beta_0=\frac{k_\infty\sup a}{\mu m_0}$. Then, there exists a solution $(c,u) $ to \eqref{pb:lineeps-beta}  satisfying $ c= c_\varepsilon^* $ and the limit conditions
    \begin{equation}\label{eq:limit-TW-eps}
        \liminf_{x\to -\infty}\inf_{y\in\Omega}u(x,y)>0 ,\qquad \lim_{x\to +\infty}\sup_{y\in\Omega} u(x,y)=0.
    \end{equation}
    Moreover,  $u$ is nonincreasing in $ x$, and there exists no positive solution to \eqref{pb:lineeps-beta} satisfying \eqref{eq:limit-TW-eps} and $0\leq c<c_\varepsilon^* $.

    Finally, we have 
    \begin{equation*}
        \lim_{x\to -\infty}\inf_{y\in\Omega}u(x,y)\geq \rho_\beta
    \end{equation*}
    where $ \rho_\beta $ is the constant defined in Lemma \ref{lem:beta-lowerbound}.
\end{thm}
Our main tool is the following \textit{comparison principle} for small densities.
\begin{lem}[Comparison principle]\label{lem:comp}
    Let Assumption \ref{hyp:gen} hold and  $ \beta\geq 0 $. Let $ u\in C^2 $ be a supersolution to 
    \begin{equation}\label{eq:lemcompsup}
        -cu_x-u_{xx}-\varepsilon\Delta_y u-\mu (M\star u-u)-u(a(y)-K\star u-\beta u)\geq 0
    \end{equation}
    and $ v\in C^2 $ be a subsolution to
    \begin{equation}\label{eq:lemcompsub}
        -cv_x-v_{xx}-\varepsilon\Delta_y v-\mu(M\star v-v)-v(a(y)-K\star v-\beta v)\leq 0.
    \end{equation}
    If there exists $ (x_0, y_0)\in\mathbb R\times\Omega $ such that $0< u(x_0, y_0)\leq\frac{\mu m_0}{k_\infty} $, $ u\geq v $ in a neighbourhood of $ \{x_0\}\times\Omega $,    and $ u(x_0, y_0)=v(x_0, y_0) $, then $ u\equiv v $.
\end{lem}
\begin{proof}
    Let $ (x_0, y_0) $ as in Lemma \ref{lem:comp}. Then $(x_0, y_0)$ is a local zero minimum to $ u-v $. We have:
    \begin{equation*}
        -c(u-v)_x(x_0, y_0) - (u-v)_{xx}(x_0, y_0)-\varepsilon\Delta_y (u-v)(x_0, y_0)\leq 0
    \end{equation*}
    and thus:
    \begin{equation}\label{eq:lemcomp1}
        \mu\big(M\star (u-v)-(u-v)\big) + a(y)(u-v) - uK\star u + vK\star  v -\beta u^2 + \beta v^2\leq 0.
    \end{equation}
    Using the fact that $ u(x_0, y_0)=v(x_0,y_0) $, we rewrite \eqref{eq:lemcomp1} as
    \begin{multline*}
        \int_\Omega\big( \mu M(y_0,z)-u(x_0, y_0)K(y_0, z)\big)\Big(\big(u(x_0,z)-v(x_0,z)\big)
        -\big(u(x_0, y_0) -v(x_0, y_0)\big)\Big)dz  \\ \leq 0.
    \end{multline*}
    Since $ \mu M(y_0, z)-u(x_0, y_0)K(y_0, z) >0 $ for any $ z\in\Omega $ and $ u(x_0, y)-v(x_0, y) $ is nonnegative for any $ y\in\Omega $, we conclude that
    $ u(x_0, y)=v(x_0, y)$ for any $y\in\Omega $.
    Applying the strong maximum principle, we have then $ u-v\equiv 0 $. This ends the proof of Lemma \ref{lem:comp}.
\end{proof}

\begin{lem}[Estimates for $ \beta\geq\frac{k_\infty\sup a}{\mu m_0} $]\label{lem:bigbetabox}
    Assume $ \beta\geq\beta_0=\frac{k_\infty\sup a}{\mu m_0} $. Then there exists a unique solution to \eqref{pb:box-beta}. Moreover, the solution to \eqref{pb:box-beta} is decreasing in $ x $, and the mapping $ c\mapsto u $ is decreasing. 
\end{lem}
\begin{proof}
	We divide the proof in four steps. Recall that, {due to} Theorem \ref{thm:solbox}, there exists a solution to \eqref{pb:box-beta}.
    \medskip

    \textbf{Step 1:} We show that any solution satisfies $ u(x, y)< p(y) $ at any interior point. 

    Let us define
    $\alpha:=\sup\{\zeta>0\,|\, \zeta u\leq p\}$.
	Since $ u $ is bounded  and $ p $ is positive on $ \overline\Omega $,  $ \alpha $ is well-defined and positive. Assume by contradiction that $ \alpha < 1$. By definition of $ \alpha $, there exists $ (x_0,y_0)\in [-l,l]\times \overline\Omega $ such that $ p(y_0)=\alpha u(x_0,y_0) $. Testing at $ x=\pm l$, we have $ \alpha u(-l, y_0)=\alpha p(y_0)<p(y_0) $ and $ \alpha u(l,y_0)=0<p(y_0) $; thus $ x_0\in (-l,l) $. If $ y_0\in\partial\Omega$, then {it follows from Hopf's Lemma that} $ \frac{\partial p-\alpha u}{\partial\nu}(x_0,y_0)<0 $, which contradicts the Neumann boundary conditions satisfied by $ u $ and $ p $. Thus $ y_0\in\Omega $. Next we remark that 
    \begin{multline*}
        -c(\alpha u)_x-(\alpha u)_{xx}-\varepsilon \Delta_y (\alpha u)-\mu\big(M\star (\alpha u)-(\alpha u)\big)-a(y_0)(\alpha u) \\ 
        =(\alpha u)(-K\star u-\beta u) < \alpha u\big(-K\star  (\alpha u)-\beta (\alpha u)\big),
    \end{multline*}
	since $ \alpha <1 $. Hence $\alpha u$ is a subsolution to \eqref{eq:lemcompsub}. Moreover $ p $ is a supersolution to \eqref{eq:lemcompsup}. Finally, {by the estimate in} Lemma \ref{lem:aprioribox} point \ref{item:aprioribox-boundinfty} and the condition $ \beta\geq\beta_0 $, we have the inequality $ \Vert u\Vert_{L^\infty} \leq \frac{\sup a}{\beta}\leq \frac{\mu m_0}{k_\infty}$, and by definition $ (x_0,y_0) $ is the global minimum of $ (p-\alpha u)$. Thus Lemma \ref{lem:comp} applies and  $\alpha u=p$, which is a contradiction.

    Thus $ \alpha\geq 1 $, which shows that $ u\leq p $. Assume now that $ u(x,y)=p(y) $ for some $ (x,y)\in(-l,l)\times\Omega $, then  Lemma \ref{lem:comp} applies and we have $ u=p $ in $ (-l, l)\times \Omega $, which is again a contradiction. We conclude that the strict inequality holds:
    \begin{equation*}
        \forall (x, y)\in (-l,l)\times\Omega, \qquad u(x,y)<p(y).
    \end{equation*}

    \textbf{Step 2:} We show that the solution $ u $ is unique. Here we use a sliding argument. Let $u,v$ be two solutions to \eqref{pb:box-beta}, and define:
    \begin{equation*}
        \bar x:=\inf\{\gamma>0\,|\, \forall (x, y)\in (-l, l)\times \Omega, u(x+\gamma, y)\leq v(x,y)\}.
    \end{equation*}
	{Because of} the boundary conditions satisfied by $ u $ and $ v$, we have $ 0\leq\bar x<2l $. Assume by contradiction that $ \bar x>0 $. 
    We remark that $ (x,y)\mapsto u(x+\bar x, y) $ is a subsolution to \eqref{eq:lemcompsub}.   
    By definition of $ \bar x $, there exists $ (x_0, y_0)\in(-l, l-\bar x)\times\Omega $ such that the equality $ u(x_0+\bar x, y_0)=v(x_0,y_0)$ holds. In view Lemma \ref{lem:comp}, this leads to a contradiction. Thus $ \bar x \leq 0 $ and $ u\leq v$. Exchanging the roles of $u$ and $v$, we have in turn $ v\leq u $. This shows the uniqueness of $u$. 

    \textbf{Step 3:} We show that $ x\mapsto u(x,y) $ is decreasing. Repeating the sliding argument in Step 2 with $ u=v $, we have
    $ u(x+\bar x, y)\leq u(x, y)$ for any $ \bar x>0$,
    which shows that $ u $ is nonincreasing. Moreover, equality cannot hold at an interior point in the above inequality, for Lemma \ref{lem:comp} would lead to a contradiction. This shows that $ x\mapsto u(x, \cdot) $ is decreasing. 

    \textbf{Step 4:} We show that $ c\mapsto u $ is decreasing.  Let $ \bar c \leq c $, $ u $ (resp. $v$) be the solution to \eqref{pb:box-beta} associated with the speed $ c $ (resp. $ \bar c  $). Let also:
    \begin{equation*}
        \bar x:=\inf\{\gamma>0\,|\,\forall y\in\Omega,  u(x+\gamma, y)\leq v(x, y) \}
    \end{equation*}
    and assume by contradiction that $ \bar x > 0 $. Then 
    \begin{align*}
        -cv_x-v_{xx}-\varepsilon \Delta_y v&= \mu (M\star v-v)+v(a-K\star u-\beta u) + (\bar c -c)v_x\\
                                           &\geq \mu (M\star v-v)+v(a-K\star u-\beta u),
    \end{align*}
    since, as shown above, $ v_x\leq 0 $. Then, $v$ is a supersolution to \eqref{eq:lemcompsup} and Lemma \ref{lem:comp} leads to a contradiction. Thus $ c\mapsto u $ is nonincreasing. 
    Moreover if $ \bar c <c $, then we deduce from the above argument that $ v>u $. Hence $ c\mapsto u $ is in fact decreasing.

    This ends the proof of Lemma \ref{lem:bigbetabox}.
\end{proof}
In particular, we notice that:
\begin{cor}[Existence of monotone fronts]\label{cor:decreasing}
    Let $ \beta\geq \beta_0 =\frac{k_\infty\sup a}{\mu m_0}$. Then the solution constructed in Corollary \ref{cor:existenceline} is decreasing in $ x $. 
\end{cor}

The next results shows that if $ u $ is a traveling wave, then $ c\geq c_\varepsilon^* $.
\begin{lem}[$c^*_\varepsilon$ is the minimal speed]\label{lem:c*-beta}
    Let Assumption \ref{hyp:gen} hold, $ \varepsilon>0 $ be such that   $ \lambda_1^\varepsilon<0 $, and $ u $ be a positive solution to \eqref{pb:lineeps-beta} with $ 0\leq c\leq c_\varepsilon^* $ and either 
    \begin{enumerate}[label=(\textit{\roman*}), ref=(\textit{\roman*})]
        \item \label{item:c*-beta-sup} $ \beta>0 $ and $ \lim_{x\to+\infty}\sup_{y\in\Omega}u(x,y)=0 $, or
        \item \label{item:c*-beta-int} $ \beta=0 $ and $ \lim_{x\to+\infty}\int_\Omega u(x,y)dy=0 $.
    \end{enumerate}
    Then $c=c^*_\varepsilon $.
\end{lem}
\begin{proof}
    It follows from our hypothesis \ref{item:c*-beta-sup} or \ref{item:c*-beta-int} that  we can find arbitrary large intervals $ [\bar x-L, \bar x+L] $ on which 
    \begin{equation}\label{eqlem:minimal}
        \sup_{(x,y)\in(\bar x-L, \bar x+L)\times\Omega} \left(\int_{\Omega}K(y,z) u(x,z)dz+\beta u (x,y)\right)\leq \delta,
    \end{equation}
	for arbitrarily small $ \delta>0$. Since equation \eqref{pb:lineeps-beta} is invariant by translation in $ x $, we may assume {without loss of generality} that $ \bar x = 0 $.

    Assume by contradiction that  $ c<c^*_\varepsilon $. Let   
    $ \theta:=\sqrt{\frac{(c^*_\varepsilon)^2-c^2}{8}} $, $ L:=\frac{\pi}{2\theta} $, $ \delta:=\frac{-\lambda_1^\varepsilon}{4}>0$,
    and $ \psi(x, y):=e^{-\frac{c}{2}x}\cos(\theta x)\varphi^\varepsilon(y) $, where $ \varphi^\varepsilon $ is the principal eigenfunction solution to \eqref{eq:eigeneps} satisfying $ \sup_{y\in\Omega}\varphi^\varepsilon = 1 $. $ \psi $ satisfies
    \begin{equation*}
        -c\psi_x-\psi_{xx}-\varepsilon\Delta_y \psi-\mu(M\star \psi-\psi)=a(y)\psi + \left(\frac{c^2}{4}+\theta^2+\lambda_1^\varepsilon\right)\psi.
    \end{equation*}

    Since $ u $ is positive in $ [-L, L]\times \overline\Omega $, we can define
    $\alpha:=\sup\{\zeta>0\,|\,\zeta\psi\leq u \}$.
	By definition of $ \alpha $ there exists $ (x_0,y_0)\in [-L, L]\times \overline\Omega $ such that  $ \alpha\psi(x_0,y_0)=u(x_0,y_0)$. {Because of} the boundary conditions satisfied by $ u $ and $ \psi $, $ (x_0, y_0) $ cannot lie on the boundary of $ [-L, L]\times\Omega $. Thus $ (x_0,y_0)$ belongs to $ (-L, L)\times \Omega $ and, since $ u $ satisfies \eqref{eqlem:minimal} we have
    \begin{align*}
        0&\geq -\varepsilon \Delta_y(u-\alpha\psi)(x_0,y_0) -(u-\alpha\psi)_{xx}(x_0,y_0)               \\
         &\quad -\mu\big(M\star(u-\alpha\psi)- (u-\alpha\psi)\big)(x_0,y_0)-a(y_0)(u-\alpha\psi)(x_0,y_0)   \\
         &\geq -\delta u(x_0,y_0)-\alpha\left(\frac{c^2}{4}+\theta^2+\lambda_1^\varepsilon\right)\psi(x_0,y_0) \\
         &=\left(-\delta -\frac{c^2}{8}-\frac{3\lambda_1^\varepsilon}{4}\right)u(x_0,y_0)\geq \left( -\delta -\frac{\lambda_1^\varepsilon}{2}\right)>0,
    \end{align*}
    since $ \delta = \frac{-\lambda_1^\varepsilon}{4}  $. This is a contradiction.
\end{proof}
\begin{lem}[Lower estimate on positive infima]\label{lem:infpos}
    Let Assumption \ref{hyp:gen} be satisfied,  let $ 0<\varepsilon\leq\varepsilon_0 $ and $ \beta\geq 0 $, where $ \varepsilon_0 $ is as in Lemma \ref{lem:beta-lowerbound}. Assume $ \lambda_1^\varepsilon < 0$. Let $ u $ be a solution to \eqref{pb:lineeps-beta} which satisfies $ \inf_{(x,y)\in\mathbb R\times\Omega} u(x,y)> 0 $. Then 
    \begin{equation*}
        \inf_{(x,y)\in\mathbb R\times\Omega}u(x,y)\geq\rho_{\max(\beta, \beta_0)} 
    \end{equation*}
    where $ \rho_\beta $ is the constant from Lemma \ref{lem:beta-lowerbound}.
\end{lem}
\begin{proof}
	For any $ B\geq 0 $, let   $ p^{B}$ be a nonnegative nontrivial solution to \eqref{eq:evolstat-beta} (substituting $ \beta $ with $B $). Since $ \inf u >0 $ and $ \sup p^{B}\leq \frac{\sup a}{B} $ ({by the estimate in} Lemma \ref{lem:apriorip} item \ref{item:apriorip-boundinfty}), there exists a constant $ \beta'>0 $ such that
    \begin{equation*}
        \beta'=\inf\{B>0\,|\, p^{B}\leq u \}.
    \end{equation*}
    Assume by contradiction that $ \beta'>\max(\beta, \beta_0) $. Then two cases may occur:

    \textbf{Case 1:} Assume there exists $ (x_0, y_0)\in \mathbb R\times \overline \Omega $ such that 
    $u(x_0, y_0)=p^{\beta'}(y_0)$.
	Assume by contradiction that $ y_0\in\partial\Omega $. Then $ y_0 $ is the minimum of $ u-p^{\beta'} $ and, {by applying} Hopf's Lemma, we have $ \frac{\partial (u-p^{\beta'})}{\partial\nu}(x_0,y_0)<0 $, which contradicts the Neumann boundary conditions satisfied by $ u $ and $ p^{\beta'} $. Thus $ y_0\in\Omega $.

    Then, since $ \beta'>\beta $, $ p^{\beta'} $ is a subsolution to \eqref{pb:lineeps-beta}, $ u\geq p^{\beta'} $ and since $ {\beta'}> \beta_0 $ we have $ \Vert p^{\beta'}\Vert_{C_b(\Omega)}< \frac{\mu m_0}{k_\infty}$. Thus Lemma \ref{lem:comp} applies and $ u=p^{\beta'} $.  Since $ {\beta'}\neq \beta $, this is a contradiction.

    \textbf{Case 2:} If the latter does not hold,  then by definition of $ \beta' $ there exists a sequence $ (x_n, y_n) $ such that $ u(x_n, y_n)-p^{\beta'}(y_n)\to 0 $. Since $ \Omega $ is bounded, up to an extraction we have $ y_n\to y_0 \in\overline\Omega $. Then 
    $u(x_n, y_n)\to_{n\to\infty} p^{\beta'}(y_0)$.

    Since equation \eqref{pb:lineeps-beta} is invariant by translation in $ x $, we consider the shifted functions $ u^n(x,y):=u(x+x^n, y) $ which also satisfy \eqref{pb:lineeps-beta}. Then from the standard elliptic estimates and up to an extraction, $ u^n $ converges locally uniformly to  $ u^\infty $,  which is a classical solution to \eqref{pb:lineeps-beta} and  also satisfies
    $u^\infty(0, y_0)=p^{\beta'}(y_0)$ and $ \forall x, y, u^\infty(x,y)\geq p^{\beta'}(y)$.
    Applying Case 1 to $ (u^\infty, p^{\beta'}) $ leads to a contradiction.

    \medskip

	We have shown that either case leads to a contradiction if $ \beta'>\max(\beta, \beta_0) $. Hence $ {\beta'}\leq \max(\beta, \beta_0) $ and we conclude {by the estimate in} Lemma \ref{lem:beta-lowerbound} that the inequality $ u\geq\rho_{\max(\beta, \beta_0)}$ holds.
\end{proof}

\begin{proof}[Proof of Theorem \ref{thm:minspeedTW-beta}]
    Let $ \tau:=\frac{1}{2}\min\left((k_0|\Omega|+\beta)\rho_\beta, {-\lambda_1^\varepsilon}\right) $, where $ \rho_\beta $ is the constant from Lemma \ref{lem:beta-lowerbound}, and $ u $ be the corresponding solution to \eqref{pb:lineeps-beta}, i.e. a solution to \eqref{pb:lineeps-beta} constructed in Corollary \ref{cor:existenceline}, which satisfies 
    \begin{equation}\label{eqthm:minspeed}
        \sup_{(x,y)\in (-l_0, l_0)\times\Omega}\left(\int_\Omega K(y,z)u(x,z)dz+\beta u(x,y)\right) = \tau \leq \frac{1}{2}(k_0|\Omega|+\beta)\rho_\beta.
    \end{equation}
	Recall that, {as stated in} Corollary \ref{cor:decreasing}, $ x\mapsto u(x,y) $ is decreasing.

    We divide the proof in three steps. 

    \textbf{Step 1:} We show that $ \inf_{(x,y)\in\mathbb R\times\Omega} u(x,y)=0 $.

	Indeed, {recalling} \eqref{eqthm:minspeed}, we have
    \begin{align*}
        (k_0|\Omega|+\beta)u(0,0)&\leq \sup_{(x,y)\in (-l_0, l_0)\times\Omega}\int_\Omega K(y,z)u(x,z)dz+\beta u(x,y)\\
                                 &\leq\frac{1}{2}(k_0|\Omega|+\beta)\rho_{\beta},
    \end{align*}
    and thus $ u(0, 0) \leq \frac{1}{2}\rho_\beta<\rho_\beta$. The contrapositive of Lemma \ref{lem:infpos} concludes.

    \textbf{Step 2:} We show that $ \lim_{x\to+\infty}\sup_{y\in\Omega}u(x,y)=0 $.

	{We proved in Step 1 that} $ \inf u=0 $. Since $ u(x,y)>0 $ for $ (x,y)\in\mathbb R\times\overline\Omega $ and $ u $ is decreasing in $ x$, we must then have $ \lim_{x\to+\infty}\inf_{y\in\Omega}u(x,y)=0 $. 

	Let $ u^n(x,y):=u(x-n, y) $ and $ y_n $ such that $ u^n(0, y_n)=\inf_{y\in\Omega}u^n(0, y) $. Since $ \Omega $ is bounded, up to the extraction of a subsequence there exists $ y\in\overline\Omega $ such that $ y_n\to y_0$. {It follows from} the classical elliptic estimates that we then extract from $ (u^n) $ a subsequence which converges locally uniformly on $ \mathbb R\times\Omega$ to a limit function $ u^0 $, which is still a classical solution to \eqref{pb:lineeps-beta}. 

	Since $ u$ is decreasing,  the equalities $ \lim_{x\to +\infty}\sup_{y\in\Omega} u(x,y)=\sup_{y\in\Omega} u^0(0,y) $ and $ 0=\lim_{x\to +\infty}\inf_{y\in\Omega} u(x,y)=\inf_{y\in\Omega} u^0(0,y)=u(0, y_0) $ hold. If $ y_0\in\partial\Omega$ and $ u^0\not\equiv 0$, then {it follows from Hopf's Lemma that} $ \frac{\partial u^0}{\partial\nu}(y_0)<0 $, which contradicts the Neumann boundary conditions satisfied by $ u^0$. If $ y \in\Omega $ then the strong maximum principle imposes $ u^0\equiv 0 $. In either case, we have $ u^0\equiv 0 $ and thus $ \lim_{x\to+\infty}\sup_{y\in\Omega} u(x,y)=0 $.

    \textbf{Step 3:} We show that $ \lim_{x\to-\infty}\inf_{y\in\Omega}u(x,y)\geq \rho_\beta $.

	Let $ u^n(x,y):=u(x+n, y) $. {Using} the classical elliptic estimates, we  extract from $ (u^n) $ a subsequence that converges locally uniformly on $ \mathbb R\times\Omega$ to a limit function $ u^0 $, which is still a classical solution to \eqref{pb:lineeps-beta}. 

    Since $ u$ is decreasing,  we have $ \lim_{x\to -\infty}\inf_{y\in\Omega} u(x,y)=\inf_{y\in\Omega, x\in\mathbb R} u^0( x,y) $. In particular, $ \inf_{(x,y)\in\mathbb R\times\Omega}u^0(x,y)>0 $. Applying Lemma \ref{lem:infpos}, we conclude that the lower estimate $ \lim_{x\to-\infty}\inf_{y\in\Omega}u(x,y)=\inf_{(x,y)\in\mathbb R\times\Omega}u^0(x,y)\geq \rho_{\beta} $ holds.

    To conclude the proof of Theorem \ref{thm:minspeedTW-beta}, we remark that Lemma \ref{lem:c*-beta} states that $ 0\leq c<c_\varepsilon^* $ is incompatible with $ \lim_{x\to+\infty}\sup_{y\in\Omega}u(x,y)=0 $. This shows that $ c=c_\varepsilon^* $. This finishes the proof of Theorem \ref{thm:minspeedTW-beta}.
\end{proof}

\subsection{Minimal speed traveling wave for $ \beta=0$}

Here we construct traveling waves for our initial regularized problem
\begin{equation}\label{eq:lineeps}
    \begin{system}{ll}
        -\varepsilon\Delta_y u-u_{xx}-cu_x=\mu(M\star u-u)+u(a(y)-K\star u)&\text{ in } \mathbb R\times \Omega \\
        \frac{\partial u}{\partial\nu}=0 & \text{ on } \mathbb R\times\partial\Omega.
    \end{system}
\end{equation}
Notice that \eqref{eq:lineeps} is exactly the equation \eqref{pb:lineeps-beta} in the special case $ \beta=0$. In particular, our results obtained in Corollary \ref{cor:existenceline} and Lemmas \ref{lem:comp}, \ref{lem:c*-beta} and \ref{lem:infpos} still apply to the solutions of \eqref{eq:lineeps}.

Our result is the following:
\begin{thm}[Regularized minimal speed traveling waves] \label{thm:reg-TW}
    Let Assumption \ref{hyp:gen} hold, $ 0<\varepsilon\leq \varepsilon_0 $ (where $ \varepsilon_0 $ is as in Lemma \ref{lem:beta-lowerbound}) and assume $ \lambda_1^\varepsilon <0$. Then, there exists a nonnegative nontrivial traveling wave $(c,u) $ for \eqref{eq:lineeps} with $ c=c_\varepsilon^*$, i.e. a bounded classical solution which satisfies:
    \begin{equation}\label{eq:reg-TW-limit}
        \liminf_{x\to -\infty}\inf_{y\in\Omega}u(x,y)>0, \qquad \limsup_{x\to +\infty}\int_{\Omega}u(x,y)dy = 0.
    \end{equation}
    Moreover, $ c_\varepsilon^* $ is the minimal speed for traveling waves in the sense that there exists no traveling wave for equation \eqref{eq:lineeps} with $ 0\leq c< c_\varepsilon^*$. 

    Finally, $ u $ can be chosen so that $ \sup_{x\in\mathbb R}\int_{\Omega}u(x,y)dy\leq \frac{\sup a}{k_0}$ and 
    \begin{equation*}
        \liminf_{x\to-\infty}\inf_{y\in\Omega} u(x,y) \geq \rho_{\beta_0},
    \end{equation*}
    where $ \beta_0=\frac{k_\infty\sup a}{\mu m_0}$ and $ \rho_{\beta_0} $ is given by Lemma \ref{lem:beta-lowerbound}.
\end{thm}

{Two key elements for the proof of Theorem \ref{thm:reg-TW} are the following Harnack-type inequality, and the following Lemma \ref{lem:localinf},  which states that $ \inf_{y\in\Omega} u(x,y) $ and $ \int_\Omega u(x,y)dy $ are locally comparable. }
\begin{lem}[Harnack inequality for the mass]\label{lem:harnackmass}
    Let Assumption \ref{hyp:gen} hold and $\varepsilon>0 $. Let $ \bar c>0 $, $ R>0 $  and $ W>0 $ be given. Let $ (c,u) $ be a solution to \eqref{eq:lineeps} with $ |c|\leq \bar c $, $ u\geq 0 $ and $\int_\Omega u(x,y)dy\leq W $ for $ x\in(-R, R)$. Then, there exists a constant $\mathcal H>0 $ depending only on $R$, $ \Vert a\Vert_{L^\infty} $, $W$, $ k_\infty$ and $ \bar c $  such that 
    \begin{equation*}
        \sup_{|x|\leq R} \int_\Omega u(x, z)dz \leq \mathcal H\inf_{|x|\leq R}\int_\Omega u(x,z)dz.
    \end{equation*}
\end{lem}
\begin{proof}
    Let $ I(x):=\int_\Omega u(x, y)dy $, then $ I $ solves
    \begin{multline*}
        -cI_x-I_{xx} = \int_\Omega a(y)u(x,y)dy - \int_\Omega (K\star u)(y)u(x,y)dy \\
        =\left(\int_\Omega a(y)\frac{u(x,y)}{\int_\Omega u(x, z)dz}dy  -\int_\Omega K\star u(x, y)\frac{u(x, y)}{\int_\Omega u(x,z)dz}dy\right) I. 
    \end{multline*}
    Now we remark that
    $\left|\int_\Omega a(y)\frac{u(x,y)}{\int_\Omega u(x, z)dz}dy\right|\leq \Vert a\Vert_{L^\infty}$ and
    \begin{equation*}
        0\leq\int_\Omega K\star u(x, y)\frac{u(x, y)}{\int_\Omega u(x,z)dz}dy \leq \Vert K\star  u\Vert_{L^\infty} \leq k_\infty\int_\Omega u(x,y)dy\leq k_\infty  W,
    \end{equation*}
    for any $ x\in\mathbb R $, so that the classical Harnack inequality \cite[Corollary 9.25]{Gil-Tru-01} applies. 
\end{proof}

\begin{lem}[Integral-infimum comparison]\label{lem:localinf}
    Let Assumption \ref{hyp:gen} hold and $\varepsilon>0 $. Let $ \bar c>0 $, $x_0\in\mathbb R $, $ \kappa>0$ and $ W>0 $ be given. Let $ (c,u) $ be a solution to \eqref{eq:lineeps} with $ |c|\leq \bar c $, $ u\geq 0 $ and $\int_\Omega u(x,y)dy\leq W $ for $ |x-x_0|\leq 1$. Assume  
    \begin{equation*}
        \int_\Omega u(x_0, y)dy\geq\kappa.
    \end{equation*}
    Then, there exists a positive constant $ \bar\kappa $ depending only on $\Vert a \Vert_{L^\infty} $, $\mu $, $m_0$, $ k_\infty $, $ \bar c $, $ W $  and $ \kappa $ such that
    \begin{equation*}
        \inf_{y\in\Omega} u(x_0, y)\geq \bar\kappa.
    \end{equation*}
\end{lem}
\begin{proof}
	Since \eqref{eq:lineeps} is translation-invariant in $ x $, we will assume {without loss of generality}  that $ x_0=0 $.

    \textbf{Step 1:} We construct a local subsolution. 

    From Lemma \ref{lem:harnackmass} there exists a constant $ \mathcal H>0 $  such that 
    \begin{equation*}
        \kappa\leq \underset{x\in (-1, 1)}\sup\int_\Omega u(x, z)dz \leq \mathcal H \underset{x\in (-1, 1)}\inf\int_\Omega u(x,z)dz\leq \mathcal H\kappa.
    \end{equation*}
    Thus $ u $ satisfies:
    \begin{equation*}
        -cu_x-u_{xx}-\varepsilon\Delta_y u\geq \mu m_0 \frac{\kappa}{\mathcal H} + \left(\inf_\Omega a-\mu-k_\infty W \right) u.
    \end{equation*}
    In particular there exists constants $ \gamma>0 $ and $ \alpha>0 $ depending only on $\Vert a \Vert_{L^\infty} $, $\mu $, $m_0$, $ k_\infty $, $ W $ and $ \kappa $ such that
    \begin{equation}\label{eq:subsolcosh}
        -cu_x-u_{xx}-\varepsilon\Delta_y u\geq \gamma - \alpha u
    \end{equation}
    We define, for $ \theta:=\frac{2}{\sqrt{c^2+4\alpha}}\mathrm{atanh}\left(\frac{c}{\sqrt{c^2+4 \alpha}}\right) $, 
    \begin{equation*}
        f^\delta(x):=\frac{\gamma}{\alpha}-\delta e^{-\frac{c}{2}(x-\theta)}\cosh\left(\frac{x-\theta}{2}\sqrt{c^2+4\alpha}\right).
    \end{equation*}
    Then $ f^\delta $ satisfies
    \begin{equation*}
        -cf^\delta_x-f^\delta_{xx}= \gamma - \alpha {f^\delta}.
    \end{equation*}
    In particular, $ f^\delta $ satisfies the equality in \eqref{eq:subsolcosh}. Moreover for any $ \delta>0 $, $ f^\delta $ has a unique maximum located at 0 and $ f^\delta\to -\infty $ as $ x\to\pm\infty $. Finally, the mapping  $ \delta\mapsto f^\delta $ is decreasing.

    \textbf{Step 2:} We identify $ \delta_0 $ such that $ u\geq f^{\delta_0} $. 

    Let $ \delta_0:=\inf\{\delta>0\,|\, \forall x\in (-1, 1), f^\delta\leq u\} $. We claim that we have either $ f^{\delta_0}(1)\geq 0 $ or $ f^{\delta_0}(-1)\geq 0 $. Indeed, assume by contradiction that the inequalities $ f^{\delta_0}(-1)<0 $  and $ f^{\delta_0}(1) < 0 $ hold. Then there exists $ x_0\in (-1, 1) $, $ y_0\in\overline\Omega $ such that $ u(x_0, y_0)=f^{\delta_0}(x_0) $. If $ y_0\in\partial\Omega $, then {it follows from Hopf's Lemma that} $ \frac{\partial (u-f^{\delta_0})}{\partial\nu}(x_0, y_0)<0 $ since 0 is a minimum for the function $ u-f^{\delta_0} $. This contradicts the Neumann boundary condition satisfied by $ u $ since $ \frac{\partial f^{\delta_0}}{\partial\nu}(x_0,y_0)=0 $. If $ y_0\in\Omega $, we have 
    \begin{multline*}
        -c(u-f^{\delta_0})_x(x_0, y_0) - (u-f^{\delta_0})_{xx}(x_0, y_0)- \varepsilon \Delta_y(u-f^{\delta_0})(x_0, y_0)\\ 
        \geq \big(\gamma-\alpha u(x_0, y_0)\big)-\big(\gamma-\alpha f^{\delta_0}(x_0, y_0)\big) 
        =0.
    \end{multline*}
	{By a direct application of} the strong maximum principle, we have then $ u=f^{\delta_0} $ in $ (-1,1)\times\overline\Omega$, which is a contradiction since $ f^{\delta_0} $ is not positive in $ (-1, 1) $.

    \textbf{Step 3:} We show that $ \delta_0 $ is bounded by a constant depending only on $ \bar c $, $ \alpha $ and $ \gamma $.

    Let us define $ \delta_1^c:=\inf\{\delta>0\,|\, f^\delta(-1)<0 ~\mathrm{and}~f^\delta(1)<0\} $. $ \delta_1^c $ is well-defined since $ \lim_{\delta\to +\infty}f^\delta(\pm 1)=-\infty $ and $ \lim_{\delta\to 0}f^\delta(\pm 1)=\frac{\gamma}{\alpha}>0 $. Moreover, we have either $ f^{\delta_1^c}(1)=0 $ or $ f^{\delta_1^c}(-1) =0 $. Thus
    \begin{equation*}
        \delta_1^c = \frac{\gamma}{\alpha} \max\left(\frac{e^{\frac{c}{2}(1-\theta)}}{\cosh\left(\frac{1-\theta}{2}\sqrt{c^2+2\alpha}\right)}, \frac{e^{\frac{c}{2}(-1-\theta)}}{\cosh\left(\frac{-1-\theta}{2}\sqrt{c^2+2\alpha}\right)}\right).
    \end{equation*}
    Since $\theta $ depends continuously on $ c$, the mapping $ c\mapsto f^{\delta_1^c}(0) $ is continuous. Moreover for any $ |c|\leq \bar c$, $ f^{\delta_1^c}(0)>0 $ since $x=0$ is the strict maximum of $ f^{\delta_1^c} $. Finally $ \delta_0\leq \delta_1^c $ since the mapping $ \delta\mapsto f^\delta $ is decreasing. We have then
    \begin{equation*}
        \inf_{y\in\Omega}u(0, y)\geq \inf_{|c|\leq \bar c}f^{\delta_1^c}(0)>0 
    \end{equation*}
    where the right-hand side depends only on $ \bar c $, $ \alpha $ and $ \gamma $.
    This finishes the proof of Lemma \ref{lem:localinf}. 
\end{proof}
\begin{lem}[Infimum estimate on the left]\label{lem:infxneg}
    Let Assumption \ref{hyp:gen} be satisfied, let  $ 0<\varepsilon\leq \varepsilon_0 $ be such that $ \lambda_1^\varepsilon <0$ (where $ \varepsilon_0 $ is given by Lemma \ref{lem:beta-lowerbound}), let finally  $ \beta'\geq \beta_0=\frac{k_\infty\sup a}{\mu m_0} $ and $u$ be a solution to \eqref{eq:lineeps} with $ 0\leq c\leq c_{\varepsilon}^* $ and $ \beta=0 $.    Suppose 
    \begin{equation*}
        \forall y\in\Omega, \quad u(0, y)\geq 2 \frac{\sup a}{\beta'}.
    \end{equation*}
    Then,
    \begin{equation*}
        \forall x\leq 0, y\in\Omega, \quad u(x, y)\geq \rho_{\beta'}
    \end{equation*}
    where $ \rho_{\beta'} $ is given by Lemma \ref{lem:beta-lowerbound}.
\end{lem}
\begin{proof} 
    We divide the proof in two step.

    \textbf{Step 1:} We show that $ \inf_{x\leq 0, y\in\Omega} u(x,y)>0 $. 

    Let $ \varphi^\varepsilon $ be a positive solution to \eqref{eq:eigeneps}, normalized so that
    \begin{equation*}
        \sup_{y\in\Omega}\varphi^\varepsilon(y)=\frac{1}{2}\min\left(\inf_{y\in\Omega}u(0, y), \frac{-\lambda_1^\varepsilon}{|\Omega|k_\infty}, \frac{\mu m_0}{k_\infty}\right)>0.
    \end{equation*}
    We define
    $\alpha:=\inf\{\zeta>0\,|\, \forall x\in (-\infty, 0), y\in\Omega, (1+\zeta x)\varphi^\varepsilon(y)\leq u(x, y)\}$.
    Remark that, since $ u $ is positive and $ \varphi^\varepsilon(y)< u(0, y) $ for any $ y\in\mathbb R $, $ \alpha $ is well-defined.

	Assume by contradiction that $ \alpha > 0$. Then by definition of $ \alpha $ there exists a point $ (x_0,y_0)\in\left(-\frac{1}{\alpha}, 0\right)\times\overline\Omega $ such that $ u(x_0,y_0)=(1+\alpha x_0)\varphi^\varepsilon(y_0) $. {Because of} the boundary conditions satisfied by $ u $ and $ (1+\alpha x)\varphi $, $ (x_0, y_0) $ cannot be in the boundary of $ \left[\frac{-1}{\alpha}, 0\right]\times\Omega $.
    Letting $ v(x,y):=(1+\alpha x)\varphi^{\varepsilon}(y) $, we remark that, since $ x_0<0 $, we have
    \begin{align*}
        -cv_x(x_0, y_0)-v_{xx}(x_0, y_0)-&\varepsilon\Delta_y v (x_0, y_0)- \mu(M\star v-v)(x_0, y_0) \\
        - v\big(a(y_0)-K\star v\big)(x_0, y_0) 
        &=-c\alpha\varphi(y_0) + \lambda_1^\varepsilon v(x_0, y_0)\\
        &\quad +v(x_0, y_0)(K\star v)(x_0, y_0) \\
        &\leq \frac{\lambda_1^\varepsilon}{2}<0, 
    \end{align*}
    since $ v(x_0,y_0)\leq \frac{-\lambda_1^\varepsilon}{2|\Omega|k_\infty} $ (recall that $ v $ is increasing in $x$). Hence $ v $ is a local subsolution to \eqref{eq:lemcompsub} near $(x_0, y_0) $, and Lemma \ref{lem:comp} leads to $ u\equiv v $, which is a contradiction.

    Thus $ \alpha=0 $ and we have shown  that $ \forall x<0 $, $ \varphi^\varepsilon(y)\leq u(x, y) $. In particular we have the lower estimate $ \inf_{x<0, y\in\Omega} u(x,y)\geq \inf_{y\in\Omega}\varphi^\varepsilon(y)>0 $.

    \medskip

    \textbf{Step 2:} We remove the dependency in $ \varepsilon $.

    Let $ v $ be  a decreasing solution to \eqref{pb:lineeps-beta} with $ c=c_\varepsilon^* $ constructed in Theorem \ref{thm:minspeedTW-beta}. Define $ \tilde v(x, y)=v(-x, y) $. Then $\tilde v $ satisfies:
    \begin{equation*}
        c^*_\varepsilon\tilde v_x -\tilde v_{xx}-\varepsilon\Delta_y \tilde v-\mu(M\star \tilde v-\tilde v)=\tilde v(a(y)-K\star\tilde v-\beta'\tilde v).
    \end{equation*}
    In particular, 
    \begin{align*}
        -c\tilde v_x -\tilde v_{xx}-\varepsilon\Delta_y \tilde v-\mu(M\star \tilde v-\tilde v)= &\tilde v(a(y)-K\star \tilde v-\beta' \tilde v)-(c+c^*_\varepsilon)\tilde v_x\\
        \leq & \tilde v(a(y)-K\star \tilde v),
    \end{align*}
	since $ \tilde v_x\geq 0 $. 
	Moreover,  $ \sup v\leq \frac{\sup a}{\beta'} $ {by the estimate in} Theorem \ref{thm:minspeedTW-beta}. {Using Lemma \ref{lem:comp} will then allow us to compare $ \tilde v$ with $u$.}

	Since $ \tilde v\to 0 $ when $ x\to -\infty $ and {as a result of  Step 1 above},  there exists a positive shift $ \zeta>0 $ such that $ \tilde v(x+\zeta, y)\leq \frac{1}{2} \inf_{\bar x<0, \bar y\in\Omega} u(\bar x,\bar y) $ for any $ (x,y)\in (-\infty, 0)\times \Omega $. {Using} a  sliding argument {simliar to the one} in Step 2 of Lemma \ref{lem:bigbetabox}, then for any $ \zeta\in\mathbb R $, $ x<0 $ and $ y\in\Omega $, we have $ u(x, y)\geq \tilde v(x+\zeta, y) $.  
	Taking the limit $ \zeta\to +\infty $, we get that $ \inf_{x<0,y\in\Omega}u(x,y)\geq \lim_{x\to +\infty}\inf_{y\in\mathbb R}\tilde v(x,y) \geq \rho_{\beta'} $, {by the estimate in} Theorem \ref{thm:minspeedTW-beta}.

    This finishes the proof of Lemma \ref{lem:infxneg}.
\end{proof}

\medskip

We are now in a position to prove Theorem \ref{thm:reg-TW}.
\begin{proof}[Proof of Theorem \ref{thm:reg-TW}]
    We divide the proof in two steps.

    \textbf{Step 1:} We construct a solution with $ \limsup_{x\to +\infty} \int_\Omega u(x, y)dy=0 $.

	Let $ (c,u) $ be the solution constructed in Corollary \ref{cor:existenceline} with $ \beta=0 $ and the normalization $ \tau=\frac{1}{2}\min\left(\rho_{\beta_0}k_0|\Omega|, \frac{-\lambda_1^\varepsilon}{2}\right) $, where $ \beta_0=\frac{k_\infty\sup a}{\mu m_0}$ and $ \rho_{\beta_0} $ is given by Lemma \ref{lem:beta-lowerbound}. Assume by contradiction that $ \limsup_{x\to +\infty} \int_\Omega u(x,y)dy > 0 $. Then by definition there exists a positive number $ \kappa>0 $ and a sequence $ x_n\to +\infty $ such that $ \int_\Omega u(x_n, y)dy \geq \kappa $. {By the estimate in} Lemma \ref{lem:localinf}, there exists $ \bar\kappa>0 $ such that for any $ n\in\mathbb N $, $ \inf_{y\in\Omega}u(x_n, y)\geq \bar \kappa $. Let $ \beta:=\max\left(2\frac{\sup a}{\kappa}, \beta_0\right) $, then a direct application of Lemma \ref{lem:infxneg} shows that for any $ n\in\mathbb N $, we have $ \inf_{x< x_n, y\in\Omega}u(x,y)>\rho_{\beta}>0 $. In particular, taking the limit $ n\to\infty $, we get $ \inf_{(x,y)\in\mathbb R\times \Omega} u(x,y) \geq \rho_{\beta}>0 $. {By the estimate in} Lemma \ref{lem:infpos}, this shows $ \inf_{(x,y)\in\mathbb R\times\Omega} u(x,y)\geq \rho_{\beta_0} $. However, {due to} the normalization satisfied by $ u $ \eqref{eq:norm}, we have
    \begin{equation*}
        k_0|\Omega|\rho_{\beta_0}\leq (K\star u)(x,0)\leq \frac{1}{2}k_0|\Omega|\rho_{\beta_0},
    \end{equation*}
    which is a contradiction. We conclude that $ \limsup_{x\to+\infty}\int_{\Omega}u(x,y)dy = 0 $.

    \textbf{Step 2:} We show that $ u $ satisfies the other properties required by Theorem \ref{thm:reg-TW}.

    Since $ u $ is given by Corollary \ref{cor:existenceline}, $u$ naturally satisfies $ \int_{\Omega} u(x,y)dy\leq \frac{\sup a}{k_0} $.

	Let us show briefly that   $ \liminf_{x\to-\infty}\inf_{y\in\Omega}u(x,y)\geq \rho_{\beta_0} $. {Applying} Lemma \ref{lem:infxneg} we have $ \liminf_{x\to-\infty}\inf_{y\in\Omega}u(x,y)>0 $.  Let $ (x_n, y_n) $ be a minimizing sequence. {By} the classical elliptic estimates, $ u(x+x_n, \cdot) $ converges locally uniformly to a solution $ \bar u $ of \eqref{eq:lineeps} with  $ \inf_{(x,y)\in\mathbb R\times\Omega} \bar u(x,y) > 0 $. Then {by the estimate in} Lemma \ref{lem:infpos}, $ \inf_{(x,y)\in\mathbb R\times\Omega} \bar u(x,y)\geq \rho_{\beta_0} $. We conclude by remarking that $ \liminf_{x\to-\infty}\inf_{y\in\Omega}u(x,y)=\inf_{(x,y)\in\mathbb R}\bar u(x,y)\geq \rho_{\beta_0} $.

    We finally remark that Lemma \ref{lem:c*-beta} item \ref{item:c*-beta-int} gives the minimality property of the speed $ c_\varepsilon^* $. In particular $ c=c_\varepsilon^* $ for the solution $(c,u) $ constructed here.

    This ends the proof of Theorem \ref{thm:reg-TW}.
\end{proof}

\medskip

Next we prove an upper estimate on the limit of $ \int_{\Omega} u(x,y)dy $ when $ x $ is in the vicinity of $ +\infty $, which is independent of $ \varepsilon $. 

\begin{lem}[$\int_\Omega u(x,y)dy \to 0 $ when $ x\to +\infty$]\label{lem:intto0}
    Let Assumption \ref{hyp:gen} hold, and suppose $ \lambda_1 <0 $. There exists $ \bar \varepsilon>0 $,  $ \tau>0 $  and a sequence $ (x_n)_{n\in\mathbb N} $ independent from $ \varepsilon $,  such that if  $ u $ solves \eqref{eq:lineeps} with $0 < \varepsilon \leq \bar \varepsilon $, $ c=c_\varepsilon^* $ and satisfies $ \int_\Omega u(x,z)dz\leq\tau $ for any $ x\geq 0 $, then
    \begin{equation*}
        \forall n\in\mathbb N, \forall x\geq x_n, \int_\Omega u(x, z)dz\leq\frac{\tau}{2^n}. 
    \end{equation*}
\end{lem}
\begin{proof}
    We divide the proof into three  steps. 

    \textbf{Step 1:} Definition of auxiliary parameters.

	Since $ a(0)=\sup a $, by the continuity of $ a $ and {Assumption \ref{hyp:gen} item 6}, there exists $ r>0 $ such that for any $ |y|\leq r $, $ a(y)-\mu\geq \frac{3}{4}(\sup a-\mu) $. In the rest of the proof we fix $ r>0 $ such that this property holds and $ B_r(y)\subset \Omega $. Notice that for $ |y|\leq r $, we have $ a(y)-\mu\geq \frac{3}{4}(\sup a-\mu) >0$. 

    We define $ \bar\varepsilon:=\min\left(\varepsilon_0, \frac{r^2(\sup a-\mu)}{2n\pi^2}\right) $, where $ \varepsilon_0>0 $ is given by Lemma \ref{lem:beta-lowerbound}. We let $ \tau:=\min\left(\frac{1}{2}\rho_{\beta_0}k_0|\Omega|, \frac{\sup a-\mu}{4k_\infty}\right) $, where $ \beta_0=\frac{k_\infty \sup a}{\mu m_0} $ and $ \rho_{\beta_0} $ is given by Lemma \ref{lem:beta-lowerbound}. In particular, arguing as in the proof of Theorem \ref{thm:reg-TW}, any solution $u$ to \eqref{eq:lineeps} with $ 0<\varepsilon\leq \bar \varepsilon $ which satisfies $ \int_{\Omega}u(0,y)dy\leq \tau $ has limit 0 near $ +\infty $, i.e. $ \int_\Omega u(x,y)dy\to_{x\to+\infty} 0 $.

	{By} Lemma \ref{lem:localinf} and \ref{lem:infxneg}, there exists $ \rho>0$ such that if  $ \int_{\Omega}u(x,y)dy\geq \frac{\tau}{2} $ holds, then for any $x'\leq x $ we have the estimate $\inf_{y\in\Omega}u(x',y)\geq \rho $. 

	We let  $ \alpha_0:=\max\left(\frac{\tau}{\int_{|y|\leq r}\cos\left(\frac{\pi |y|}{2r}\right)dy}, 2\rho\right) $,  $ \gamma:=\min\left(1,\left( \frac{\sup a-\mu}{8(c_\varepsilon^*\sqrt{\alpha_0}+1)}\rho\right)^2\right) $. Notice in particular that $ 2c_\varepsilon^*\sqrt{\gamma}{\alpha_0} + 2\gamma - \frac{\sup a-\mu}{4} \rho\leq 0 $. Finally we define  $ \bar x:=\sqrt{\frac{\alpha_0}{\gamma} } $. Remark that, since $ c_\varepsilon^*\to 2\sqrt{-\lambda_1}>0 $ when $ \varepsilon \to 0 $ ({by} Theorem \ref{thm:eigen}), $ \bar x $ is uniformly bounded when $ \varepsilon\to 0$.

	Since \eqref{eq:lineeps} is invariant by translation in $ x $ we will assume {without loss of generality} that $ \int_\Omega u(x,y)dy\leq \tau $ for $ x\geq -\bar x $ instead of $ x\geq 0 $.

    \medskip

    \textbf{Step 2:} We show that if $ \int_\Omega u(x, y)dy\leq\tau $ for $ x\geq -\bar x $ then  $ \int_\Omega u(\bar x, y)dy\leq \frac{\tau}{2} $. 

    Here we let $ u $ be a solution to \eqref{eq:lineeps} with $ 0<\varepsilon\leq \bar\varepsilon $,  $ c=c_\varepsilon^*$ and the upper estimate $ \int_\Omega u(x,y)\leq \tau $ for $ x\geq -\bar x $. We assume by contradiction that $ \int_\Omega u(\bar x, y) dy > \frac{\tau}{2} $. We will first use another proof by contradiction to show that, in that case, the mass of $ u $ can be controlled from below.

    Since $ u>0 $ on  $(-\bar x, \bar x)\times \Omega $, we define:
    \begin{equation*}
        \alpha:=\sup\left\{\zeta>0\,|\, \forall x\in(-\bar x, \bar x), \forall |y|\leq r, \quad (\zeta-\gamma x^2) \cos\left(\frac{\pi|y|}{2r}\right)\leq u(x,y) \right\}.
    \end{equation*}
    Assume by contradiction that $ \alpha < \alpha_0 $. Then for any $(x,y)\in [-\bar x, \bar x]\times \Omega $ we have $ (\alpha-\gamma x^2)\cos\left(\frac{\pi |y|}{2r}\right)\leq u(x,y) $,   and there exists a point $ x_0\in [-\bar x, \bar x] $ and $ y_0 $ with $ |y_0|\leq r $ such that $ u(x_0,y_0)=(\alpha-\gamma x_0^2)\cos\left(\frac{\pi |y_0|}{2r}\right) $. Let $ v:=(\alpha-\gamma x^2)\cos\left(\frac{\pi |y|}{2r}\right) $.  We have:
    \begin{align*}
        0&\leq -c_\varepsilon^*(v-u)_x(x_0, y_0) - (v-u)_{xx}(x_0, y_0) -\varepsilon \Delta_y (v-u)(x_0, y_0) \\
         &= 2c_\varepsilon^* \gamma x_0 + 2\gamma + n \varepsilon \left(\frac{\pi}{2r}\right)^2 v(x_0, y_0) - \mu (M\star u)(x_0, y_0) \\
         &\quad - u(x_0, y_0)\big(a(y_0)-\mu - (K\star u)(x_0, y_0)\big)\\
         &< 2\left(c_\varepsilon^*\sqrt{\alpha_0} + \sqrt{\gamma}\right)\sqrt{\gamma} +n\varepsilon\left(\frac{\pi}{2r}\right)^2 + 0 \\
         &\quad - \left(\frac{3(\sup a-\mu)}{4}-k_\infty\int_\Omega u(x_0, z)dz\right)u(x_0, y_0) \\
         &= \left[2\left(c_\varepsilon^*\sqrt{\alpha_0} + \sqrt{\gamma}\right)\sqrt{\gamma}-\frac{\sup a-\mu}{4}\rho\right] + \left(\varepsilon \left(\frac{\pi}{2r}\right)^2-\frac{\sup a-\mu}{4}\right) u(x_0, y_0) \\
         &\leq 0,
    \end{align*}
    recalling that $ \inf_y u(x_0, y) \geq \rho $ since $ x_0\leq \bar x$.

    Hence, we have a contradiction and $ \alpha\geq \alpha_0\geq\frac{\tau}{\int_{|y|\leq r}\cos\left(\frac{\pi |y|}{2r}\right)dy}$. In particular, we have $ (\alpha_0-\gamma x^2)\cos\left(\frac{\pi |y|}{2r}\right)\leq u(x,y) $ and 
    \begin{equation*}
        \tau\leq \alpha_0\int_{|y|\leq r}\cos\left(\frac{\pi |y|}{2r}\right)dy < \int_{\Omega}u(0,y)dy, 
    \end{equation*}
	{where the strict inequality holds because $u(0, y)>0$ on $\Omega\bsl B(0,r)$.}
    This contradicts our hypothesis $ \int_\Omega u(x, y)dy\leq \tau $ when $ x\geq -\bar x $. We conclude that $ \int_{\Omega}u(\bar x, y)dy\leq \frac{\tau}{2} $.

    \medskip

    \textbf{Step 3:} Bootstrapping

    In Step 2 we have shown that for a $ \bar x $ which is uniformly bounded in $ \varepsilon $, we have 
    \begin{equation*}
        \left(\forall x\geq -\bar x, \int_\Omega u(x,y)dy\leq\tau\right)\Rightarrow \left(\int_\Omega u(\bar x, y)dy\leq \frac{\tau}{2}\right).
    \end{equation*}
    Since \eqref{eq:lineeps} is invariant by translation, this implication still holds for $ u(x,y) $ replaced by $ u(x+\delta, y) $ for any $ \delta>0 $. In particular, 
    \begin{equation*}
        \left(\forall x\geq -\bar x, \int_\Omega u(x,y)dy\leq\tau\right)\Rightarrow \left(\forall x\geq \bar x, \int_\Omega u( x, y)dy\leq \frac{\tau}{2}\right).
    \end{equation*}
    Thus we can reproduce Step 1 and 2 replacing $ \tau $ by $ \frac{\tau}{2} $ and $ u(x,y) $ by its shift $ u(x+\bar x, y) $. We thus find by an elementary recursion a sequence of points $ x_n $ such that for $ x\geq x_n $, $ \int_\Omega u(x, y)dy \leq \frac{\tau}{2^n} $. 

    This ends the proof of Lemma \ref{lem:intto0}.
\end{proof}

\subsection{Proof of Theorem \ref{thm:TW}}
We are now in a position to let $ \varepsilon\to 0 $ and construct a traveling wave for equation \eqref{eq:evol}, thus proving our main result Theorem \ref{thm:TW}.

\begin{proof}[Proof of Theorem \ref{thm:TW}]
    We divide the proof in three steps. 

    \textbf{Step 1:} Construction of a converging sequence to a transition kernel.

	Let $ \varepsilon_n $ be a decreasing sequence with $ \lim \varepsilon_n = 0 $ and $ \varepsilon_0\leq \bar\varepsilon $ (where $ \bar\varepsilon $ is given by Lemma \ref{lem:intto0}) such that for any $ 0<\varepsilon\leq \varepsilon_0 $, $ \lambda_1^\varepsilon <0 $ (such a $ \varepsilon_0 $ exists {by} Theorem \ref{thm:eigen}). Since \eqref{eq:lineeps} is invariant by translations in $ x$, for each $ \varepsilon_n $ we can choose $ u^n $ given by Theorem \ref{thm:reg-TW} (with $ \varepsilon=\varepsilon_n $), which satisfies moreover
    \begin{equation}\label{eqthm:epsto0-norm}
        \int_\Omega u^n(0, y)dy = \min\left(\frac{\rho_{\beta_0}}{2}, \tau\right), \qquad \forall x\geq 0, \int_\Omega u^n(x, y)dy\leq \tau,
    \end{equation}
    where $ \tau $ is given by Lemma \ref{lem:intto0}, $ \beta_0=\frac{k_\infty \sup a}{\mu m_0} $ and $ \rho_{\beta_0} $ is given by Lemma \ref{lem:beta-lowerbound}. 

	For any $ k\leq n $, let $ u^n_k $ be the restriction of $ u^n $ to the set $ [-k, k]\times\Omega $. Then $ u^n_k$ belongs to $ M^1([-k, k]\times\overline\Omega)=(C_b([-k, k]\times\overline\Omega))^* $. Since $ \int_\Omega u^n(x,y)dy\leq \frac{\sup a}{k_0} $ for $ x\in\mathbb R $, we have $ \int_{-k}^k \int_\Omega u^n(x,y)dydx\leq 2k\frac{\sup a}{k_0} $, and thus the sequence $ (u_k^n)_{n>k} $ is uniformly bounded in  variation norm. Moreover $ [-k, k] \times\overline\Omega$ is compact, and thus $ (u_k^n)_{n>k} $ is uniformly tight. {Applying} Prokhorov's Theorem \cite[Theorem 8.6.2]{Bog-07}, the sequence $ (u_k^n)_{n>k} $ is relatively compact in $ (C_b([-k, k]\times\overline\Omega))^* $. Then, {by a} classical diagonal extraction process, there exists a subsequence, still denoted  $ u^n $, and a measure $ u\in M^1(\mathbb R\times\overline\Omega) $ such that $ u^n\rightharpoonup u $, in the sense that 
    \begin{equation}\label{eqthm:TW-weak-conv}
        \forall \psi\in C_c(\mathbb R\times\overline\Omega), \int_{\mathbb R\times\Omega}\psi(x,y)u^n(x,y)dydx \to \int_{\mathbb R\times\overline\Omega} \psi(x,y) u(dx,dy).
    \end{equation}

	Finally, for $ a<b $, {by a classical result}  \cite[Theorem 8.2.3]{Bog-07}, for any Borel set $ \omega\subset\overline\Omega $ we have
    \begin{equation*}
        u\big((a,b)\times\omega\big)\leq u\big((a,b)\times\overline\Omega\big)\leq \liminf_{n\to\infty} \int_a^b \int_{\Omega} u^n(x,y)dydx\leq |b-a|\frac{\sup a}{k_0}.
    \end{equation*}
    Hence, Lemma \ref{lem:transition} applies and $ u $ is a transition kernel,  satisfying the equation $ u(dx, dy)=u(x,dy) dx$.

	{Let us stress at this point that the possibility to think of $u$ as a transition kernel, i.e. a function which takes values in a measure space, is important for the rest of the proof, as it allows us to consider $M\star u(x,y)=\int_{\overline \Omega}M(y,z) u(x, dz)$ and $K\star u(x,y)=\int_{\overline \Omega}K(y,z) u(x, dz)$ as real functions of $x$ and $y$, even for singular traveling waves. Handling a term like $\int_{\overline\Omega}M(y,z)u(dx, dy) $ would indeed be quite difficult, if ever possible -- let aside $ (K\star u)u$, which would involve the product of two measures. Also, it is the only regularity that we can get on the solution at the present time.}
    \medskip

    \textbf{Step 2:} We show that $ u $ satisfies the limit conditions \eqref{eq:limit-TW-left} and \eqref{eq:limit-TW-right} of Definition \ref{def:TW}.

    By construction, the function $ u^n $ satisfies $ \int_\Omega u^n(0,y)dy = \min\left(\tau,\frac{\rho_{\beta_0}}{2}\right) $. Applying Lemma \ref{lem:localinf} and Lemma \ref{lem:infxneg}, there exists a positive constant $ \rho>0 $ (independent from $ n $) such that $ \inf_{y\in\Omega}u^n(x,y) \geq \rho $ for any $ x\leq 0 $. In particular, taking the limit $ n\to\infty$, we have for any positive $ \psi\in C_c\big((-\infty, 0)\times\overline\Omega\big) $
    \begin{align*}
        \int_{\mathbb R\times\overline\Omega}\psi(x,y)u(x,dy)dx &=\lim_{n\to\infty}\int_{\mathbb R\times\Omega}\psi(x,y)u^n(x,y)dxdy \\
                                                                &\geq \rho \int_{\mathbb R\times\Omega}\psi(x,y)dxdy>0.  
    \end{align*}
    Hence $\liminf_{\bar x\to+\infty}\int_{\mathbb R\times\overline\Omega} \psi(x+\bar x, y)u(x,dy)dx \geq \rho\int_{\mathbb R\times\Omega}\psi(x,y) dxdy>0 $, and  $ u $ satisfies \eqref{eq:limit-TW-left}. 

	Let us show that $ u $ satisfies \eqref{eq:limit-TW-right}, i.e. vanishes near $ +\infty$. {Applying} Lemma \ref{lem:intto0}, there exists a sequence $ x_k $ independent from $ n $ such that we have $ \int_{\Omega}u^n(x,y)dy\leq\frac{\tau}{2^k} $ for any $ x\geq x_k $. In particular for any positive $ \psi\in C_c\big((x_k, +\infty)\times\overline\Omega\big) $, we have
    \begin{align*}
        \int_{\mathbb R\times\overline\Omega}\psi(x-x_k,y)u(x,dy)dx& =\lim_{n\to\infty}\int_{\mathbb R\times\Omega}\psi(x-x_k,y)u^n(x,y)dxdy \\
                                                                   &\leq \frac{\tau}{2^k} \mathrm{~diam~supp}~\psi\sup_{(x,y)\in\mathbb R\times\Omega} \psi(x,y),
    \end{align*}
    where $ \mathrm{diam~supp~}\psi = \sup\left\{d\big((x,y), (x', y')\big)\,|\, \psi(x,y)>0 \text{ and } \psi(x',y')>0 \right\} $ is the diameter of the support of $ \psi $. Thus
    \begin{equation*}
        \limsup_{\bar x\to +\infty} \int_{\mathbb R\times\overline\Omega} \psi(x-\bar x,y)u(x,dy)dx = \limsup_{k\to+\infty}\int_{\mathbb R\times\overline\Omega}\psi(x-x_k, y)u(x,dy)dx = 0, 
    \end{equation*}
    and $ u $ satisfies indeed  \eqref{eq:limit-TW-right}.

    Let us stress that, since $ u $ satisfies \eqref{eq:limit-TW-left} and \eqref{eq:limit-TW-right}, $ u $ is neither 0 nor a nontrivial stationary state to \eqref{eq:evol}.

    \textbf{Step 3:} We show that $ u $ satisfies \eqref{eq:TW} in the sense of distributions. 

    Let $ F_0^y:=\left\{\psi\in C^2_c(\mathbb R\times\overline\Omega)\,|\, \forall x\in\mathbb R, \forall y\in \partial\Omega, \frac{\partial \psi}{\partial\nu}(x,y)=0\right\} $ as in Lemma \ref{lem:zerofluxdense}. We fix $ \psi\in F_0^y $. Our goal here is to show that 
    \begin{align}
        \notag  &\quad c^*\int_{\mathbb R\times\overline\Omega}\psi_{x}u(x,dy)dx -\int_{\mathbb R\times\overline\Omega}\psi_{xx}u(x,dy)dx\\
        \notag  &= \int_{\mathbb R\times\overline\Omega}\int_{\overline\Omega}M(y,z)u(x,dz)\psi(x,y)dxdy +\int_{\mathbb R\times\overline\Omega} (a(y)-\mu)\psi(x,y)u(x,dy)dx \\
                &\quad \label{eq:weak-TW}        -\int_{\mathbb R\times\overline\Omega}\int_{\overline\Omega} \psi(x,y)K(y,z)u(x,dz)u(x,dy)dx,
    \end{align}
    where $ c^*=2\sqrt{-\lambda_1} $.
    Multiplying \eqref{eq:lineeps} by $ \psi $ and integrating by parts, we have
    \begin{align}
        \notag   &\quad c_{\varepsilon_n}^*\int_{\mathbb R\times\Omega} \psi_x(x,y)u^n(x,y)dxdy - \int_{\mathbb R\times\Omega}\psi_{xx}(x,y)u^n(x,y)dxdy  \\
        \notag   &=\varepsilon_n\int_{\mathbb R\times\overline\Omega} \Delta_y\psi(x,y)u^n(x,y)dxdy + \int_{\mathbb R\times\Omega} (a(y)-\mu)\psi(x,y)u^n(x,y)dxdy \\
        \notag   &\quad + \int_{\mathbb R\times\Omega} \psi(x,y) \int_\Omega M(y,z)u^n(x,z)dz dxdy \\
		 &\quad\label{eq:deg-TW}       -\int_{\mathbb R\times\Omega} \psi(x,y)\int_\Omega K(y,z){u^n(x,z)}dz u^n(x,y)dxdy.
    \end{align}
    Clearly, the difficulty here resides in the last two lines of equation \eqref{eq:deg-TW} (recall the formula  $ c^*_\varepsilon=2\sqrt{-\lambda_1^\varepsilon}\underset{\varepsilon\to 0}{\to} 2\sqrt{-\lambda_1}= c^*$). Let us focus on those.

    \medskip

    {We first remark that 
    \begin{multline*}
	    \int_{\mathbb R\times\overline\Omega} \psi(x,y)(M\star u^n)(x,y) dxdy=\int_{\mathbb R\times\overline\Omega} \check M\star \psi(x,z) u^n(x,z) dxdz \\
	    \underset{n\to\infty}{\longrightarrow} \int_{\mathbb R\times\overline\Omega} \check M\star \psi(x,z) u(x,dz) dx =\int_{\mathbb R\times\overline\Omega}\psi(x,y)\int_{\overline\Omega}M(y,z) u(x,dz) dxdy,
    \end{multline*}
    where $\check M(y,z)=M(z,y)$, since $\check M\star \psi(x, y)$ is a valid test function.
    }

    \medskip

    The convergence of the nonlinear term requires more work. For $ i\in\mathbb N $, let $ K^i(y,z)\in F_0^2 $ be such that $ \Vert K-K^i\Vert_{C_b(\overline\Omega\times\overline\Omega)}\leq \frac{1}{i} $ and $ \Vert K^i\Vert_{C^\alpha(\overline\Omega\times\overline\Omega)} \leq C $, where $ F_0^2 $ is the set of smooth kernels with null boundary flux in $z$, and $ C $ is independent from $ i $ (see Lemma \ref{lem:zerofluxdense} item \ref{item:zeroflux-Omega2}). We want to complete, up to extractions, the informal diagram
    \begin{equation*}
        \begin{array}{ccc}
            v_i^n(x,y):=\int_{\Omega} K^i(y,z) u^n(x,z)dz & \overset{?}{\underset{n\to +\infty}{\longrightarrow}} & v_i(x,y) := \int_{\overline\Omega} K^i(y,z) u(x,z)dz \\
            \downarrow~ \scriptstyle{ i\to\infty} & & \downarrow~ \scriptstyle{ i \to \infty} \\
            v^n(x,y):=\int_{\Omega} K(y,z) u^n(x,z)dz & \overset{?}{\underset{n\to +\infty}{\longrightarrow}} & v(x,y):=\int_{\overline\Omega} K(y,z) u(x,z)dz .
        \end{array}
    \end{equation*}

    We first show that $ v^n_i(x,y)\to v_i(x,y) $ when $ n\to\infty $ in $ C_b\big([-R,R]\times\overline\Omega\big) $ for arbitrary $ R>0 $. We fix $ R $ so that $ \mathrm{supp}~\psi\subset [-R, R]\times\overline\Omega $.
    Substituting $ z $ to $ y $, multiplying equation \eqref{eq:lineeps} by $ K^i $ and integrating in $ z $, we have
    \begin{equation*}
        -c_{\varepsilon_n}^*({v^n_i})_x-({v^n_i})_{xx}=R^n(x,y)
    \end{equation*}
    where  $ R^n(x,y) $ is bounded in $ L^\infty $ uniformly in $ n $:
    \begin{align*}
        |R^n(x,y)|&=\bigg|\varepsilon_n\int_\Omega \Delta_z{K^i}(y,z) u^n(x,z)dz+\mu\int_\Omega {K^i}(y,z) (M\star u^n)(x,z) dz\\
                  &\quad+\int_{\Omega}{K^i}(y,z)(a(z)-\mu-K\star u^n)u^n(x,z)dz\bigg|\\
                  &\leq \varepsilon_n \Vert {K^i}\Vert_{C_b(\overline\Omega, C^2(\overline\Omega))}\frac{\sup a}{k_0} + \mu m_\infty |\Omega| \frac{\sup a}{k_0} \Vert {K^i}\Vert_{C_b(\overline\Omega\times\overline\Omega)} \\
                  &\quad + \left(\sup a+\mu+k_\infty\frac{\sup a}{k_0}\right)\frac{\sup a}{k_0}\Vert{K^i}\Vert_{C_b(\overline\Omega\times\overline\Omega)}.
    \end{align*}
    For $ n $ large enough so that $ \varepsilon_n \leq \frac{1}{\Vert {K^i}\Vert_{C_b(\overline\Omega, C^2(\overline\Omega))}} $, {by the estimate in} \cite[Theorem 9.11]{Gil-Tru-01} and the classical Sobolev embeddings,  $ \Vert v^n_i(\cdot, y)\Vert_{C^\alpha([-R,R])} $ is uniformly bounded by a constant independent from $ n $, $i$ and $ y\in\overline\Omega $. Since we have $ K^i\in C^\alpha(\overline\Omega\times\overline\Omega) $ uniformly in $ i $, $ v^n_i $ is then uniformly Hölder in $ x$ and  $ y $ and  we have $ \Vert v^n_i\Vert_{C^\alpha([-R,R]\times\overline\Omega)}\leq C_R $ with $ C_R $ independent from $ n $ and $ i $. In particular, there exists an extraction $ \varphi^i(n) $ such that 

    -- $ \Vert v^{\varphi^i(n)}_i\Vert_{C^\alpha([-R,R]\times\overline\Omega)}\leq C_R $, and

    --$ \Vert v^{\varphi^i(n)}_i - \tilde v_i \Vert_{C^{\alpha /2}([-R,R]\times\overline\Omega)}\to_{n\to\infty} 0 $,

    \noindent for a function $ \tilde v_i(x,z)\in C^{\alpha/2}\big([-R, R]\times\overline\Omega\big) $. Notice that we can assume {without loss of generality} that $ \varphi^i(n) $ is extracted from $ \varphi^{i-1}(n) $. Finally, for any test function $ \xi(x)\in C_c\big([-R, R]\big) $, we have 
    \begin{multline*}
        \int_{-R}^R \xi(x)v_i^{\varphi^i(n)}(x,y)dx = \int_{-R}^R \int_\Omega \xi(x)K^i(y,z)u^{\varphi^i(n)}(x,z)dzdx \\
        \to_{n\to\infty}\int_{-R}^R \int_{\overline\Omega} \xi(x)K^i(y,z)u(x, dz)dx = \int_{-R}^R \xi(x)v_i(x, z)dx,
    \end{multline*}
    since $ u^n $ converges to $ u $ in the sense of measures. This shows  $ \tilde v^i(x,y) = v_i(x,y) $ for almost every $ x\in [-R, R] $. 

    Moreover since $ \Vert v_i\Vert_{C^{\alpha /2}([-R, R]\times\overline\Omega)}\leq C'_R $, there exists an extraction $ \zeta $ such that  $ v_{\zeta(i)} $  converges in $ C_b\big([-R, R]\times\overline\Omega\big) $ to $ v(x,y)=\int_{\overline\Omega} K(y,z)u(x,dy) $, which shows a $C^0 $ regularity on $ v$.

    We can then construct an extraction $ \varphi(i) $  such that 

    -- $ \Vert v^{\varphi(i)}_{\zeta(i)} -  v_{\zeta(i)} \Vert_{C_b([-R,R]\times\overline\Omega)}\to_{i\to\infty} 0 $, and 

    -- $ \Vert v_{\zeta(i)}-v\Vert_{C_b([-R,R]\times\overline\Omega)}\to_{i\to\infty} 0$.

    \noindent Along this subsequence, we have then:
    \begin{align*}
        &\quad \left|\int_\Omega K(y,z)u^{\varphi(i)}(x, y) dx dy - \int_{\overline\Omega} K(y,z)u(x,dy) \right|\\ 
        &\leq\left|\int_{\Omega} \big(K(y,z)-K^{\zeta(n)}(y,z)\big) u^{\varphi(i)}(x,y)dxdy \right| + \left\Vert v_{\zeta(i)}^{\varphi(i)} - v_{\zeta(i)} \right\Vert_{C_b([-R, R]\times\overline\Omega)} \\
        &\quad+ \Vert v_{\zeta(i)}-v\Vert_{C_b([-R,R]\times\overline\Omega)}\\
        &\leq \Vert K - K^{\zeta(i)}\Vert_{C_b(\overline\Omega\times\overline\Omega)}\frac{\sup a}{k_0} + o_{i\to\infty}(1)
    \end{align*}
    which shows that $ \int_\Omega K(y,z)u^{\varphi(i)}(x,y)dy \to \int_{\overline\Omega}K(y,z)u(x,dy) $ in   $ C_b([-R,R]\times\overline\Omega) $.

    We are now in a position to handle the nonlinear term, {by using} the previously constructed subsequence. 
    We write 
    \begin{align*}
        &\quad \int_{\mathbb R\times\Omega\times\Omega} \psi(x,y)K(y,z)u^{\varphi(n)}(x,z)u^{\varphi(n)} (x,y)dxdydz  \\
        &=\int_{\mathbb R\times\Omega} \psi(x,y)\int_{\overline\Omega}K(y,z)u(x,dz) u^{\varphi(n)}(x,y) dxdy   \\
        &\quad+\int_{\mathbb R\times\Omega} \psi(x,y)\bigg(\int_{\Omega}K(y,z)u^{\varphi(n)}(x,z)dz - \int_{\overline\Omega}K(y,z)u(x,dz)\bigg)\\
        &\quad\times u^{\varphi(n)}(x,y)dxdy  \\ 
        &=\int_{\mathbb R\times\Omega} \psi(x,y)\int_{\overline\Omega}K(y,z)u(x,dz) u^{\varphi(n)}(x,y) dxdy\\
        &\quad+ {O}\left(\Vert v^{\varphi(n)}(x,y)-v(x,y)\Vert_{C_b([-R, R]\times\overline\Omega)}\right),
    \end{align*}
    where  $ v^{\varphi(n)}(x,y)=\int_\Omega K(y,z) u^{\varphi(n)}(x,z)dz $ and $ v(x,y)=\int_{\overline\Omega} K(y,z)u(x,dz) $.
    Since $ \psi(x,y)\int_{\overline\Omega}K(y,z)u(x,dz) $ is a continuous, compactly supported function, we have shown that 
    \begin{multline*}
        \int_{\mathbb R\times\Omega\times\Omega} \psi(x,y)K(y,z)u^{\varphi(n)}(x,z)u^{\varphi(n)} (x,y)dxdydz \\
        \to_{n\to\infty} \int_{\mathbb R\times\overline\Omega\times\overline\Omega} \psi(x,y)K(y,z)u(x, dz)u(x,dy)dx. 
    \end{multline*}

    \medskip

    Finally we can take the limit in \eqref{eq:deg-TW} along the subsequence $ \varphi(n)$. This shows that $ u $ satisfies  \eqref{eq:TW} in a weak sense, where the test functions are taken in $ F_0^y $. Since $ F_0^y $ is dense in $ C^2_c\big(\mathbb R, C_b(\overline\Omega)\big)$, equation \eqref{eq:weak-TW} holds for test functions $ \psi $ taken in $ C^2_c(\mathbb R, C_b(\overline\Omega)) $. In particular, $ u $ satisfies \eqref{eq:TW} in the sense of distributions. 

    This ends the proof of Theorem \ref{thm:TW}.
\end{proof}

\section*{Acknowledgements}

The author wishes to thank M. Alfaro, H. Matano and G. Raoul for fruitful discussions, as well as two anonymous reviewers who contributed greatly to the quality of the paper.

\addcontentsline{toc}{section}{Appendix}
\appendix 

\section{Appendix}

\subsection{Density of the space of functions with null boundary flux }

Here we prove elementary results that are crucial to our proofs of Theorem \ref{thm:eigen}, Theorem \ref{thm:survival} and Theorem \ref{thm:TW}.

\begin{lem}[Density of spaces of functions with null boundary flux]\label{lem:zerofluxdense}
    Let $ \Omega \subset \mathbb R^n $ be a bounded open set with $ C^3 $ boundary. 
    \begin{enumerate}[label=(\textit{\roman*}), ref=(\textit{\roman*})]
        \item \label{item:zeroflux-Omega} The function space
            \begin{equation*}
                F_0:=\left\{\psi\in C^2(\overline\Omega)\,|\, \forall y\in\partial\Omega, \frac{\partial \psi}{\partial\nu}(y)=0 \right\}
            \end{equation*}
            is dense in $ C_b(\overline\Omega) $.
        \item \label{item:zeroflux-ROmega}  The function space 
            \begin{equation*}
                F_0^y:=\left\{\psi\in C^2_c(\mathbb R\times\overline\Omega)\,|\, \forall x\in\mathbb R, \forall y\in \partial\Omega, \frac{\partial \psi}{\partial\nu}(x,y)=0\right\}
            \end{equation*}
            is dense in $ C^2_c(\mathbb R, C_b(\overline\Omega))$.
        \item \label{item:zeroflux-Omega2} The function space 
            \begin{equation*}
                F_0^2:=\left\{\psi\in C^2(\overline\Omega\times\overline\Omega)\,|\, \forall (y, z)\in\overline\Omega\times\partial\Omega, \frac{\partial \psi}{\partial\nu_z}(y, z)=0 \right\}
            \end{equation*}
            is dense in $ C_b(\overline\Omega\times\overline\Omega) $. Moreover for any $ \alpha\in (0,1) $ and any function $ \psi\in C^2(\overline\Omega\times\overline\Omega) $ there exists a constant $ C $ and  a sequence $ \psi^r\to\psi $ such that we have  $ \Vert \psi^r\Vert_{C^\alpha(\overline\Omega\times\overline\Omega)}\leq C\Vert\psi\Vert_{C^\alpha(\overline\Omega\times\overline\Omega)} $.
    \end{enumerate}
\end{lem}
\begin{proof}
    Let us denote $ d(y) :=\inf_{z\in\partial\Omega}|y-z|$ the distance function.
	We recall that there exists $ R>0 $ such that $ y\mapsto d(y,\partial\Omega) $ is $ C^3 $ in the tubular neighbourhood $\Omega_R:=\{y\in\Omega\,|\, d(y,\partial\Omega) < R\} $ \cite{Foo-84}. We fix a smooth function $ \theta:\mathbb R\to\mathbb R $ such that $\theta(x)=0$ for $ x\leq 0 $, $\theta(1)=1$ for $ x\geq 1 $, and  $ \forall k>0 $, $ \theta^{(k)}(0)=\theta^{(k)}(1)=0 $.
    Finally for $ y\in\Omega $, we let $ P(y) $ be the projection of $ y $ on $ \partial\Omega $, which is well-defined and $ C^2$ on $ \Omega_R $.
    \medskip

    With these notations, establishing Item \ref{item:zeroflux-Omega} and \ref{item:zeroflux-ROmega} is elementary by considering (for $0<r<R $)  $\psi^r(y):=  \left(1-\theta\left(\frac{d(y)}{r}\right)\right)\psi(P(y)) + \theta\left(\frac{d(y)}{r}\right)\psi(y) $ and similarly $\psi^r(x,y):=\left(1-\theta\left(\frac{d(y)}{r}\right)\right)\psi\big(x,P(y)\big) + \theta\left(\frac{d(y)}{r}\right)\psi(x,y)$ for a function $ \psi\in C^2(\overline\Omega) $ and $ \psi\in C^2(\mathbb R\times\overline\Omega) $, respectively. We turn to the proof of Item  \ref{item:zeroflux-Omega2}
    \medskip 

    Let $ \psi\in C^2(\overline\Omega\times\overline\Omega)$. Let $\psi^r(y,z):=\left(1-\theta\left(\frac{d(z)}{r}\right)\right)\psi\big(y,P(z)\big) + \theta\left(\frac{d(z)}{r}\right)\psi(y,z)$ for $ 0<r<\frac{R}{2} $.
    Clearly, $ \psi^r\in F_0^2 $ and $ \psi^r\to\psi $ in $ C_b(\Omega) $. Moreover, for each $ (y,z)\in\overline\Omega\times\overline\Omega $ we have  
    \begin{align*}
        &  \frac{\Big|\big(\psi^r(y,z)-\psi(y,z)\big)-\big(\psi^r(y,z')-\psi(y,z')\big)\Big|}{|z-z'|^\alpha} \\
        &\leq\frac{\left|\theta\left(\frac{d(z)}{r}\right)-\theta\left(\frac{d(z')}{r}\right)\right|}{|z-z'|^\alpha}\big(\psi(y, P(z))-\psi(y,z)\big) \\
        &\quad +\left|1-\theta\left(\frac{d(z')}{r}\right)\right|\left(\frac{|\psi(y, P(z))-\psi(y, P(z'))|}{|z-z'|^\alpha} + \frac{|\psi(y,z)-\psi(y,z')|}{|z-z'|^\alpha}\right) \\
        &\leq\Vert \theta\Vert_{C^\alpha([0, 1])} \Vert d\Vert_{C^{0,1}(\Omega_{R/2})}^\alpha \Vert \psi\Vert_{C^\alpha(\overline\Omega\times\overline\Omega)} + 2\Vert \psi\Vert_{C^\alpha(\overline\Omega\times\overline\Omega)}, 
    \end{align*}
    since $ |\psi(y, P(z))-\psi(y,z)|\leq r^\alpha\Vert \psi\Vert_{C^{\alpha}(\overline\Omega\times\overline\Omega)} $.
    This shows item \ref{item:zeroflux-Omega2}. 
\end{proof}

\subsection{A topological theorem}\label{appendix:topo}

For the sake of completeness, let us recall a result that we proved in a joint work with M. Alfaro \cite{Alf-Gri-18}, and that we use in the construction of stationary states. 

\begin{thm}[Bifurcation under Krein-Rutman assumption]\label{thm:orientedbifurcation}
    Let $E$ be a Banach space. Let $ C \subset E$ be a closed convex cone with nonempty interior $ Int\,C\neq\varnothing $ and of vertex 0, i.e.
    such that $ C\cap -C=\{0\}$. Let 
    $F $    be a continuous and compact operator $ \mathbb R\times E\longrightarrow E$.  Let us define
    \begin{equation*}
        \mathcal S:=\overline{\{(\alpha, x)\in \mathbb R\times E\bsl\{0\}\,|\,F(\alpha, x)=x\}}
    \end{equation*}
    the closure of the set of nontrivial fixed points of $ F$, and 
    \begin{equation*}
        \mathbb P_\mathbb R\mathcal S:=\{\alpha\in\mathbb R\,|\, \exists x\in C\bsl \{0\}, (\alpha, x)\in\mathcal S\}
    \end{equation*}
    the set of nontrivial solutions in $ C $.

    Let us assume the following.
    \begin{enumerate}
        \item \label{item:0isauniversalsolution} $ \forall \alpha\in\mathbb R$, $ F(\alpha, 0)=0 $. 
	\item \label{item:frechetdifferentiability} $ F $ is Fr\'echet differentiable near $ \mathbb R\times \{0\} $ with derivative $ \alpha T $ locally uniformly {with respect to} $ \alpha $, i.e. for any $ \alpha_1< \alpha_2 $ and $ \epsilon>0 $ there exists $ \delta>0 $ such that  
            \begin{equation*}
                \forall \alpha\in(\alpha_1, \alpha_2),\;  \Vert x\Vert\leq\delta\Rightarrow \Vert F(\alpha, x)-\alpha Tx\Vert\leq\epsilon \Vert x\Vert.
            \end{equation*}
        \item \label{item:kr} $ T $ satisfies the hypotheses of  the Krein-Rutman Theorem. We denote by $\lambda_1 (T)>0 $ its principal eigenvalue. 
	\item  \label{item:theobiflocalbound} $ \mathcal S\cap (\{\alpha\}\times C) $ is bounded locally uniformly {with respect to} $ \alpha\in\mathbb R $.
        \item \label{item:noescapethroughboundary} $ \mathcal S\cap \mathbb (\mathbb R\times(\partial C\backslash\{0\}))=\varnothing$, i.e. 
            there  is no fixed point on the boundary of $ C $.
    \end{enumerate}

    Then, either $ \left(-\infty, \frac{1}{\lambda_1(T)}\right)\subset\mathbb P_\mathbb R\mathcal S $ or $ \left(\frac{1}{\lambda_1(T)}, +\infty\right)\subset\mathbb P_\mathbb R\mathcal S$.
\end{thm}
The proof can be found in \cite{Alf-Gri-18}.

\subsection{Existence of a transition kernel}\label{annex:measure}

Our final lemma is crucial for the construction of traveling waves.
\begin{lem}[Existence of a transition kernel]\label{lem:transition}
    Let $ \Omega $ be an open domain  $ \Omega\subset\mathbb R^d $, and let $ \mu $ be a nonnegative measure defined on $ \mathcal B(\mathbb R\times \overline \Omega) $. Assume there exists a constant $ C\geq 0 $ such that 
    \begin{equation*}
        \forall a<b, \forall \omega\in\mathcal B(\overline\Omega), \mu([a,b]\times \omega)\leq C|b-a|.
    \end{equation*}
    Then there exists a function $ \nu:\mathbb R\times \mathcal B(\overline\Omega)\longrightarrow \mathbb R^+ $ such that 

    1. for almost every $ x\in\mathbb R $, $ \omega\mapsto \nu(x, \omega) $ is a nonnegative finite measure on $ \mathcal B(\overline\Omega) $

    2. for every $ \omega\in\mathcal B(\overline\Omega) $, $ x\mapsto \nu(x, \omega) $ is a Lebesgue-measurable function in $ L^1_{loc}(\mathbb R) $

    3. $    \mu(dx, dy)=\nu(x, dy)dx$,
    in the sense that 
    \begin{equation*}
        \forall \varphi\in C_c(\mathbb R\times\overline\Omega), \int_{\mathbb R\times\overline\Omega}\varphi(x,y)\mu(dx, dy)=\int_{\mathbb R\times \overline\Omega} \varphi(x,y)\nu(x, dy)dx.
    \end{equation*}

    Finally $ \nu $ is unique up to a Lebesgue negligible set, and satisfies
    \begin{equation*}
        \nu(x, \overline\Omega)\leq C\qquad a.e. 
    \end{equation*}
\end{lem}
\begin{proof}
    We divide the proof in four steps.

    \textbf{Step 1:} We construct a density for $ \mu(A\times\omega)$, $ \omega\in\mathcal B(\overline\Omega) $.

    Let us take $ \omega\in\mathcal B(\overline\Omega) $, and define $ A\in\mathcal B(\mathbb R)\overset{\mu_\omega}{\mapsto}\mu(A\times \omega) $. 
    Then $ \mu_\omega $ is a nonnegative Borel-regular measure on $ \mathcal B(\mathbb R) $. Indeed $ \mu_\omega $ is clearly well-defined on $ \mathcal B(\mathbb R) $, satisfies the $ \sigma $-additivity property and is finite on any compact set.
    Then, for any open set $ U\subset \mathbb R $, we have 
    $\mu_{\omega}(U)\leq C{\mathcal L}(U)$.
    Indeed we can write $ U=\underset{n\in\mathbb N}{\bigcup} K_n $ where $ K_n $ is an increasing sequence of compact sets of the form $ K_n=\underset{i=0}{\overset{m_n}{\bigsqcup}} [a_i^n, b_i^n] $ (with $ a_i^n<b_i^n<a_{i+1}^n ... $), for which the property holds by assumption. Thus
    \begin{equation*}
        \mu_\omega(U)=\lim_{n\to+\infty}\mu_{\omega}(K_n)\leq C\lim_{n\to+\infty}{\mathcal L}(K_n)=C\mathcal L(U) 
    \end{equation*}
    Finally, $ \mu_\omega\ll \mathcal L$, where $ \mathcal L $ is the Lebesgue measure on $ \mathbb R $. Indeed, let us take $ E\subset \mathbb R $ bounded such that $ \mathcal L(E) = 0$. Then by the regularity of $ \mu_\omega $ \cite[Theorem 2.18]{Rud-74}, we have 
    \begin{equation*}
        \mu_\omega(E)=\inf_{U~\mathrm{open}, U\supset E}\mu_{\omega}(U)\leq C\inf_{U~\mathrm{open}, U\supset E}{\mathcal L}(E)=0.
    \end{equation*}

	{Applying} the Radon-Nikodym Theorem \cite[Theorem 6.10]{Rud-74}, there exists then a unique measurable function $ h_\omega\in L^1_{loc}(\mathbb R ) $ such that 
    \begin{equation*}
        \mu_\omega = h_\omega {\mathcal L} = h_\omega dx.
    \end{equation*}

    \textbf{Step 2:} We show that the density $ h_\omega $ is well-defined up to a negligible set independent from $ \omega $.

    Let $ \omega_n $ be an enumeration of the sets of the form 
    \begin{equation*}
        \overline \Omega \cap \underset{i=1}{\overset{d}{\prod}}[a_i, b_i]
    \end{equation*}
    where $a_i, b_i\in\mathbb Q $. Clearly, $ \omega_n $ is stable by finite intersections, and the associated monotone class is $ \mathcal B(\overline\Omega) $. We let $ h_n:=h_{\omega_n} \in L^1_{loc}(\mathbb R) $ be the previously constructed density associated with $ \mu_{\omega_n} $. Then $ h_n $ is well-defined on a set $ \mathcal D_n $ satisfying $ {\mathcal L}(\mathbb R\bsl\mathcal D_n)=0 $. We let $ \mathcal D=\underset{n\in\mathbb N}{\bigcap}\mathcal D_n $, then $ {\mathcal L} (\mathbb R\bsl\mathcal D)=0 $ and  by construction, every $ h_n $ is well-defined on $ \mathcal D $.

    We take $\omega\in\mathcal B(\overline\Omega) $ and show that, up to a redefinition on a negligible set, the function $ h_\omega $  is well-defined on $ \mathcal D$. If $ \omega $ is open, then we can write 
    $\omega = \underset{n\in\mathbb N}{\bigsqcup} \omega'_n$
    for a well-chosen extraction $ \omega'_n $ of $ \omega_n $. Thus for any $ A\in\mathcal B(\mathbb R) $, we have the formula  
    $\mu_\omega(A)=\mu(A\times\omega)=\underset{n\in\mathbb N}{\sum}\mu(A\times\omega'_n)$
    and by the uniqueness in the Radon-Nikodym Theorem, we have:
    \begin{equation*}
        h_{\omega}=\underset{n\in\mathbb N}{\sum}h_{\omega'_n} \qquad {\mathcal L}-a.e.
    \end{equation*}
	In the general case  we have           $\mu(A\times\omega)=\inf_{U~\mathrm{open}, U\supset \omega}\mu(A\times U)$ for $ A\in\mathcal B(\mathbb R) $ {because of} the Borel regularity of $ \mu $,
    which shows that $ h_\omega $ is well-defined on $ \mathcal D $.
    \medskip

    \textbf{Step 3:} We verify that the constructed family of functions form a nonnegative measure on $ \overline\Omega $ for $\mathcal L$-a.e. $ x\in\mathbb R $.

    Let $ w_n\in\mathcal B(\overline\Omega) $ be a countable collection of Borel sets with $ w_i\cap w_j=\varnothing $ if $ i\neq j $.
    Then
    \begin{equation*}
        \mu(A\times \underset{n\in\mathbb N}{\bigsqcup}w_n)=\underset{n\in\mathbb N}{\sum}\mu(A\times w_n)
    \end{equation*}
	for any $ A\in\mathcal B(\mathbb R) $, and {by} the uniqueness in the Radon-Nikodym theorem we have
    \begin{equation*}
        h_{\underset{n\in\mathbb N}{\bigsqcup}w_n}=\underset{n\in\mathbb N}{\sum}h_{w_n}\qquad {\mathcal L}-a.e.
    \end{equation*}
    Thus, for any $ x\in\mathcal D $, the function $ \omega\mapsto h_\omega(x) $ is $ \sigma $-additive. Since $ h_\omega $ is nonnegative by construction,  $  \omega\mapsto h_\omega(x) $ is a nonnegative measure on $ \mathcal B(\overline\Omega) $. 

    We define $ \nu(x, \omega):= h_{\omega}(x) $. Then $ \nu $ matches the definition of a transition kernel (Definition \ref{def:TK}).
    \medskip

    \textbf{Step 4:} Conclusion.

    Since $ \nu(x,dy)dx $ coincides with $ \mu $ on the monotone class $ A\times\omega_n $, where $ A\in\mathcal B(\mathbb R) $, we have $ \mu(dx,dy)=\nu(x,dy)dx$ on $ \mathcal B(\mathbb R\times\overline\Omega) $.

    Finally, since $ x\mapsto \nu(x, \overline\Omega) $ is in $ L^1_{loc}(\mathbb R) $, then almost every point of $ \nu(x, \overline\Omega) $ is a Lebesgue point ({see Rudin} \cite[Theorem 7.7]{Rud-74}) and thus:
    \begin{equation*}
        \nu(x_0, \overline\Omega)=\lim_{r\to 0}\frac{1}{2r}\int_{x_0-r}^{x_0+r}\nu(x_0+s, \overline\Omega) ds \leq \frac{1}{2r}(2rC)=C
    \end{equation*}
    for $\mathcal L $-a.e. $ x_0\in\mathbb R $.

    This finishes the proof of Lemma \ref{lem:transition}
\end{proof}


\begin{thebibliography}{10}

\bibitem{Add-Ber-Pen-17}
Louigi {Addario-Berry}, Julien {Berestycki}, and Sarah {Penington}.
\newblock {Branching Brownian motion with decay of mass and the non-local
  Fisher-KPP equation}.
\newblock {\em ArXiv e-prints}, December 2017.

\bibitem{Alf-Ber-Rao-17}
Matthieu Alfaro, Henri Berestycki, and Ga\"el Raoul.
\newblock The {E}ffect of {C}limate {S}hift on a {S}pecies {S}ubmitted to
  {D}ispersion, {E}volution, {G}rowth, and {N}onlocal {C}ompetition.
\newblock {\em SIAM J. Math. Anal.}, 49(1):562--596, 2017.

\bibitem{Alf-Car-17}
Matthieu Alfaro and R\'emi Carles.
\newblock Replicator-mutator equations with quadratic fitness.
\newblock {\em Proc. Amer. Math. Soc.}, 145(12):5315--5327, 2017.

\bibitem{Alf-Cov-Rao-13}
Matthieu Alfaro, J\'er\^ome Coville, and Ga\"el Raoul.
\newblock Travelling waves in a nonlocal reaction-diffusion equation as a model
  for a population structured by a space variable and a phenotypic trait.
\newblock {\em Comm. Partial Differential Equations}, 38(12):2126--2154, 2013.

\bibitem{Alf-Gri-18}
Matthieu Alfaro and Quentin Griette.
\newblock Pulsating fronts for {F}isher--{KPP} systems with mutations as models
  in evolutionary epidemiology.
\newblock {\em Nonlinear Anal. Real World Appl.}, 42:255--289, 2018.

\bibitem{Ber-Ham-02}
Henri Berestycki and Fran\c{c}ois Hamel.
\newblock Front propagation in periodic excitable media.
\newblock {\em Comm. Pure Appl. Math.}, 55(8):949--1032, 2002.

\bibitem{Ber-Jin-Syl-16}
Henri Berestycki, Tianling Jin, and Luis Silvestre.
\newblock Propagation in a non local reaction diffusion equation with spatial
  and genetic trait structure.
\newblock {\em Nonlinearity}, 29(4):1434--1466, 2016.

\bibitem{Ber-Nad-Per-Ryz-09}
Henri Berestycki, Gr\'egoire Nadin, Benoit Perthame, and Lenya Ryzhik.
\newblock The non-local {F}isher-{KPP} equation: travelling waves and steady
  states.
\newblock {\em Nonlinearity}, 22(12):2813--2844, 2009.

\bibitem{Ber-Nic-Sch-1985}
Henri Berestycki, Basil Nicolaenko, and Bruno Scheurer.
\newblock Traveling wave solutions to combustion models and their singular
  limits.
\newblock {\em SIAM J. Math. Anal.}, 16(6):1207--1242, 1985.

\bibitem{Ber-Nir-91}
Henri Berestycki and Louis Nirenberg.
\newblock On the method of moving planes and the sliding method.
\newblock {\em Bol. Soc. Brasil. Mat. (N.S.)}, 22(1):1--37, 1991.

\bibitem{Bog-07}
Vladimir~I. Bogachev.
\newblock {\em Measure theory. {V}ol. {I}, {II}}.
\newblock Springer-Verlag, Berlin, 2007.

\bibitem{Bon-Cov-Leg-17}
Olivier Bonnefon, J\'er\^ome Coville, and Guillaume Legendre.
\newblock Concentration phenomenon in some non-local equation.
\newblock {\em Discrete Contin. Dyn. Syst. Ser. B}, 22(3):763--781, 2017.

\bibitem{Bou-Hen-Ryz-17-2}
E.~{Bouin}, C.~{Henderson}, and L.~{Ryzhik}.
\newblock {The Bramson delay in the non-local Fisher-KPP equation}.
\newblock {\em ArXiv e-prints}, October 2017.

\bibitem{Bou-Cal-14}
Emeric Bouin and Vincent Calvez.
\newblock Travelling waves for the cane toads equation with bounded traits.
\newblock {\em Nonlinearity}, 27(9):2233, 2014.

\bibitem{Bou-Cal-Meu-Mir-Per-Rao-12}
Emeric Bouin, Vincent Calvez, Nicolas Meunier, Sepideh Mirrahimi, Beno\^it
  Perthame, Ga\"el Raoul, and Rapha\"el Voituriez.
\newblock Invasion fronts with variable motility: phenotype selection, spatial
  sorting and wave acceleration.
\newblock {\em C. R. Math. Acad. Sci. Paris}, 350(15-16):761--766, 2012.

\bibitem{Bou-Cha-Hen-Kim-17}
Emeric {Bouin}, Matthew~H. {Chan}, Christopher {Henderson}, and Peter~S. {Kim}.
\newblock {Influence of a mortality trade-off on the spreading rate of cane
  toads fronts}.
\newblock {\em ArXiv e-prints}, February 2017.

\bibitem{Bou-Hen-Ryz-17}
Emeric Bouin, Christopher Henderson, and Lenya Ryzhik.
\newblock The {B}ramson logarithmic delay in the cane toads equations.
\newblock {\em Quart. Appl. Math.}, 75(4):599--634, 2017.

\bibitem{Bre-11}
Haim Brezis.
\newblock {\em Functional analysis, {S}obolev spaces and partial differential
  equations}.
\newblock Universitext. Springer, New York, 2011.

\bibitem{Cov-10}
J\'er\^ome Coville.
\newblock On a simple criterion for the existence of a principal eigenfunction
  of some nonlocal operators.
\newblock {\em J. Differential Equations}, 249(11):2921--2953, 2010.

\bibitem{Cov-Dav-Mar-13}
J\'er\^ome Coville, Juan D\'avila, and Salom\'e Mart\'inez.
\newblock Pulsating fronts for nonlocal dispersion and {KPP} nonlinearity.
\newblock {\em Ann. Inst. H. Poincar\'e Anal. Non Lin\'eaire}, 30(2):179--223,
  2013.

\bibitem{Cov-13}
Jérôme Coville.
\newblock Singular measure as principal eigenfunction of some nonlocal
  operators.
\newblock {\em Applied Mathematics Letters}, 26(8):831 -- 835, 2013.

\bibitem{Fay-Hol-15}
Gr\'egory Faye and Matt Holzer.
\newblock Modulated traveling fronts for a nonlocal {F}isher-{KPP} equation: a
  dynamical systems approach.
\newblock {\em J. Differential Equations}, 258(7):2257--2289, 2015.

\bibitem{Fis-1937}
Ronald~A. Fisher.
\newblock The wave of advance of advantageous genes.
\newblock {\em Annals of Eugenics}, 7(4):355--369, 1937.

\bibitem{Foo-84}
Robert~L. Foote.
\newblock Regularity of the distance function.
\newblock {\em Proc. Amer. Math. Soc.}, 92(1):153--155, 1984.

\bibitem{Gar-82}
Robert~A. Gardner.
\newblock Existence and stability of travelling wave solutions of competition
  models: a degree theoretic approach.
\newblock {\em J. Differential Equations}, 44(3):343--364, 1982.

\bibitem{Gar-Gil-Ham-Roq-12}
Jimmy Garnier, Thomas Giletti, Fran\c{c}ois Hamel, and Lionel Roques.
\newblock Inside dynamics of pulled and pushed fronts.
\newblock {\em J. Math. Pures Appl. (9)}, 98(4):428--449, 2012.

\bibitem{Gen-Vol-Aug-06}
St\'ephane Genieys, Vitaly Volpert, and Pierre Auger.
\newblock Pattern and waves for a model in population dynamics with nonlocal
  consumption of resources.
\newblock {\em Math. Model. Nat. Phenom.}, 1(1):65--82, 2006.

\bibitem{Gil-Ham-Mar-Roq-17}
Marie-Ève Gil, Francois Hamel, Guillaume Martin, and Lionel Roques.
\newblock Mathematical properties of a class of integro-differential models
  from population genetics.
\newblock {\em SIAM J. Appl. Math.}, 77(4):1536--1561, 2017.

\bibitem{Gil-Tru-01}
David Gilbarg and Neil~S. Trudinger.
\newblock {\em Elliptic partial differential equations of second order}.
\newblock Classics in Mathematics. Springer-Verlag, Berlin, 2001.
\newblock Reprint of the 1998 edition.

\bibitem{Gir-18}
L\'eo Girardin.
\newblock Non-cooperative {F}isher-{KPP} systems: traveling waves and long-time
  behavior.
\newblock {\em Nonlinearity}, 31(1):108--164, 2018.

\bibitem{Gou-00}
Stephen~A. Gourley.
\newblock Travelling front solutions of a nonlocal {F}isher equation.
\newblock {\em J. Math. Biol.}, 41(3):272--284, 2000.

\bibitem{Gri-Rao-16}
Quentin Griette and Ga\"el Raoul.
\newblock Existence and qualitative properties of travelling waves for an
  epidemiological model with mutations.
\newblock {\em J. Differential Equations}, 260(10):7115--7151, 2016.

\bibitem{Gri-Rao-Gan-15}
Quentin Griette, Gaël Raoul, and Sylvain Gandon.
\newblock Virulence evolution at the front line of spreading epidemics.
\newblock {\em Evolution}, 69(11):2810--2819, 2015.

\bibitem{Ham-Ryz-14}
Fran\c{c}ois Hamel and Lenya Ryzhik.
\newblock On the nonlocal {F}isher-{KPP} equation: steady states, spreading
  speed and global bounds.
\newblock {\em Nonlinearity}, 27(11):2735--2753, 2014.

\bibitem{Has-Kop-Nab-Tro-16}
Karel Hasik, Jana Kopfov\'a, Petra N\'ab\v~elkov\'a, and Sergei Trofimchuk.
\newblock Traveling waves in the nonlocal {KPP}-{F}isher equation: different
  roles of the right and the left interactions.
\newblock {\em J. Differential Equations}, 260(7):6130--6175, 2016.

\bibitem{Hol-09}
Edward Holmes.
\newblock {\em The evolution and emergence of RNA viruses}.
\newblock Oxford University Press, 2009.

\bibitem{Kol-Pet-Pis-1937}
Andre\u{\i}~N. Kolmogorov, Ivan~G. Petrovsky, and N.~S. Piskunov.
\newblock Étude de l'\'equation de la diffusion avec croissance de la
  quantit\'e de mati\`{e}re et son application \`{a} un probl\`{e}me
  biologique.
\newblock {\em Bull. Univ. Etat Moscou}, S\'er. Inter. A 1:1--26, 1937.

\bibitem{Guo-11}
Guo Lin.
\newblock Spreading speeds of a {L}otka-{V}olterra predator-prey system: the
  role of the predator.
\newblock {\em Nonlinear Anal.}, 74(7):2448--2461, 2011.

\bibitem{Lui-89}
Roger Lui.
\newblock Biological growth and spread modeled by systems of recursions. {I}.\
  {M}athematical theory.
\newblock {\em Math. Biosci.}, 93(2):269--295, 1989.

\bibitem{Ogi-Mat-99-DCDS}
Toshiko Ogiwara and Hiroshi Matano.
\newblock Monotonicity and convergence results in order-preserving systems in
  the presence of symmetry.
\newblock {\em Discrete Contin. Dynam. Systems}, 5(1):1--34, 1999.

\bibitem{Pen-17}
Sarah {Penington}.
\newblock {The spreading speed of solutions of the non-local Fisher-KPP
  equation}.
\newblock {\em ArXiv e-prints}, August 2017.

\bibitem{Per-Phi-Bas-Has-13}
T.~Alex Perkins, Benjamin~L. Phillips, Marissa~L. Baskett, and Alan Hastings.
\newblock Evolution of dispersal and life history interact to drive
  accelerating spread of an invasive species.
\newblock {\em Ecology Letters}, 16(8):1079--1087, 2013.

\bibitem{Phi-Pus-13}
Ben~L. Phillips and Robert Puschendorf.
\newblock Do pathogens become more virulent as they spread? {E}vidence from the
  amphibian declines in {Central America}.
\newblock {\em Proceedings of the Royal Society of London B: Biological
  Sciences}, 280(1766), 2013.

\bibitem{Rud-74}
Walter Rudin.
\newblock {\em Real and complex analysis}.
\newblock McGraw-Hill Book Co., New York-D\"usseldorf-Johannesburg, second
  edition, 1974.
\newblock McGraw-Hill Series in Higher Mathematics.

\bibitem{Sto-1976}
A.~N. Stokes.
\newblock On two types of moving front in quasilinear diffusion.
\newblock {\em Math. Biosci.}, 31(3-4):307--315, 1976.

\bibitem{Wax-Pec-98}
David Waxman and Joel~R. Peck.
\newblock {Pleiotropy and the preservation of perfection}.
\newblock {\em Science}, {279}({5354}):{1210--1213}, {1998}.

\bibitem{Wax-Pec-06}
David Waxman and Joel~R. Peck.
\newblock {The frequency of the perfect genotype in a population subject to
  pleiotropic mutation}.
\newblock {\em Theoretical Population Biology}, {69}({4}):{409--418}, {2006}.

\bibitem{Wei-82}
Hans~F. Weinberger.
\newblock Long-time behavior of a class of biological models.
\newblock {\em SIAM J. Math. Anal.}, 13(3):353--396, 1982.

\bibitem{Wei-02}
Hans~F. Weinberger.
\newblock On spreading speeds and traveling waves for growth and migration
  models in a periodic habitat.
\newblock {\em J. Math. Biol.}, 45(6):511--548, 2002.

\bibitem{Wei-Lew-Li-02}
Hans~F. Weinberger, Mark~A. Lewis, and Bingtuan Li.
\newblock Analysis of linear determinacy for spread in cooperative models.
\newblock {\em J. Math. Biol.}, 45(3):183--218, 2002.

\bibitem{Xin-00}
Jack Xin.
\newblock Front propagation in heterogeneous media.
\newblock {\em SIAM Rev.}, 42(2):161--230, 2000.

\bibitem{Zei-86}
Eberhard Zeidler.
\newblock {\em Nonlinear functional analysis and its applications. {I}}.
\newblock Springer-Verlag, New York, 1986.
\newblock Fixed-point theorems, Translated from the German by Peter R. Wadsack.

\end{thebibliography}
\end{document}